\documentclass[12pt,letterpaper]{amsart}
\usepackage{amsthm,amsfonts,amssymb,amsmath}
\usepackage{amsmath, bm}
\usepackage{tabularx}
\usepackage{graphicx} 
\usepackage{tikz}
\usetikzlibrary{patterns}
\usetikzlibrary{arrows.meta}
\usetikzlibrary{arrows}
\usepackage{amscd}
\usepackage{parskip}
\usepackage{enumitem}
\usepackage{comment}  
\newcommand\T{\rule{0pt}{2.7ex}}       

\newcommand{\tclb}{\textcolor{blue}}
\newcommand{\tclr}{\textcolor{red}}
\newcommand{\svwtwo}{\textcolor{black}}
\newcommand{\svw}{\textcolor{black}}
\newcommand{\sheila}{\textcolor{black}}
\newcommand{\sheilaFeb}{\textcolor{black}}  

\newcommand{\bQ}{\mathbb{Q}}

\newtheorem{theorem}{Theorem}[section]
\newtheorem{lemma}[theorem]{Lemma}

\newtheorem{proposition}[theorem]{Proposition}

\theoremstyle{definition}
\newtheorem{definition}[theorem]{Definition}
\newtheorem{example}[theorem]{Example}

\newtheorem{remark}[theorem]{Remark}

\newcommand{\poRI}{\preccurlyeq_{\mathcal{R}{\mathfrak{S}}^\ast _\alpha}}  
\newcommand{\poA}{\preccurlyeq_{\mathcal{A}^\ast _\alpha}}  
\newcommand{\poAbar}{\preccurlyeq_{\mathcal{\bar{A}}^\ast _\alpha} } 
\newcommand{\poRIcover}{\prec_{\mathcal{R}{\mathfrak{S}}^\ast _\alpha}}
 
\newcommand{\poIcover}{\prec_{{\mathfrak{S}}^\ast _\alpha}} 

\newlength{\cellsize}
\newcommand\tableau[1]{
\vcenter{
\let\\=\cr
\baselineskip=-16000pt
\lineskiplimit=16000pt
\lineskip=0pt
\halign{&\tableaucell{##}\cr#1\crcr}}}

\newcommand{\tableaucell}[1]{{%
\def \arg{#1}\def \void{}%
\ifx \void \arg
\vbox to \cellsize{\vfil \hrule width \cellsize height 0pt}%
\else
\unitlength=\cellsize
\begin{picture}(1,1)
\put(0,0){\makebox(1,1)[c]{$#1$}}
\put(0,0){\line(1,0){1}}
\put(0,1){\line(1,0){1}}
\put(0,0){\line(0,1){1}}
\put(1,0){\line(0,1){1}}
\end{picture}%
\fi}}

\DeclareMathOperator{\comp}{comp}

\DeclareMathOperator{\set}{set}
\newcommand{\dI}{\mathfrak{S}^*}
\newcommand{\rdI}{\mathcal{R}\mathfrak{S}^*}

\DeclareMathOperator{\QSym}{QSym}
\DeclareMathOperator{\Des}{Des}
\DeclareMathOperator{\DesI}{Des_{\dI}}

\DeclareMathOperator{\inv}{inv}

\DeclareMathOperator{\chr}{\mathrm{ch}}  
\DeclareMathOperator{\bA}{\bar{\mathcal{A}}} 
\DeclareMathOperator{\bpi}{\bar{\pi}} 

\newcommand{\Qsym}{\ensuremath{\operatorname{QSym}}}
\newcommand{\DI}{{\mathfrak{S}}^\ast}	         
\newcommand{\RI}{\mathcal{R}{\mathfrak{S}}^\ast} 
\newcommand{\SIT}{\ensuremath{\operatorname{SIT}}}
\newcommand{\SET}{\ensuremath{\operatorname{SET}}} 
\newcommand{\NSET}{\ensuremath{\operatorname{NSET}}}
\newcommand{\SRCT}{\ensuremath{\operatorname{SRCT}}}

\newcommand{\Nsym}{\ensuremath{\operatorname{NSym}}}
\newcommand{\nce}{\mathbf{e}}         	

\newcommand{\hn}{H_n(0)}

\newcommand{\des}{\mathrm{Des}} 
\newcommand{\desri}{\mathrm{Des}_{\mathcal{R}\mathfrak{S}^\ast}} 
\newcommand{\ninv}{\mathrm{inv}} 
\newcommand{\rtau}{{T}} 

\newcommand{\suchthat}{\;|\;}

\newcommand{\spam}{\operatorname{span}}

\DeclareFontFamily{U}{rcjhbltx}{}  
\DeclareFontShape{U}{rcjhbltx}{m}{n}{<->rcjhbltx}{}
\DeclareSymbolFont{hebrewletters}{U}{rcjhbltx}{m}{n}
\DeclareMathSymbol{\shin}{\mathord}{hebrewletters}{152}

\author[Niese, Sundaram, van Willigenburg, Vega, Wang]{Elizabeth Niese, Sheila Sundaram, Stephanie van Willigenburg,\\ Julianne Vega, Shiyun Wang}

\address{Elizabeth Niese: Marshall University, Huntington, WV 25755, USA} 
\email{niese@marshall.edu}
\address{Sheila Sundaram: Pierrepont School, Westport, CT 06880, USA}
\email{shsund@comcast.net}
\address{Stephanie van Willigenburg: University of British Columbia, \svw{Vancouver,} BC V6T 1Z2, Canada}
\email{steph@math.ubc.ca}
\address{Julianne Vega: Kennesaw State University, \svw{Kennesaw, GA 30144,} USA}
\email{jvega30@kennesaw.edu}
\address{Shiyun Wang: University of Southern California, \svw{Los Angeles, CA 90089-2532,} USA}
\email{shiyunwa@usc.edu}

\title[Row-strict 0-Hecke algebra modules]{0-Hecke modules for row-strict dual immaculate functions}
\date{\today}

\begin{document}


\begin{abstract}{We introduce a new basis of quasisymmetric functions, the row-strict dual immaculate functions.  We construct a cyclic, indecomposable 0-Hecke algebra module for these functions.   Our row-strict immaculate functions are related to the dual immaculate functions of Berg-Bergeron-Saliola-Serrano-Zabrocki (2014-15) by the involution $\psi$ on the ring $\Qsym$ of quasisymmetric functions.  We give an explicit description of the effect of $\psi$ on the associated 0-Hecke modules, via the poset induced by the 0-Hecke action on standard immaculate tableaux.  This remarkable poset reveals other 0-Hecke submodules and quotient modules, often cyclic and indecomposable, notably for a row-strict analogue of the extended Schur functions studied in Assaf-Searles (2019). 

Like the dual immaculate function, the  row-strict dual immaculate function is the generating function of a suitable set of tableaux, corresponding to a specific  descent set. We give a complete combinatorial and representation-theoretic picture by constructing 0-Hecke modules for the remaining variations on descent sets,  and showing that \emph{all} the possible variations for generating functions of tableaux occur as characteristics of the 0-Hecke modules determined by these descent sets. }
\end{abstract}
\maketitle
\tableofcontents
\section{Introduction}\label{sec:Intro}
\svw{A recent flourishing area of research is that of Schur-like functions, whose properties are analogous to the ubiquitous Schur functions that arise in many areas, from enumerative combinatorics where they are generating functions for Young tableaux, to representation theory where they are the irreducible representations for the general linear groups, \sheila{as well as being intimately connected to representations of the symmetric group}.}

\svw{The area of Schur-like functions began with quasisymmetric Schur functions \cite{HLMvW2011},  followed by discoveries of row-strict quasisymmetric Schur functions \cite{MR2014}, Young quasisymmetric Schur functions \cite{LMvW2013, MN2015}, noncommutative Schur functions \cite{BLvW2011} and immaculate functions \cite{BBSSZ2014},  quasisymmetric Schur $Q$-functions \cite{JL2015},   quasisymmetric Macdonald polynomials \cite{CHMMW2022}, and Schur functions in noncommuting variables \cite{ALvW2021}.}

The results of this paper were  first announced in \cite{NSvWVW2022FPSAC}. 
In this paper and its companion \cite{NSvWVW2023}, we introduce a new family of quasisymmetric functions, the \emph{row-strict dual immaculate} functions.  Our focus in the present work is the study of the associated 0-Hecke algebra modules, which we define and analyse \svw{in order to develop the analogy with Schur functions in the representation theory context}. This programme was first carried out for the dual immaculate functions \cite{BBSSZ2015}, the quasisymmetric Schur functions \cite{TvW2015}, and subsequently for the row-strict Young quasisymmetric functions \cite{BS2021} and the extended Schur functions \cite{S2020}.

\svw{In further analogy with Schur functions, the row-strict dual immaculate functions $\rdI_\alpha$ are  defined in  \cite{NSvWVW2023} as generating functions for certain types of tableaux of composition shape $\alpha$; by identifying the correct descent set, it follows that the functions 
$\rdI_\alpha$ expand positively in the basis of fundamental quasisymmetric functions, and are the image,} under the involutive algebra automorphism $\psi$ of the Hopf algebra $\Qsym$ of quasisymmetric functions, of the well-studied dual immaculate functions of \cite{BBSSZ2015}.

The descent set  determines a 0-Hecke algebra action on the set $\SIT(\alpha)$ of standard immaculate tableaux of composition shape $\alpha$, yielding a cyclic indecomposable module whose quasisymmetric characteristic is $\rdI_\alpha$.  The action defines a partial order on the set $\SIT(\alpha)$ which turns out to be dual to the partial order in \cite{BBSSZ2015}, see Lemma~\ref{lem:sameposet}.  The resulting graded poset $P\rdI(\alpha)$, which we call the immaculate Hecke poset,  has remarkable properties, leading to the discovery of several other 0-Hecke modules with interesting quasisymmetric characteristics.  Among these is an analogue of the extended Schur functions defined by Assaf and Searles \cite{AS2019}. An examination of the poset (see Figure~\ref{fig:Poset}) reveals various subposets that are closed under either our action or the dual immaculate action of \cite{BBSSZ2015}.  We  investigate the associated submodules.

One of the key contributions of this paper, then, is the 
careful scrutiny of the partial order resulting originally from the 0-Hecke algebra action defined in \cite{BBSSZ2015}, and the discovery that hidden within it are special standard immaculate tableaux, which generate interesting 0-Hecke algebra submodules and quotient modules, for both the dual immaculate action and our new row-strict dual immaculate action.
The duality in the poset reflects the action of the involution $\psi$ on $\Qsym$: we prove first that the poset is graded, and has unique top and bottom elements, see Definition~\ref{def:bot-elt} and Definition~\ref{def:top-elt}. The cyclic generator for the dual immaculate module was shown to be the top element in \cite{BBSSZ2015};  for the row-strict dual immaculate module, we show that  the cyclic generator is the  bottom element of the poset.  Similarly, the cyclic generators for the extended Schur Hecke-module of \cite{S2020} and our row-strict extended Schur Hecke-module (see Theorem~\ref{thm:Z-indecomp}) are the top and bottom elements of the interval $[S^{col}_\alpha, S^{row}_\alpha]$, the column superstandard and row superstandard tableaux of   \svw{Definition~\ref{def:column-bot-elt} and Definition~\ref{def:top-elt}}, respectively.
Finally we show how another special standard tableau $S^{row*}_\alpha$, see Definition~\ref{def:top-elt-SIT*}, also generates a cyclic submodule of the dual immaculate 0-Hecke algebra module.
The poset duality phenomenon can also be used to explain, for example, the passage between the modules in \cite{TvW2015} and \cite{BS2021}, see  Section~\ref{sec:dualImmsubmodules}. 

Our proofs are technical,  relying heavily on two  straightening algorithms which produce saturated chains in the poset $P\rdI(\alpha).$   We prove  indecomposability by following the pioneering work in \cite{TvW2015} and  \cite{BBSSZ2015},  with considerable technical modifications.

The paper is organised as follows.  The basic definitions regarding quasisymmetric functions appear in Section~\ref{sec:Background}, which also includes the facts that we need about dual immaculate functions. Our new family of row-strict dual immaculate functions is introduced in Section~\ref{sec:RSImmFns}. In  Section~\ref{sec:Hecke-action-RSImm}, we briefly review the necessary facts about 0-Hecke algebras, and then define a new  0-Hecke algebra action on standard immaculate tableaux. Section~\ref{sec:partial-order} describes a partial order on these tableaux, which, by standard arguments, leads to a filtration showing that our 0-Hecke module has quasisymmetric characteristic equal to the row-strict dual immaculate function.

Section~\ref{sec:indecomp-module-and-poset} is devoted to showing that our new 0-Hecke module is cyclic (Theorem~\ref{thm:cyclic-rdI}) and indecomposable, culminating in Theorem~\ref{thm:Indecomp}.  The technicalities here hinge on two key straightening algorithms, described in Propositions~\ref{prop:bot-elt} and~\ref{prop:top-elt}, which we use to establish that the poset of Section~\ref{sec:partial-order} is bounded, with a unique minimal and maximal element.  The minimal element, defined in Definition~\ref{def:bot-elt}, is shown to be the cyclic generator of our module. The maximal element is the cyclic generator of the module of \cite{BBSSZ2015}.\footnote{It is not explicitly established in \cite{BBSSZ2015} that the poset of standard immaculate tableaux defined by the 0-Hecke action has a maximal element.} The straightening algorithms play an important role in the technical lemmas leading to the indecomposability proof.  The work of this section also reveals a remarkable connection, see  Lemma~\ref{lem:sameposet}, between our partial order and that of \cite{BBSSZ2015}, showing how the map $\psi$ between the quasisymmetric characteristics of the dual immaculate and row-strict dual immaculate modules, is reflected in the duality of the posets.

The motivation for Section~\ref{sec:row-strict-ext} and Section~\ref{sec:dualImmsubmodules} comes from  a closer examination of the poset of standard immaculate tableaux defined by the 0-Hecke action, and the remarkable properties to which we have alluded above.  In Section~\ref{sec:row-strict-ext} we show how the poset reveals a 0-Hecke module whose characteristic is the row-strict analogue of the extended Schur function defined in \cite{AS2019}, for which a 0-Hecke module was constructed in \cite{S2020}.   In fact we show that there are not one but two related modules, a submodule  (Theorem~\ref{thm:Z-indecomp}) and a quotient module (Theorem~\ref{thm:another-bigone}) of the row-strict dual immaculate action, both cyclic and indecomposable.  To complete this analogy,  in Section~\ref{sec:dualImmsubmodules} we show once again how the poset of standard immaculate tableaux leads to the discovery of  more cyclic modules, a submodule and two quotient modules of the (original) dual immaculate action, as well as a quotient module of the extended row-strict dual immaculate module (Theorem~\ref{thm:dualImm-submodule-quotient}, Theorem~\ref{thm:SET-intersect-SIT*}).     Figure~\ref{fig:Poset} indicates the essential representation-theoretic properties of the immaculate Hecke  poset.

\sheila{
We conclude in Section 9 by considering the two remaining choices for the descent set of a standard immaculate tableau, apart from the two which determine the dual immaculate and row-strict dual immaculate functions.  Although the corresponding quasisymmetric functions are no longer independent, interestingly, both of the new variants come with associated 0-Hecke actions.  Furthermore, these actions determine the same immaculate Hecke poset, and consequently our straightening algorithms apply, giving cyclic 0-Hecke modules as in Section~\ref{sec:indecomp-module-and-poset}, generated respectively by $S^0_\alpha$ and $S^{row}_\alpha$, as well as  submodules generated respectively by $S^{col}_\alpha$ and $S^{row*}_\alpha$ as in  Section~\ref{sec:row-strict-ext} and  Section~\ref{sec:dualImmsubmodules}.} \sheilaFeb{Our final result, Theorem~\ref{thm:8flavours-tableaux}, shows that these modules complete the combinatorial picture of  tableaux  considered in this paper, by accounting for all the possible variations on increasing rows and columns.  This information is captured in Figure~\ref{fig:4-descent-sets}. }

\begin{figure}[htb] \centering
\scalebox{0.7}{	
\begin{tikzpicture}
\newcommand*{\xdist}{*3}
			\newcommand*{\ydist}{*2.2}
			\node (P0) at (0.00\xdist,0\ydist) {\tclb{$\boldsymbol{\mathcal{T}_\alpha({\text{1st col}<, \text{rows}<):\bA_\alpha}}$}};
			\node (Q0) at (1.2\xdist, 0\ydist) {\ (Section~\ref{sec:new-descent-set})};
\node (P11) at (-2\xdist, 1\ydist) {\tclb{$\boldsymbol{\mathcal{T}_\alpha({\text{1st col}<, \text{rows}\le}):\dI_\alpha}$} \cite{BBSSZ2015}};
\node (P10) at (-2\xdist, 2\ydist) {\tclr{$\boldsymbol{\mathcal{T}_\alpha({\text{cols}<, \text{rows}\le}):\mathcal{E}_\alpha}$} \cite{S2020}, \cite{CFLSX2014}};
\node (P12) at (2.2\xdist, 2\ydist) {\tclb{$\boldsymbol{\mathcal{T}_\alpha({\text{1st col}\le, \text{rows}<}):\rdI_\alpha}$} (Section~\ref{sec:RSImmFns})};
\node (P2) at (0\xdist, 3\ydist) { \tclb{$\boldsymbol{\mathcal{T}_\alpha({\text{1st col}\le, \text{rows}\le}):\mathcal{A}_\alpha}$}};
\node (Q2) at (-1.2\xdist, 3\ydist) {(Section~\ref{sec:new-descent-set})\ \  };
\node (P20) at (2.2\xdist, 1\ydist) {\tclr{$\boldsymbol{\mathcal{T}_\alpha({\text{cols}\le, \text{rows}<}):\mathcal{R}\mathcal{E}_\alpha}$} (Section~\ref{sec:row-strict-ext})};
\node (P3) at (0.00\xdist, 1.5\ydist) {$\boldsymbol{\psi}$};
\node (P01) at (-2.5\xdist , 0\ydist) {\tclr{$\boldsymbol{\mathcal{T}_\alpha({\text{cols}<, \text{rows}<}):
\bar{\mathcal{A}}_{\SET(\alpha)} }$ }};
\node (P22) at (2.5\xdist , 3\ydist) {\tclr{$\boldsymbol{\mathcal{T}_\alpha({\text{cols}\le, \text{rows}\le}):\mathcal{A}_{\SET(\alpha)}}$}};
\draw[ultra thick, magenta, left hook-latex] (P20) -- (P12);
\draw[ultra thick, -{Stealth}{Stealth}] (P11) -- (P10); 
\draw[ultra thick, dashed, cyan, {Stealth}-] (P11) -- (P3);
\draw[ultra thick, dashed, cyan, -{Stealth}] (P3) -- (P12);
\draw[ultra thick, dashed, cyan, {Stealth}-] (P10) -- (P3);
\draw[ultra thick, dashed, cyan, -{Stealth}] (P3) -- (P20);
\draw[ultra thick, dashed, cyan, {Stealth}-] (P0) -- (P3);
\draw[ultra thick, dashed, cyan, -{Stealth}] (P3) -- (P2);
\draw[ultra thick, dashed, cyan, {Stealth}-] (P0) -- (P3);
\draw[ultra thick, dashed, cyan, -{Stealth}]  (P3) -- (P01);
\draw[ultra thick, dashed, cyan, -{Stealth}]  (P3) -- (P22);
\draw[ultra thick, -{Stealth}{Stealth}] (P0) -- (P01);
\draw[ultra thick, magenta, left hook-latex] (P22) -- (P2);
\end{tikzpicture}
}
\caption{The eight flavours of tableaux, their characteristics and  $\hn$-modules, related in pairs by the involution $\psi$, from the four descent sets. The double arrow-head indicates a quotient module, and the hooked arrow indicates a submodule.}\label{fig:4-descent-sets}
\end{figure}
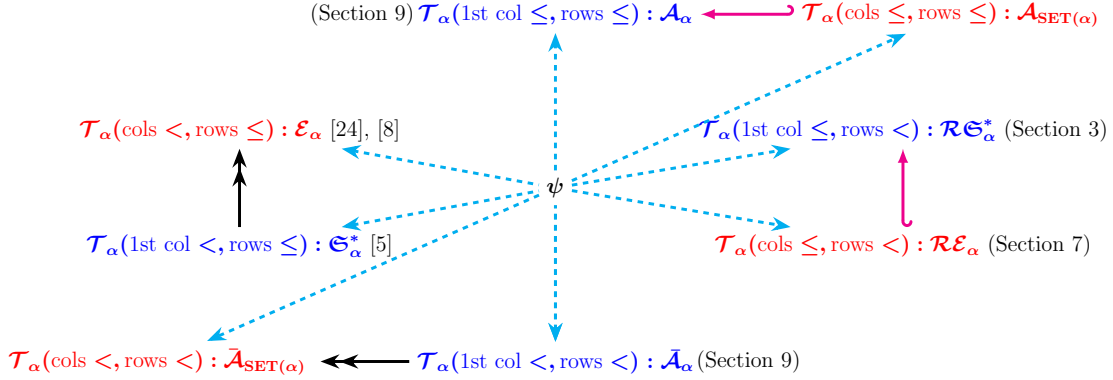

Table~\ref{table:standard-tableaux-acronyms} below provides a list of tableau  acronyms used in this paper.
\begin{table}[htbp]
\caption{Standard tableaux of composition shape $\alpha$ of $n$, with distinct entries $\{1,2,\ldots,n\}$, where \textbf{all rows increase (strictly)  left to right}:}
\begin{center}
\scalebox{1}{
\begin{tabular}{|c|c|}
\hline
$\NSET(\alpha)$ & at least one column does NOT increase bottom to top\\[2pt]
$\SIT(\alpha)$ & first column increases bottom to top\\[2pt]
$\SIT^*(\alpha)$ & first column filled bottom to top consecutively with $1,2,\ldots$\\[2pt]
$\SET(\alpha)$ & all columns increase bottom to top\\ \hline
\end{tabular}
}
\end{center}
\label{table:standard-tableaux-acronyms}
\end{table}

In Section~\ref{sec:new-descent-set}, Table~\ref{table:RSImm}  provides a summary of our results, as well as comparison with prior work.  Table~\ref{table:QSymfunctions-modules} compiles the known quasisymmetric bases to date, together with their 0-Hecke modules, as well as the new quasisymmetric functions, all of which expand positively in the fundamental basis, arising from the 0-Hecke submodules and quotient modules   discovered in this paper.

\noindent \svw{\textbf{Acknowledgements.} The authors would like to thank the Algebraic Combinatorics Research Community program at ICERM through which this research took place. The third author was supported in part by the National Sciences Research Council of Canada.  The authors also express their gratitude to the  FPSAC 2022 referees and the reviewer of this journal for their careful reading of the paper and detailed comments.}

\section{Background}\label{sec:Background}
We refer the reader to \cite{LMvW2013} for basic definitions.

A composition of $n$ is a positive sequence of integers $\alpha = (\alpha_1, \alpha_2, \ldots,\alpha_k)$ summing to $n$, which we depict as    a collection of left-justified boxes with $\alpha_i$ boxes in row $i$, where row $1$ is the bottom row, in the ``French'' convention.  We call this the diagram of $\alpha$.

It is well known that compositions of $n$ are in bijection with subsets of $\{1,2,\ldots,n-1\}$.  Write $\alpha\vDash n$ for a composition $\alpha = (\alpha_1, \alpha_2,\ldots,\alpha_k)$ of $n$; the corresponding set is $\set(\alpha) = \{\alpha_1,\alpha_1+\alpha_2,\ldots,\alpha_1+\cdots+\alpha_{k-1}\}$.
Given a subset $S=\{s_1,s_2,\ldots,s_j\}$ of $\{1,\ldots,n-1\}$, the corresponding composition of $n$ is $\comp(S)=(s_1,s_2-s_1,\ldots,s_j-s_{j-1},n-s_j)$. 
A function $f\in \bQ[[x_1,x_2,\ldots]]$ is {\em quasisymmetric} if the coefficient of $x_1^{\alpha_1}x_2^{\alpha_2}\cdots x_k^{\alpha_k}$ is the same as the coefficient of $x_{i_1}^{\alpha_1}x_{i_2}^{\alpha_2}\cdots x_{i_k}^{\alpha_k}$ for every $(\alpha_1, \alpha_2,\ldots,\alpha_k)$ and $i_1<i_2<\cdots <i_k$.  The set of all quasisymmetric functions forms a ring graded by degree, $\QSym = \bigoplus_{n} \QSym_n$, where each $\QSym_n$ is a vector space over $\bQ$ with bases indexed by compositions of $n$. 

Given a composition $\alpha=(\alpha_1,\alpha_2,\ldots,\alpha_k)$ of $n$, the {\em fundamental quasisymmetric function} indexed by $\alpha$ is 
\[F_\alpha(x_1,x_2,\ldots) = \sum_{\substack{i_1\leq i_2\leq \cdots \leq i_n\\i_j<i_{j+1} \text{ if } j\in\set(\alpha)}} x_{i_1}x_{i_2}\cdots x_{i_n}.\]
Note that 
$\{F_\alpha:\alpha\vDash n\}$ is a basis for $\QSym_n$, the \emph{fundamental} basis.

The {\em complement} of a composition $\alpha$, denoted $\alpha^c$, is the composition obtained from $\alpha$ by taking the complement of the set corresponding to  $\alpha$.  That is, 
$\alpha^c= \comp(\set(\alpha)^c)$.
 In $\QSym$ we have the involutive automorphism   $\psi,$ defined on the fundamental basis by
\begin{equation}
\psi(F_\alpha) = F_{\alpha^c} \label{eqn:psiF}.
\end{equation}

With notation as in \cite{LMvW2013}, the algebra $\Nsym=\bQ\langle\nce_1,\nce_2,\ldots\rangle$ of {\em noncommutative symmetric functions},  a Hopf algebra dual to $\Qsym$ \cite{GKLLRT1995},   
is generated by noncommuting indeterminates $\nce_n$ of degree $n$. 
We briefly review concepts we will need from the work of Berg-Bergeron-Saliola-Serrano-Zabrocki, who introduced the 
\emph{immaculate functions} $\mathfrak{S}_\alpha$ as a basis of $\Nsym$ formed by iterated creation operators~\cite{BBSSZ2014}.  
Their dual in $\QSym$ are the {\em dual immaculate functions}, $\dI_\alpha$.  These functions can be defined  as the generating function for {\em immaculate tableaux}.  
\begin{definition}\cite[Definition 2.1]{BBSSZ2015}
Given $\alpha\vDash n$, an immaculate tableau of shape $\alpha$ is a filling, $D$, of the cells of the diagram of $\alpha$ such that 
\begin{enumerate}[itemsep=1pt]
\item The leftmost column entries strictly increase from bottom to top.
\item The row entries weakly increase from left to right. 
\end{enumerate}
An immaculate tableau of shape $\alpha\vDash n$ is  a \emph{standard} immaculate tableau if it is filled with distinct entries taken from $\{1,2,\ldots,n\}.$ 
\end{definition}
Given an immaculate tableau $D$, define $x^D:=x_1^{d_1}x_2^{d_2}\cdots x_{k}^{d_k}$, where  $d_i$ is the number of $i$'s in the tableau $D$.  Thus if $D$ is standard and $\alpha\vDash n$, \svw{$x^D=x_1x_2\cdots x_n$.}
\begin{definition} \label{def:dIfunction}
The {\em dual immaculate function} indexed by $\alpha\vDash n$ is $\dI_\alpha = \sum_D x^D$, summed over all immaculate tableaux of shape $\alpha$.
\end{definition}
 
\begin{theorem}\cite[Definition~2.3, Proposition~3.1]{BBSSZ2015}  
The set $\{\dI_\alpha\}_{\alpha\vDash n}$ is a basis for $\Qsym_n.$
Given a standard immaculate tableau $S$, its $\dI$-{\em descent set}  $\DesI(S)$ of $S$ is 
\[\DesI(S):=\{i: i+1 \text{ appears strictly above }i \text{ in } S \}.\]  
Then 
$\dI_\alpha = \sum_S F_{\comp(\DesI(S))},$ summed  over all standard immaculate tableaux of shape $\alpha$.  
\end{theorem}

For $\alpha=(1,2),$ the unique standard immaculate tableau $\tableau{2&3\\1}$ 
 has $\dI$-descent set $\{1\}$; \quad thus $\dI_{(1,2)}=F_{\comp \{1\}}=F_{(1,2)}.$


\section{Row-strict dual immaculate functions}\label{sec:RSImmFns}
In this section we define a new quasisymmetric function, which we call the {\em row-strict dual immaculate function.}  These were introduced in the companion paper  \cite{NSvWVW2023}, where it is shown that they form a basis of $\Qsym$. 

\begin{definition}\label{def:rdIfunction} Let $\alpha\vDash n$.  A row-strict immaculate tableau of shape $\alpha$ is a filling $U$ such that 
\begin{enumerate}[itemsep=1pt]
\item The leftmost column entries weakly increase from bottom to top.
\item The row entries strictly increase from left to right.
\end{enumerate}
A  row-strict immaculate tableau  of shape $\alpha\vDash n$ is \emph{standard} if it  is filled with distinct entries taken from $\{1,2,\ldots,n\}.$  Note that standard row-strict immaculate tableaux coincide with standard immaculate tableaux.

The {\em row-strict dual immaculate function} indexed by $\alpha$ is $ \rdI_\alpha = \sum_U x^U$ where the sum is over all row-strict immaculate tableaux of shape $\alpha$, and $x^U:=x_1^{d_1}x_2^{d_2}\cdots x_{k}^{d_k}$, where  $d_i$ is the number of $i$'s in the tableau $U$.  
\end{definition}

\begin{theorem}\label{thm:rsfunddecomp}\cite{NSvWVW2023}
Let $\alpha \vDash n$.  Define the $\rdI$-descent set of a standard row-strict immaculate tableau $S$ as
\[\Des_{\rdI}(S):=\{i:i+1\text{ is weakly below } i \text{ in } S\}.\]Then 
\[\rdI_\alpha = \sum_{S} F_{\comp(\Des_{\rdI}(S))}\]
where the sum is over all standard row-strict immaculate tableaux of shape $\alpha$.
\end{theorem}
For $\alpha=(1,2),$ the unique standard immaculate tableau $\tableau{2&3\\1}$ 
has $\rdI$-descent set $\{2\}$; thus $\rdI_{(1,2)}=F_{\comp \{2\}}=F_{(2,1)}.$

Clearly for any standard immaculate tableau $S$, $\des_{\DI}(S) = \des_{\rdI}(S)^c$, and hence applying the involution $\psi$  immediately gives $\psi(\dI_\alpha) =\rdI_\alpha $. Consequently, we have: 

\begin{theorem}\label{the:RIasDI}\cite{NSvWVW2023}
 $\{ \RI _\alpha \suchthat \alpha \vDash n\}$ is a basis for $\Qsym _n$.
\end{theorem} 

\section{A new 0-Hecke algebra action} 
\label{sec:Hecke-action-RSImm}

Recall that the symmetric group $S_n$ can be defined via generators $s_i, 1\le i\le n-1,$ the adjacent transpositions, subject to the relations 
\begin{align*} {s_i}^2 &=1; \\
               s_i s_{i+1}s_i &=s_{i+1}s_is_{i+1};\\
                  s_is_j &=s_js_i, \ |i-j|\ge 2.
\end{align*}

\begin{definition}\cite{PamelaBromwichNorton1979, Mathas1999} Let $\mathbf{K}$ be any field.  The 0-Hecke algebra $H_n(0)$ is the $\mathbf{K}$-algebra with 
generators $\pi_i, 1\le i\le n-1$ and relations 
\begin{align*} {\pi_i}^2 &=\pi_i; \\
               \pi_i\pi_{i+1}\pi_i &=\pi_{i+1}\pi_i\pi_{i+1};\\
                  \pi_i\pi_j &=\pi_j\pi_i, \ |i-j|\ge 2.
\end{align*}
The algebra $H_n(0)$ has dimension $n!$ over  $\mathbf{K}$, with basis elements $\{\pi_\sigma: \sigma\in S_n\},$ where $\pi_\sigma=\pi_{i_1}\ldots \pi_{i_m}$ if $\sigma=s_{i_1}\ldots s_{i_m}$ is a reduced word. This is well-defined by standard Coxeter group theory, 
see \cite{BjBrenti2005}.
\end{definition}

From \cite{PamelaBromwichNorton1979}, the 0-Hecke algebra admits  $2^{n-1}$ irreducible one-dimensional modules $L_\alpha=\mathrm{Span}(v_\alpha)$, one for each composition $\alpha\vDash n,$ carrying an action defined by 
\begin{equation}\label{eqn:HeckeIrreps} \pi_i(v_\alpha)=\begin{cases} 0, \ i\in \set(\alpha),\\
  v_\alpha, \text{ otherwise.}\end{cases}
\end{equation}
Here $\set(\alpha)=\{\alpha_1, \alpha_1+\alpha_2,\ldots, \alpha_1+\ldots+\alpha_{k-1}\}$ is the subset of $[n-1]$ associated to the composition $\alpha=(\alpha_1, \ldots , \alpha_k)$ of $n$ of length $k.$  

\begin{definition}\cite{DKLT1996, KrobThibon1997}
The \textit{quasisymmetric characteristic} $\mathrm{ch}$ is an isomorphism from the Grothendieck ring of $H_n(0)$ with respect to the induction product, to the ring of quasisymmetric functions $\QSym,$ sending the isomorphism  class $[L_\alpha]$ of 
the 0-Hecke algebra irreducible module $L_\alpha$ indexed by the composition $\alpha\vDash n$  to 
the fundamental quasisymmetric function indexed by $\alpha$:
\[ \mathrm{ch}([L_\alpha])= F_\alpha.\]

For the purposes of this paper, the pertinent  description of the quasisymmetric characteristic is in \cite[Section 5.4]{KrobThibon1997}. Let $M$ be a finite-dimensional $\hn$-module;  let $M=M_1\supset M_2\supset \cdots \supset M_k\supset M_{k+1}=\mathbb{K}$ be a composition series of submodules for $M$, so that   each successive quotient $M_i/M_{i+1}$ is irreducible, and thus equal to $L_{\alpha^i}$ for some composition $\alpha^i\vDash n$. The \emph{quasisymmetric characteristic} of the module $M$, $\chr(M)$, is then defined to be  the sum of fundamental quasisymmetric functions  $\sum_{i=1}^k F_{\alpha^i}$.
In particular we will henceforth simply  write $\chr(L_\alpha)=F_\alpha$ for each $\alpha\vDash n$.
\end{definition}

The following is our restatement of the main result of \cite{BBSSZ2015}.
\begin{theorem}\label{thm:BBSSZ2dualImm}\cite[Theorem 3.12]{BBSSZ2015} There is an indecomposable cyclic 0-Hecke algebra module $\mathcal{W}_\alpha$ whose quasisymmetric characteristic is the dual immaculate function $\dI_\alpha, $  
\[ \chr(\mathcal{W}_\alpha)= \dI_\alpha.\]

The module $\mathcal{W}_\alpha$ has dimension equal to the number of standard immaculate tableaux of shape $\alpha.$

The $\dI$-action of the  0-Hecke algebra generator $\pi_i$  on the set of standard immaculate tableaux of shape $\alpha$,  for $\alpha\vDash n$, may be described as follows:
\begin{equation}\label{eqn:defn-dualImm-pi(T)}\pi_i^{\dI}(T)=\begin{cases} T, &\text{if $i+1$ is in a row weakly below }i,\\
0,  & \text{if $i,i+1$ are in column 1 of } T,\\
s_i(T), & \text{if $i+1$ is strictly above $i$ in $T$},\\ &\text{and $i,i+1$ are NOT in column 1},
\end{cases}\end{equation}
where $s_i(T)$ is the standard immaculate tableau obtained from $T$ by swapping $i$ and $i+1.$ 

Furthermore, the dual immaculate function $\dI_\alpha$ expands positively in the fundamental quasisymmetric functions as follows.  Define the descent set $\Des_{\dI}(T)$ of a standard immaculate tableau $T$ to be the set 
\[\mathrm{Des}_{\dI}(T):=\{i: i+1 \text{ is strictly above } i \text{ in } T\}.\]
Then 
\[\mathrm{ch}(\mathcal{W}_\alpha)=\dI_\alpha=\sum_T F_{\mathrm{comp}(\mathrm{Des}_{\dI}(T))},\]
where the sum runs over all standard immaculate tableaux $T$ of shape $\alpha.$
\end{theorem}

Our goal in this section is to define, for each composition $\alpha$ of $n,$ a module $\mathcal{V}_\alpha$ whose image under the characteristic map is the quasisymmetric function $\mathcal{R}\mathfrak{S}^*_\alpha.$  Following \cite{BBSSZ2015}, we consider the vector space 
$\mathcal{V}_\alpha$ whose basis vectors are the standard immaculate tableaux of shape $\alpha.$ 
Define, for each $1\le i\le n-1$ and each standard immaculate tableau $T$ of shape $\alpha,$ the $\rdI$-action of the generator $\pi_i$ on $T$ to be 
\begin{equation}\label{eqn:defn-RSdualImm-pi(T)}\pi_i^{\rdI}(T)=\begin{cases} T, &\text{if } 
i\notin \mathrm{Des}_{\mathcal{R}\mathfrak{S^*}}(T), \\
0,& i\in \mathrm{Des}_{\mathcal{R}\mathfrak{S^*}}(T) \text{ and swapping } i \text{ and } i+1 \text{ in } T \\
\phantom{0,} & \text{does NOT result in a standard immaculate tableau},\\
s_i(T), & \text{otherwise},
\end{cases}\end{equation}
where $s_i(T)$ is the standard immaculate tableau obtained from $T$ by swapping $i$ and $i+1.$ 

As in \cite{BS2021}, we say an entry $j$ in a tableau is right-adjacent to an entry $i$ if $i,j$ are in the same row and in adjacent columns, and $j$ is to the right of $i$.    Note that from the definition of the descent set, if 
$i\in \mathrm{Des}_{\mathcal{R}\mathfrak{S^*}}(T)$ and $i, i+1$ are in the same row, then $i+1$ must be right-adjacent to $i.$ 

To avoid cumbersome notation, we write simply $\pi_i(T)$ for the row-strict immaculate action, using 
   $\pi_i^{\rdI}$ and $\pi_i^{\dI}$ only when there is explicit need to distinguish between the actions of~\eqref{eqn:defn-RSdualImm-pi(T)} and~\eqref{eqn:defn-dualImm-pi(T)}.  We refer to these as the $\rdI$-action and the $\dI$-action respectively.  Likewise we may refer to the resulting $\hn$-modules as the  $\rdI$-Hecke module and the $\dI$-Hecke module respectively.
\begin{lemma}\label{lem:prelim} Let $T$ be a standard immaculate tableau and let $i\in \mathrm{Des}_{\mathcal{R}\mathfrak{S^*}}(T).$  Then 
\begin{enumerate}
\item
$i, i+1$ cannot both be in the leftmost column of $T;$
\item 
$ i \text{ and } i+1 \text{ are in the same row } \iff i+1\text{ is right-adjacent to }i \iff \pi_i(T)=0;$
\item if $s_i(T)$ is a standard immaculate tableau, then $i\notin \mathrm{Des}_{\mathcal{R}\mathfrak{S^*}}(s_i(T)).$
\end{enumerate}
\end{lemma}
\begin{proof} The first claim is immediate from the definitions. 

For the second claim, one direction is clear: if $i+1$ is right-adjacent to $i$, clearly swapping $i$ and $i+1$ will make a non-increasing row, so $\pi_i(T)=0.$ 

Now suppose $i\in \mathrm{Des}_{\mathcal{R}\mathfrak{S^*}}(T)$ and $i, i+1$ are not in the same row, so that $i+1$ is in a row strictly below $i.$  Here it is clear that swapping $i, i+1$ does not violate row-increase (since $a<i<b\iff a<i+1<b$ for $a, b \notin \{i,i+1\}$).  Since by the first claim, $i, i+1$ are not both in the leftmost column, the latter is still increasing by the same argument.  Hence $s_i(T)$ is also an immaculate tableau, and $\pi_i(T)=s_i(T)\ne 0.$

Note from above that $s_i(T) $ is an immaculate tableau $\iff$ $i+1$ is strictly below $i$ in $T,$ and hence $i+1$ is strictly above $i$ in $s_i(T).$ This verifies the third and final claim.
\end{proof}
In particular, if $1\in \mathrm{Des}_{\mathcal{R}\mathfrak{S^*}}(T),$ then $\pi_1(T)=0.$ 

\begin{example} Consider the standard row-strict immaculate tableau
\[S=\tableau{6\\4&5&8&10\\3&7\\1&2&9}.\]   Then $\Des_{\rdI}(S)=\{1,4,6,8\}$ and 
$\mathrm{Des}_{\mathfrak{S^*}}(T)=\{2,3,5,7,9\}$ (the complement in the set $\{1,2,\ldots,9\}$).
Hence we have \[\pi^{\rdI}_i(T)=T \text{ for } i\in\{2,3,5,7,9\}, \ \ \pi^{\rdI}_1(T)=0=\pi^{\rdI}_4(T), \] and 
\[\pi^{\rdI}_6(T)=s_6(T)=\tableau{{\bf 7}\\4&5& 8&10\\3& {\bf 6}\\1&2&9} , \quad  \pi^{\rdI}_8(T)=s_8(T)=\tableau{6\\4&5&  {\bf 9}&10\\3&7\\1&2& {\bf 8}}.\]
\end{example}

We can therefore reformulate the description of the action of $\pi_i^{\rdI}$ on $T$ more succinctly as follows:

\begin{equation}\label{eqn:defn-pi(T)}\pi_i(T)=\pi_i^{\rdI}(T)=\begin{cases} T, & \text{if $i+1$ is strictly above $i$ },\\
0,  & \text{if $i+1$ is right-adjacent to $i$},\\
s_i(T), & \text{if $i+1$ is strictly below $i$ in $T$}.
\end{cases}\end{equation}

\begin{theorem}\label{thm:Hecke-action-RSImm} The operators $\pi^{\rdI}_i$ define an action of the 0-Hecke algebra on the vector space $\mathcal{V}_\alpha.$ 
\end{theorem}
\begin{proof}  Clearly from the preceding analysis, $\pi^{\rdI}_i(T)=\pi_i(T)\in \mathcal{V}_\alpha$  for every standard immaculate tableau $T$ of shape $\alpha.$ We must verify that the operators satisfy the 0-Hecke algebra relations.

To show $\pi_i^2(T)=\pi_i(T),$ we need only check the case when $i+1$ is strictly below $i$ in $T$. In this case $\pi_i(T)=s_i(T),$ and $i$ is now strictly below $i+1$. Hence $\pi_i(s_i(T))=s_i(T)$ and we are done.

Let $1\le i,j\le n-1$ with $|i-j|\ge 2.$ Then $\{i, i+1\}\cap \{j, j+1\}=\emptyset,$ so the actions of $\pi_i$ and $\pi_j$ are independent of each other, and hence commute.

It remains to show that 
\begin{equation}\label{eqn:keyclaim}\pi_i \pi_{i+1} \pi_i(T)=\pi_{i+1} \pi_i \pi_{i+1}(T).
\end{equation}  We examine  four separate cases.

\begin{enumerate}
\item[Case 1:] Assume $i\notin \mathrm{Des}_{\mathcal{R}\mathfrak{S^*}}(T), i+1\notin \mathrm{Des}_{\mathcal{R}\mathfrak{S^*}}(T)$: 
Then $\pi_i(T)=T, \pi_{i+1}(T)=T, $ and the claim is clear.

\item[Case 2:] Assume $i\in \mathrm{Des}_{\mathcal{R}\mathfrak{S^*}}(T), \text{ but }i+1\notin \mathrm{Des}_{\mathcal{R}\mathfrak{S^*}}(T)$: Then $\pi_{i+1}(T)=T.$ 
If $\pi_i(T)=0,$ then the left-hand side of \eqref{eqn:keyclaim} is 0, and so is the right-hand side.

If $\pi_i(T)\ne 0,$ then $\pi_i(T)=s_i(T),$ and the left-hand side of \eqref{eqn:keyclaim} equals 
$\pi_i\pi_{i+1}(s_i(T)),$ while the right-hand side is $\pi_{i+1}(s_i(T)).$ Hence \eqref{eqn:keyclaim} becomes 
\begin{equation}\label{eqn:step1}\pi_i \pi_{i+1} ( s_i(T))= \pi_{i+1}(s_i(T)),
\end{equation}
which we need to verify.

Assume $i+1\notin \mathrm{Des}_{\mathcal{R}\mathfrak{S^*}}(s_i(T)).$ 
The left-hand side then equals $\pi_i(s_i(T))=s_i(T)$  by Lemma~\ref{lem:prelim}, and this is also the right-hand side.

Finally assume $i+1\in \mathrm{Des}_{\mathcal{R}\mathfrak{S^*}}(s_i(T)).$   We now have $i+1$ strictly below $i$ in $T,$ so that $i+1$ is strictly above $i$ in $s_i(T),$ and $i+2$ weakly below $i+1$ in $s_i(T).$   Also recall that $i+2$ was strictly above $i+1$ in $T.$ It follows that 
\begin{equation}\label{eqn:step2} \text{In } s_i(T),  i+2 \text{ is now strictly above } i 
\text{ and weakly below }i+1. \end{equation}

If $i+2$ is right-adjacent to $i+1,$ then $\pi_{i+1}(s_i(T))=0$ and Equation~\eqref{eqn:step1} is immediate.

If not, then 
\begin{equation}\label{eqn:step3} \text{In } s_i(T),  i+2 \text{ is now strictly above } i 
\text{ and strictly below }i+1. \end{equation} 
This implies $\pi_{i+1}(s_i(T))=s_{i+1}(s_i(T)),$ and in the latter we now have $i$ below $i+1,$ which in turn is below $i+2.$ In particular $i$ is not a descent of $\pi_{i+1}(s_i(T))=s_{i+1}(s_i(T)),$ and hence the latter tableau is fixed by $\pi_i.$ Equation~\eqref{eqn:step1} is thus verified.

\item[Case 3:]  Assume $i\notin \mathrm{Des}_{\mathcal{R}\mathfrak{S^*}}(T), \text{ but }i+1\in \mathrm{Des}_{\mathcal{R}\mathfrak{S^*}}(T)$: 

Then $\pi_i(T)=T$ and the left-hand side of Equation~\eqref{eqn:keyclaim} is $\pi_i(\pi_{i+1}(T))$ 
\[=\begin{cases} 0, & i+2 \text{ is right-adjacent to } i+1 \text{ in } T, \\
                          \pi_i(s_{i+1}(T)), 
                          &\text{otherwise.} \end{cases}\]
The right-hand side of Equation~\eqref{eqn:keyclaim} is $\pi_{i+1}\pi_i\pi_{i+1}(T)$ 
\[=\begin{cases} 0, &i+2 \text{ is right-adjacent to } i+1 \text{ in } T, \\
                          \pi_{i+1}(\pi_i(s_{i+1}(T))), &\text{otherwise.} \end{cases}\]
                          
Thus we need only consider the case when $i+2$  is not right-adjacent to $i+1$, so that necessarily $i+2$ is strictly below $i+1$ in $T.$ Also $i$ is  strictly BELOW $i+1$ in $T$. That is, $i+1$ is strictly above both $i$ and $i+2$ in $T,$ and thus  $i+2$ is strictly above both $i,i+1$ in $s_{i+1}(T).$ 
Hence we have  two possibilities for $s_{i+1}(T)$:

 Either $i$ is below $i+1$, which is below $i+2$, and hence 
$\pi_i$ and $\pi_{i+1}$ both fix $s_{i+1}(T),$
or $i+1$ is below $i$, and $i$ is below $i+2$.  In the latter case, applying $\pi_i$ to $s_{i+1}(T)$ switches $i$ and $i+1$, so that in $\pi_i(s_{i+1}(T))$ we now have $i$ below $i+1$, and $i+1$ (still) below $i+2$.  But then $\pi_{i+1}$ fixes $\pi_i(s_{i+1})(T)$.  Equation~\eqref{eqn:keyclaim} has been established.

\item[Case 4:] Assume $i, i+1\in \mathrm{Des}_{\mathcal{R}\mathfrak{S^*}}(T)$:

First suppose $i+1$ is right-adjacent to $i,$ so that $\pi_i(T)$ equals zero, and so  does the left-hand side of 
Equation~\eqref{eqn:keyclaim}.  If $\pi_{i+1}(T) =0,$ we are done; if not $\pi_{i+1}(T)=s_{i+1}(T),$ and the right-hand side of Equation~\eqref{eqn:keyclaim} is 
\[\pi_{i+1}\pi_i (s_{i+1}(T)).\]

Since in $T$ we have $i+1$ right-adjacent to $i$ and $i+2$ strictly below both, this means in $s_{i+1}(T)$ we have $i+2$ right-adjacent to $i$ and $i+1$ strictly below them, forcing $i$ to be a descent of $s_{i+1}(T).$ Hence in $\pi_i (s_{i+1}(T)) $, we have $i+2$ right-adjacent to $i+1$ 
and $i$ strictly below.  But then $\pi_{i+1} (  \pi_i (s_{i+1}(T)) ) =0, $ as desired.

Finally, suppose $\pi_i(T)\ne 0,$ so that $i+1$ is strictly below $i$  in $T$ and $i+2$ is weakly below $i+1.$ For clarity we consider two sub-cases:
\begin{enumerate}
\item[Case 4a:] $i+2$ is right-adjacent to $i+1$ in $T$. This immediately makes the right-hand side of Equation~\eqref{eqn:keyclaim} equal to 0, since $\pi_{i+1}(T)=0.$ 
 Then the left-hand side of 
 Equation~\eqref{eqn:keyclaim} equals $\pi_i\pi_{i+1} s_i(T) =\pi_i s_{i+1}(s_i(T))$ since 
$i+1$ is now above $i+2$ in $s_i(T).$   But now $i+1$ is right-adjacent to $i$ in $s_{i+1}(s_i(T))$,  so $\pi_i s_{i+1}(s_i(T))$ reduces to 0, as desired.

\item[Case 4b:]  $i+2$ is strictly below $i+1,$ which is strictly below $i$ in $T,$ so that $\pi_{i+1}(T)\ne 0, \pi_i(T)\ne 0.$ Then Equation~\eqref{eqn:keyclaim} becomes 
\begin{equation}\label{eqn:step4} s_i s_{i+1} s_i(T) = s_{i+1} s_i s_{i+1}(T), \end{equation}
and it is easy to see that this is indeed true.  
\end{enumerate}
\end{enumerate}
We have verified Equation~\eqref{eqn:keyclaim} in all cases, thereby completing the proof that the action of the generators $\pi_i$ extends to an action of $H_n(0)$ on $\mathcal{V}_\alpha.$ 
\end{proof}

\section{A partial order and a $\hn$-module $\mathcal{V}_{\alpha}$ for $\rdI_\alpha$}\label{sec:partial-order} 

Let $\alpha \vDash n$, and let $\SIT(\alpha)$ denote the set of all standard immaculate tableaux of shape $\alpha$. Given $T\in \SIT(\alpha)$, let $\sigma (T) \in S_n$ be the permutation in one-line notation obtained from $T$ by reading the entries of $T$ from \emph{right to left} in each row, and the rows from \emph{top to bottom}.

\begin{example}\label{ex:perm}
If $T= \tableau{4&5\\2\\1&3}$ then $\sigma(T) = 5\ 4\ 2\ 3\ 1$.
\end{example}

Given $\sigma \in S_n$ recall that its \emph{inversion set} is
$$\mathrm{Inv} (\sigma) = \{ (p,q)\suchthat 1\leq p < q \leq n \mbox{ and } \sigma(p) > \sigma (q)\}.$$The number of inversions is denoted by $\ninv(\sigma) = |\mathrm{Inv}(\sigma)|.$

\begin{example}\label{ex:inv} 
If $\sigma(T) = 5\ 4\ 2\ 3\ 1$, then $\ninv(\sigma) = 9$ from 
$$\mathrm{Inv}(\sigma) = \{(1,2), (1,3), (1,4), (1,5),  (2,3), (2,4), (2,5), (3,5), (4,5)\}.$$
\end{example}

Observe by our definition of $\pi _i$ that if $T_1, T_2 \in \SIT(\alpha)$ and
$$\pi _i (T_1) = T_2$$with $T_1\neq T_2$, this means that
\begin{enumerate}
\item in $T_1$ we have that $i$ appears strictly above $i+1$, so in $\sigma(T_1)$ we have that $i$ appears left of $i+1$.
\item in $T_2 = \pi _i(T_1)$ we have that $i$ appears strictly below $i+1$, so in $\sigma(T_2)$ we have that $i$ appears right of $i+1$.
\end{enumerate}
Consequently, since all  entries other than $i$ and $i+1$ are fixed by $\pi _i$, we have that if $T_2 = \pi _i(T_1)$ with $T_1\neq T_2$, then
\begin{equation}\label{eq:inv}
\ninv (\sigma (T_2)) = \ninv (\sigma (T_1))+1.
\end{equation}

\begin{proposition}\label{prop:poset} Let $\alpha \vDash n$ and $T_1, T_2 \in \SIT (\alpha)$. Define $\poRI$ on $\SIT(\alpha)$ by
$$T_1 \poRI T_2 \mbox{ if and only if there exists a permutation } s_{i_1} \cdots s_{i_\ell} \in S_n$$such that
$$\pi _{i_1} \cdots \pi _{i_\ell} (T_1) = T_2.$$Then $\poRI$ is a partial order on $\SIT(\alpha)$.
\end{proposition}

\begin{proof} That $\poRI$ is reflexive and transitive is immediate from the definition. To prove antisymmetry, if  $T_1 \poRI T_2$ then by Equation~\eqref{eq:inv}
$$\ninv (\sigma (T_1)) \leq \ninv (\sigma (T_2))$$and if $T_2 \poRI T_1$ then by Equation~\eqref{eq:inv}
$$\ninv (\sigma (T_2)) \leq \ninv (\sigma (T_1)).$$ Thus $\ninv (\sigma (T_1)) = \ninv (\sigma (T_2)).$ However, by Equation~\eqref{eq:inv} we know each non-zero, non-identity action of a generator $\pi_i$  increases the number of inversions, hence $T_1=T_2$ as desired.
\end{proof}

We will now use our partial order to define an $\hn$-module indexed by a composition $\alpha \vDash n$, whose quasisymmetric characteristic is the row-strict dual immaculate function $\RI _\alpha$. More precisely, given a composition $\alpha \vDash n$ extend the partial order $\poRI$ on $\SIT (\alpha)$ to an arbitrary total order on $\SIT (\alpha)$, denoted by $\poRI ^t$. Let the elements of $\SIT (\alpha)$ under $\poRI ^t$ be $$\{\rtau _1 \poRI ^t  \cdots \poRI ^t \rtau _m  \}.$$ Now let $\mathcal{V}_{\rtau_i}$ be the $\mathbb{C}$-linear span 
 $$\mathcal{V}_{\rtau_i} = \spam \{ \rtau _j \suchthat \rtau _i \poRI ^t \rtau _j \}\quad \text{ for } 1\leq i \leq m$$
and observe that the definition of $\poRI ^t$ implies that $\pi_{i_1}\cdots \pi _{i_\ell}\mathcal{V}_{\rtau_i}\subseteq \mathcal{V}_{\rtau_i}$ for any $s_{i_1}\cdots s _{i_\ell}\in S_n$. This observation combined with the fact that the operators $\{\pi_i\}_{i=1}^{n-1}$ satisfy the same relations as $\hn$ by Theorem~\ref{thm:Hecke-action-RSImm} gives the following result.
\begin{lemma}\label{lem:hnmodule}
$\mathcal{V}_{\rtau_i}$ is an $\hn$-module.
\end{lemma}

Note that $\mathcal{V}_{\rtau_1}$ is precisely the module $\mathcal{V}_\alpha$ of the preceding section.

Given the above construction, define $\mathcal{V}_{\rtau_{m+1}}$ to be the trivial $\hn$-module, and consider the following filtration of $\hn$-modules.
$$
\mathcal{V}_{\rtau_{m+1}}\subset \mathcal{V}_{\rtau_m} \subset \cdots \subset \mathcal{V}_{\rtau_{2}}\subset \mathcal{V}_{\rtau_1} 
$$
Then the quotient modules $\mathcal{V}_{\rtau_{i-1}}/\mathcal{V}_{\rtau_{i}}$ for $2\leq i\leq m+1$ are $1$-dimensional $\hn$-modules spanned by $\rtau_{i-1}$. Furthermore, they are irreducible modules. We can identify which $\hn$-module they are by looking at the action of $\pi_{j}$ on $\mathcal{V}_{\rtau_{i-1}}/\mathcal{V}_{\rtau_{i}}$ for $1\leq j\leq n-1$. We have
\begin{eqnarray*}
\pi_{j}(\rtau_{i-1})&=& \left \lbrace \begin{array}{ll}0 & j\in \desri(\rtau_{i-1})\\\rtau_{i-1} & \text{otherwise.}\end{array}\right.
\end{eqnarray*}
Thus, as an $\hn$-representation, $\mathcal{V}_{\rtau_{i-1}}/\mathcal{V}_{\rtau_{i}}$ is isomorphic to the irreducible representation $L_{\beta}$ where $\beta$ is the composition corresponding to the descent set $\desri(\rtau_{i-1})$. Hence 
$\mathrm{ch}(\mathcal{V}_{\rtau_{i-1}}/\mathcal{V}_{\rtau_{i}})=F_{\comp(\desri(\rtau_{i-1}))}$, and 
\begin{eqnarray*}
\mathrm{ch}(\mathcal{V}_\alpha)=\mathrm{ch}(\mathcal{V}_{\rtau_1})
&=& \displaystyle \sum_{i=2}^{m+1}F_{\comp(\desri(\rtau_{i-1}))}\nonumber\\ 
&=& \displaystyle \sum_{\rtau \in \SIT(\alpha)}F_{\comp(\desri(\rtau))}\nonumber\\ 
&=& \RI_{\alpha}.
\end{eqnarray*}
Consequently, we have established the following.  

\begin{theorem}\label{the:bigone}
Let $\alpha \vDash n$, and let  $\rtau _1 \in \SIT (\alpha)$ be the minimal element under the total order $\poRI ^t$. Then $\mathcal{V}_\alpha=\mathcal{V}_{\rtau_1}$ is an $\hn$-module whose quasisymmetric characteristic is the row-strict dual immaculate function $\RI_{\alpha}$.\end{theorem}

\section{The 0-Hecke module structure of $\mathcal{V}_\alpha$}\label{sec:indecomp-module-and-poset}

Our next goal is to analyse the structure of the module $\mathcal{V}_\alpha$ in Theorem~\ref{the:bigone}. 
We will show that our partial order, and hence any linear extension of it, has a unique bottom element $S^0_\alpha$. We will also show that $\mathcal{V}_\alpha$  is cyclic, and generated by the 
standard tableau $S^0_\alpha$. 

\begin{lemma}\label{lem:uniqueness1} Let $\alpha\vDash n$ and $T\in \SIT(\alpha).$  If $\pi_i(T)=s_i(T), \pi_j(T)=s_j(T)\in \SIT(\alpha)$ and  $\pi_i(T)=\pi_j(T),$ then necessarily $i=j.$
\end{lemma}
\begin{proof} If $i\ne j,$ we may assume $i< j.$ Suppose $j+1$ occupies cell $x$ in $T.$ In $\pi_i(T)$,  $j+1>i+1$ is unchanged from its position in $T.$  However, in $\pi_j(T)=s_j(T)$, $j+1$ is now strictly above $j,$ and cell $x$ is now occupied by $j.$ It follows that $\pi_i(T),\pi_j(T),$ differ at least in cell $x,$ a contradiction.
\end{proof}

The cover relation for our poset $P\rdI_\alpha$ for row-strict dual immaculate tableaux of shape $\alpha\vDash n$  is 
\begin{equation}\label{eqn:rdIcover} 
S\poRIcover  T
\iff \exists\, i \text{ such that }  T=\pi_i^{\rdI}(S),
\end{equation}
 \noindent with respect to the  \textbf{row-strict} 0-Hecke action defined by Theorem~\ref{thm:Hecke-action-RSImm}.

On the other hand, the cover relation for the poset $P\dI_\alpha$ for dual immaculate tableaux of shape $\alpha\vDash n,$ as described in \cite{BBSSZ2015}, is
\begin{equation}\label{eqn:dIcover} 
S\poIcover  T
\iff \exists\, i \text{ such that }  S=\pi^{\dI}_i(T).
\end{equation}
\noindent with respect to the 
 \textbf{dual immaculate} 0-Hecke action defined  by Theorem~\ref{thm:BBSSZ2dualImm}.

Note that by Lemma~\ref{lem:uniqueness1}, in each case, a cover relation is determined by a unique generator of the 0-Hecke algebra.

\begin{lemma}\label{lem:sameposet} Let $\alpha\vDash n$ and $S,T\in \SIT(\alpha).$ Then 
\[S\poIcover  T \iff S\poRIcover  T.\]
Hence  the two posets $P\dI_\alpha$ and $P\rdI_\alpha$ are isomorphic.
\end{lemma}
\begin{proof} We have 
\begin{align*} S\poIcover  T&\iff \pi_i^{\dI}(T)=S, \text{ where }S\ne T, S\ne 0\\
& \iff i+1 \text{ is strictly above $i$ in } T, \\ 
& \qquad\quad i, i+1 \text{ not both in column 1, by } \eqref{eqn:defn-dualImm-pi(T)}\\
&\iff i+1 \text{ is strictly BELOW $i$ in } S\\
&\iff \pi_i^{\rdI} (S)=s_i(S)=T, \text{ by } \eqref{eqn:defn-pi(T)}\\
&\iff S\poRIcover  T.
\end{align*}
The claim follows.
\end{proof}

From  Equation~\eqref{eq:inv} and remarks preceding Proposition~\ref{prop:poset}, it follows that the posets $P\dI_\alpha\simeq P\rdI_\alpha$ are graded by the number of inversions. We call the poset $P\dI_\alpha\simeq P\rdI_\alpha$ the \emph{immaculate Hecke poset} associated to the composition $\alpha$.

Next we show that the poset has a unique bottom element $S_\alpha^0$. 
From this we will show that the unique bottom element $S_\alpha^0$ is the cyclic generator for the module $\mathcal{V}_\alpha$ in   Theorem~\ref{the:bigone}, whose quasisymmetric characteristic is $\rdI_\alpha.$  Later we will also show that it has a unique top element $S^{row}_\alpha$, see Definition~\ref{def:top-elt}.

\begin{definition}\label{def:bot-elt} Let $\alpha\vDash n$ be of length $\ell=\ell(\alpha).$ Define $S^0_\alpha$ to be the standard tableau of shape $\alpha$ with entries $1,\ldots, \ell$ in column 1, increasing bottom to top, and then fill the remaining rows, top to bottom, left to right with consecutive integers starting at $\ell+1$ and ending at $n.$ Thus 
\begin{enumerate}
    \item the top row, row $\ell,$ contains the entry $\ell$ followed by the interval $[\ell+1, \ell+\alpha_\ell-1];$ note that 
    $n=\sum_{i=1}^\ell \alpha_i\ge (\ell-1)+\alpha_\ell;$
    \item the next row, row $\ell-1,$ contains the entry $\ell-1$ followed by the interval $[\ell+\alpha_\ell, \ell+\alpha_\ell+\alpha_{\ell-1}-2];$
    \item row $i$ (from the bottom) contains the entry $i$ followed by the interval     $[\ell+\alpha_\ell+\alpha_{\ell-1}+\cdots +\alpha_{i+1}-(\ell-i-1), \ell+\alpha_\ell+\cdots+\alpha_{i}-(\ell-i+1)]$ (note that there are $\alpha_i -(\ell-i+1)+(\ell-i-1)+1 +1=\alpha_i$ entries; also note that $n=\sum_{i=1}^\ell \alpha_i\ge (i-1)+ \sum_{j=i}^\ell \alpha_i$);
    \item the bottom row, row 1, contains 1 followed by the interval $[\ell+\alpha_\ell+\alpha_{\ell-1}+\cdots +\alpha_{2}-(\ell-2), \ell+\alpha_\ell+\cdots+\alpha_{1}-\ell].$
\end{enumerate}
\end{definition}
\begin{example}\label{ex:bot-elt1} We have 
\[S^0_{43423}=\tableau{\bf{5} &6 &7\\
                      \bf{4} &8\\
                      \bf{3}&9 & 10 & 11\\
                      \bf{2} & 12 & 13\\
                      \bf{1} & 14 & 15 & 16}\ , \qquad  
S^0_{43411}=\tableau{\bf{5} \\
                      \bf{4} \\
                      \bf{3}&6 & 7 & 8\\
                      \bf{2} & 9 & 10\\
                      \bf{1} & 11 & 12 & 13}\ .\]
\end{example}
 The complement of the descent set of $S^0_\alpha$ is precisely $[\ell-1]$; all entries not in column 1 are descents, and $\ell$ is also a descent if $\ell\ne n$.  Furthermore, the entries $i$ at the end of a row of size 2 or more, are the only ones for which 
$\pi_i^{\rdI}(S^0_\alpha)=s_i(S^0_\alpha)\ne 0.$  The following facts are a consequence of \eqref{eqn:defn-pi(T)}, and are illustrated by the above examples.
\begin{lemma}\label{lem:bot-elt-facts} Let $\alpha\vDash n$ and $\ell=\ell(\alpha).$ Then
\begin{enumerate}
\item 
\begin{equation}\label{eqn:fix-bot-elt}  
\pi_i^{\rdI}(S^0_\alpha)=S^0_\alpha\iff i\in [\ell-1].
\end{equation}
\item Assume $\alpha_\ell\ge 2.$ Then 
\begin{equation}\label{eqn:fix-bot-elt1}  
\pi_\ell^{\rdI}(S^0_\alpha)=0, \text { and }
\end{equation}
\noindent
\begin{equation}\label{eqn:annihilate-bot-elt1}
 \pi_j^{\rdI}(S^0_\alpha)\notin \{S^0_\alpha, 0\}
 \iff j\in \{\ell+\alpha_\ell+\cdots+\alpha_{i}-(\ell-i+1), 2\le i\le \ell\}.
\end{equation}
(Note that $i=1$ corresponds to $j=n.$) 
\item Assume $\alpha_\ell=1$ and $k\le \ell-1$ is maximal such that $\alpha_k\ge 2.$   Then
\begin{equation}\label{eqn:annihilate-bot-elt2}
\pi_j^{\rdI}(S^0_\alpha)\notin \{S^0_\alpha, 0\}
\iff j=\ell \text{ or } j\in \{\ell+\alpha_\ell+\cdots+\alpha_{i}-(\ell-i+1),\, 2\le i\le k\}.
\end{equation}
\item If there is no $j$ such that $\pi_j(S^0_\alpha)=0,$ then $\alpha_\ell=1$ and $\alpha_k\le 2$ for all $k\le \ell-1.$
\item Finally if $\alpha$ is a hook of the form $ (1^{\ell-1},n-\ell+1), 1\le \ell\le n$, then $S^0_\alpha$ is the unique standard tableau of shape $\alpha.$
\end{enumerate}
\end{lemma}

To avoid trivialities, unless otherwise stated, in everything that follows we will assume that  $\alpha\notin\{(1^{\ell-1},n-\ell+1): 1\le \ell \le n\}$.

\begin{proposition}\label{prop:bot-elt} Let $\alpha\vDash n$ and consider the standard immaculate tableau $S^0_\alpha$.  Then for any $T\in \SIT(\alpha), T\ne S^0_\alpha,$ there is a sequence of generators $\pi_{j_i}, $ and tableaux $T_i\in \SIT(\alpha),$ $i=1,\ldots, r,$  such that $\pi_{j_i}(T_{i})=T_{i-1}, i=1,2,\ldots ,r,$ where we set $T_{0}=T$ and $T_r=S^0_\alpha.$ Hence we conclude 
\[ S^0_\alpha \poRI T \text{ and } T=\pi_{j_1}\pi_{j_{2}}\cdots\pi_{j_r}(S^0_\alpha).\]
In particular $S^0_\alpha$ is the unique minimal element of the poset $P\rdI_\alpha.$
\end{proposition}

The proof consists of a straightening algorithm which we first illustrate with an example.  Let $T\in \SIT(\alpha)\ne S^0_\alpha.$  We describe how to work backwards from $T$ to $S^0_\alpha$ in the poset along a saturated chain. This can be viewed as a \lq\lq straightening" of the tableau $T$, which  transforms it into $S^0_\alpha.$

\begin{example}\label{ex:rdI-straightening1}  Let $\alpha=223,$ so that 
 $S^0_\alpha=\tableau{3 &4 &5\\2 &6\\1 &7},$ and let $T=\tableau{5 &6 &7\\2 &4\\1&3}$.
 First we straighten column 1 of $T$ to match column 1 of $S^0_\alpha$: Start with the lowest
  entry $a_j$  in column 1 of $T$ such that $a_j\ne j,$ and exchange it with $a_j-1.$ 
  Continue in this manner until column 1 has the entries $3,2,1$ from top to bottom.
 \[\small{T=T_0\stackrel{\pi_4}{\longleftarrow} T_1=\tableau{\bf{4} &6 &7\\2 &5\\1&3}
 \stackrel{\pi_3}{\longleftarrow} T_2=\tableau{\bf{3} &6 &7\\2 &5\\1&4}}.\]
 Next we work on the top row of $T_2,$ starting with the smallest entry which differs from the corresponding entry in the same cell of $S^0_\alpha,$ and continuing until the top rows match:
 \[T_2 \!  \stackrel{\pi_5}{\longleftarrow}\! T_3=\tableau{3 &\bf{5} &7\\2 &6\\1&4}\!\!\!
  \stackrel{\pi_4}{\longleftarrow} T_4=\tableau{3 &\bf{4} &7\\2 &6\\1&5}\!\!\!
  \stackrel{\pi_6}{\longleftarrow} T_5=\tableau{3 &4 &\bf{6}\\2 &7\\1&5}\!\!\!
 \stackrel{\pi_5}{\longleftarrow} T_6=\tableau{3 &4 &\bf{5}\\2 &7\\1&6}\!\!.
  \]
  Now move down to the next row from the top, and proceed in the same manner, finding the smallest entry that differs from the corresponding entry in $S^0_\alpha$:
\[T_6   \stackrel{\pi_6}{\longleftarrow} T_7=\tableau{3 &4 &5\\2 &\bf{6}\\1&7}=S^0_\alpha.
  \]
Hence we have 
\[T=\pi_4\pi_3\pi_5\pi_4\pi_6\pi_5\pi_6(S^0_\alpha).\]
\end{example}

\begin{proof}[Proof of Proposition~\ref{prop:bot-elt}]  As illustrated by Example~\ref{ex:rdI-straightening1}, the following algorithm identifies a unique saturated  chain from $S^0_\alpha$ to $T.$  In what follows, when an integer  $a$ occupies row $i$ and column $j$ of $T,$ by the counterpart of $a$ in $S^0_\alpha$ we will mean the integer occupying the same cell, row $i$ and column $j$, in $S^0_\alpha.$ 
\begin{enumerate}
    \item[Step 1:] We begin by making the first column of $T$ match the first column of $S^0_\alpha.$  Find the least $j$, $2\le j\le \ell,$ such that the entry $x$ in cell $(j,1)$ is not equal to $j.$ Then $x-1$ is in a lower row, not in column 1 by minimality of $j$. Hence $T=\pi_{x-1}(T_1),$ such that in $T_1\in\SIT(\alpha),$  $x-1$ is now a descent strictly \textit{higher} than $x.$  Now repeat this procedure until $x$ is replaced by $j.$ Then continue with the next entry in column 1 of $T$ which does not match 
    in $S^0_\alpha.$  Clearly this process ends with a tableau $T_r=\pi_{x_r}\pi_{x_{r-1}}\ldots \pi_{x_1}(T),$ whose first column matches  column 1 of $S^0_\alpha.$
    
    Note that $T=T_r$ if $T$ and $S^0_\alpha$ already agree in the first column.
    \item[Step 2:] First observe that $T_r$ and $S^0_\alpha$ now agree for all entries less than or equal to $\ell=\ell(\alpha).$ Now consider the top-most row of length greater than 1, say row $k$.  Find the least entry, say $y,$ in this row of $T_r$ which differs from the corresponding entry in $S^0_\alpha.$ Note that $y$ is then necessarily larger than its counterpart in $S^0_\alpha,$ by definition of the latter. Then $y\ge \ell+1$ and $y-1$ is strictly below $y$ in $T_r.$  Hence 
    $T_r=\pi_{y-1}(T_{r+1})$ for $T_{r+1}\in \SIT(\alpha),$ such that $y-1$ is a descent in $T_{r+1}$  strictly \textit{higher} than $y$.  We repeat this step until $y$ has been replaced by its counterpart in $S^0_\alpha.$  
    
    \item[Step 3:] Continue in this manner to the end of the row. We now have a sequence of operators $\pi_{i_j}$ and tableaux $T_j\in \SIT(\alpha)$ such that $T_{j-1}=\pi_{i_j}(T_j)$, and the final tableaux $T_s$ agrees with $S^0_\alpha$ for all entries $\le \ell+(\alpha_k-1).$ 
    
    \item[Step 4:] Proceed  downwards to the next row where an entry in $T_s$ differs from its counterpart in $S^0_\alpha,$ and repeat Steps 2 and 3, until all rows are exhausted. 
\end{enumerate}
Since at the end of each iteration of Step 3,  the number of entries that are in agreement  with $S^0_\alpha$ increases, we see that the algorithm produces a saturated chain from $S^0_\alpha$ to $T$ in the poset $P\rdI_\alpha$ as claimed.
\end{proof}

We now immediately have:
\begin{theorem}\label{thm:cyclic-rdI}  The module $\mathcal{V}_\alpha$ of Theorem~\ref{the:bigone}, whose quasisymmetric characteristic is $\rdI_\alpha$,  is cyclic and generated by the unique minimal element $S^0_\alpha$  of the poset $P\rdI_\alpha.$
\end{theorem}
\begin{proof} This is clear since  Proposition~\ref{prop:bot-elt} shows that 
\[T\in \SIT(\alpha), T\ne S^0_\alpha\Rightarrow S^0_\alpha\poRI T,\] and that there is a sequence of generators $\pi_{j_1}\pi_{j_1}\cdots \pi_{j_r} $ such that $T=\pi_{j_1}\pi_{j_1}\cdots \pi_{j_r}(S^0_\alpha).$
\end{proof}
Before proceeding with our analysis of the structure of $\mathcal{V}_\alpha$, some remarks on the analogous results in \cite{BBSSZ2015} are in order. There   the existence of the unique top element  of $P\dI_\alpha$ is implicitly deduced via a map between an analogue of Yamanouchi words which the authors call $\mathcal{Y}$-words, which are in  bijection with standard tableaux. The authors   show indirectly that the top element of the poset $P\dI_\alpha\simeq P\rdI_\alpha$ is a cyclic generator for the module $\mathcal{W}_\alpha$ of Theorem~\ref{thm:BBSSZ2dualImm}, by appealing to a parent module generated by a special Yamanouchi word; they invoke the fact that  $\mathcal{W}_\alpha$ is a quotient of this module.

In order to give a complete and self-contained analysis of the immaculate Hecke poset $P\dI_\alpha\simeq P\rdI_\alpha,$  we will explicitly establish the existence of the unique top element  by means of a straightening algorithm similar to Proposition~\ref{prop:bot-elt}.  In analogy with Theorem~\ref{thm:cyclic-rdI}, we will also be able to deduce that the $\hn$-module $\mathcal{W}_\alpha$  is cyclically generated by this top element.

\begin{definition}\label{def:top-elt} Define   
$ S^{row}_\alpha$ to be the \textit{row superstandard} tableau in $\SIT(\alpha),$ whose entries left to right, bottom to top, are $1,2, \ldots, n-1, n$ in consecutive order. Note that if $\alpha$ is a hook of the form $ (1^{\ell-1},n-\ell+1), 1\le \ell\le n$,  then $\SIT(\alpha)$ has cardinality one and $S^0_\alpha=S^{row}_\alpha$.
\end{definition}
\begin{example}\label{ex:top-elt1} We have 
\[S^{row}_{43423}=\tableau{\bf{14} &15 &16\\
                      \bf{12} &13\\
                      \bf{8}&9 & 10 & 11\\
                      \bf{5} & 6 & 7\\
                      \bf{1} & 2 & 3 & 4}\ , \qquad 
                      S^{row}_{43411}=\tableau{\bf{13} \\
                      \bf{12} \\
                      \bf{8}&9 & 10 & 11\\
                      \bf{5} & 6 & 7\\
                      \bf{1} & 2 & 3 & 4}\ .\]
\end{example}

\begin{proposition}\label{prop:top-elt} Let $\alpha\vDash n$ and consider the standard immaculate tableau $ S^{row}_\alpha$.  Then for any $T\in \SIT(\alpha), T\ne S^{row}_\alpha,$ there is a sequence of generators $\pi_{j_i}, $ and tableaux $T_i\in \SIT(\alpha),$ $i=1,\ldots, r,$ such that $\pi_{j_i}(T_{i-1})=T_{i}, i=1,2,\ldots ,r,$ where we set $T_{0}=T$ and 
$T_r=S^{row}_\alpha.$ Hence we conclude 
\[ T \poRI S^{row}_\alpha \text{ and } \pi_{j_r}\pi_{j_{r-1}}\cdots\pi_{j_1}(T)=S^{row}_\alpha.\]
In particular $ S^{row}_\alpha$ is the unique maximal element of the poset $P\rdI_\alpha.$
\end{proposition}

We illustrate the straightening algorithm with an example.
\begin{example}\label{ex:rdI-top-straightening}
We have 
 $ S^{row}_{223}=\tableau{5 &6 &7\\3 &4\\1&2}.$ Let  $T=T_0=\tableau{3 &4 &7\\2 &6\\1 &5}$. Start with the largest entry in the top-most row of $T$ that differs from its counterpart in $ S^{row}_{223},$ and repeat:
 \[T_0\stackrel{\pi_4}{\longrightarrow} \tableau{3 &\bf{5} &7\\2 &6\\1 &4}\!\!\!\!=T_1
 \stackrel{\pi_5}{\longrightarrow}\tableau{3 &\bf{6} &7\\2 &5\\1 &4}\!\!\!\!=T_2
 \stackrel{\pi_3}{\longrightarrow} \tableau{\bf{4} &6 &7\\2 &5\\1 &3}\!\!\!\!=T_3
 \stackrel{\pi_4}{\longrightarrow} \tableau{\bf{5} &6 &7\\2 &4\\1 &3}\!\!\!\!=T_4.
 \]
Once the top row coincides in both, move down to the next row where there is disagreement, again starting with the largest such element in that row:
\[T_4   \stackrel{\pi_2}{\longrightarrow} \tableau{5 &6 &7\\\bf{3} &4\\1 &2}=S^{row}_\alpha.\]
Hence $\pi_2\pi_4\pi_3\pi_5\pi_4(T)=S^{row}_\alpha.$
\end{example}

\begin{proof}[Proof of Proposition~\ref{prop:top-elt}]
In this case the algorithm is simpler to describe. We start with the top-most row of $T$ which differs from $ S^{row}_\alpha,$ and find the \textit{largest} entry, say $x$, in that row which differs from its counterpart in $ S^{row}_\alpha.$ Then necessarily $x+1$ is in a lower row of $T=T_0,$ and hence $T_1=\pi_x(T_0)$ for the tableau $T_1\in \SIT(\alpha)$ obtained from $T$ by switching $x, x+1.$  We continue in this manner along each row, top to bottom, right to left.  Note that at the start of a new row $j,$ the tableau $T_k$ agrees with $ S^{row}_\alpha$ for all entries in the complement of $[\alpha_1+\cdots+\alpha_j].$

It is clear that this algorithm terminates in the tableau $ S^{row}_\alpha$, producing a saturated chain in the poset from $T$ to $ S^{row}_\alpha.$
\end{proof}

In view of Lemma~\ref{lem:sameposet}, and combined with the filtration in \cite{BBSSZ2015}, analogous to our filtration of Lemma~\ref{lem:hnmodule}, we recover Theorem~\ref{thm:BBSSZ2dualImm}, specifically the results of \cite[Theorem 3.5, Lemma 3.10]{BBSSZ2015}:

\begin{theorem}\label{thm:BBSSZ2-cyclic}\cite[Lemma 3.10]{BBSSZ2015} The module $\mathcal{W}_\alpha$ whose  quasisymmetric characteristic is equal to $\dI_\alpha,$ is cyclically  generated by the standard immaculate tableau $ S^{row}_\alpha.$
\end{theorem}

Comparing the two filtrations which give the $\hn$-modules, we may summarise the situation as follows:

We may take the  total order on the poset $P\dI_\alpha\simeq P\rdI_\alpha$ to be  the same in both cases, thanks to Lemma~\ref{lem:sameposet}:
\[S^0_\alpha= T_1\preccurlyeq^t T_2 \preccurlyeq^t\dots \preccurlyeq^t T_m=S^{row}_\alpha,\]
where $m=|\SIT(\alpha)|.$

For $\mathcal{W}_\alpha,$ in \cite{BBSSZ2015} the authors use 
the filtration  
\[(0)\subset \mathrm{span}([S^0_\alpha, T_1])
\subset \dots \subset 
\mathrm{span}([S^0_\alpha, T_i])\subset \dots \subset \mathrm{span}([S^0_\alpha, S^{row}_\alpha]),\]
while for $\mathcal{V}_\alpha,$
our filtration (see Lemma~\ref{lem:hnmodule}) is
\[ (0)\subset\mathrm{span}([T_m,S^{row}_\alpha]) 
\subset\dots\subset \mathrm{span}([T_i,S^{row}_\alpha])\subset \dots \subset 
\mathrm{span}([S^0_\alpha, S^{row}_\alpha]).\]
\begin{remark}\label{rem:poset-rank}
The poset $P\rdI_\alpha$ is ranked by the function $\inv$. By computing the difference $\inv(S^{row}_\alpha)-\inv(S^0_\alpha)$, we see that 
its rank is given by 
\[\binom{n}{2}+\binom{\ell}{3}-\sum_{i=1}^{\ell} \binom{\alpha_{i}+i-1}{2}\]
\end{remark}
Returning to the row-strict dual immaculate functions, our final task in this section is to show that 
the cyclic module $\mathcal{V}_\alpha$ generated by $S^0_\alpha$ is in fact indecomposable.
Since in any ring $R$ with unity, if $e$ is an idempotent we have 
the decomposition $R=eR\oplus (1-e)R,$ to show indecomposability, it is enough to show that if $f$ is an idempotent endomorphism of the module $V,$ then $f=0$ or $f$ is  the identity \cite[Proposition 3.1]{Jacobson1989}.

\begin{theorem}\label{thm:Indecomp} Let $\mathcal{V}_\alpha$ be the cyclic $\hn$-module generated by $S^0_\alpha$. Then $\mathcal{V}_\alpha$ is indecomposable, and $\chr(\mathcal{V}_\alpha)=\rdI_\alpha$.
\end{theorem}

The proof of this theorem will require a series of technical lemmas.

\begin{lemma}\label{lem:indecomplemma0}\cite[Proof of Theorem 3.12]{BBSSZ2015} Let $P\in \SIT(\alpha)$ and let $j\in [n-1]$ such that $\pi_j(P)\ne P.$ Then $P$ cannot equal $\pi_j(T)$ for any tableau $T\in \SIT(\alpha).$
\end{lemma}

\begin{proof} Immediate since otherwise we would have $P=\pi_j(T)\Rightarrow \pi_j(P)=\pi_j^2(T)=\pi_j(T)=P,$ 
a contradiction.
\end{proof}

\begin{definition}\label{def:SIT-1st-col} For a composition $\alpha\vDash n$ of length $\ell,$ denote by $\SIT^*(\alpha)$ the set of standard immaculate tableaux whose first (left-most) column consists of the integers 
$1,\ldots,\ell.$ Equivalently, $\SIT^*(\alpha)$ is the set of standard immaculate tableaux such that $\ell=\ell(\alpha)$ is in row $\ell,$ the top row of $\alpha.$
\end{definition}

\begin{lemma}\label{lem:indecomplemma1} Suppose $f$ is an idempotent  endomorphism of the $\hn$-module $\mathcal{V}_\alpha$, with 
\[f(S^0_\alpha)=\sum_{T\in \SIT(\alpha)} a_T T. \] 
Let $P\in \SIT(\alpha), P\ne S^0_\alpha,$ and 
suppose there is a $j$ such that $j\in \mathrm{Des}_{\rdI}(P)\setminus\Des_{\rdI}(S^0_\alpha).$
Then $a_P=0.$ 

As a consequence,  we have 
\[f(S^0_\alpha)=\sum_{T\in \SIT^*(\alpha)} a_T T.\]
\end{lemma}

\begin{proof} From Equation~\eqref{eqn:defn-pi(T)}, we have the following implications:
\[j\in \Des_{\rdI}(P) \Rightarrow \pi_j(P)\ne P,\] 
 and 
\[  j\notin \Des_{\rdI}(S^0_\alpha) \Rightarrow \pi_j(S^0_\alpha)=S^0_\alpha.   \]

Hence
\[f(S^0_\alpha)=f(\pi_j(S^0_\alpha))=\pi_j(f(S^0_\alpha))=\sum_T a_T \pi_j(T),\]
and by Lemma~\ref{lem:indecomplemma0}, this does not contain $P$ in its expansion. 
That is, 
$f(S^0_\alpha)=\sum_T a_T T$ does not contain $P$ in its expansion, and hence $a_P=0.$
It follows that $a_P=0$ unless $\mathrm{Des}_{\rdI}(P)$ is contained in $\Des_{\rdI}(S^0_\alpha),$ which equals the complement of $[\ell-1]$ from Equation~\eqref{eqn:fix-bot-elt}. In particular, $a_P=0$ unless $1, \ldots , \ell-1$ are NOT descents of $P.$ But then it is easy to see (since rows must increase left to right) that $1, \ldots , \ell$ must occupy the first column of $P.$  Hence we have 
\[a_P\ne 0 \Rightarrow 1, \ldots , \ell \text{ occupy the first column of } P,\]
which is the claim. 
\end{proof}

\begin{lemma}\label{lem:indecomplemma4-a} Let $\alpha\vDash n, \ \ell(\alpha)=\ell,$ and let $P \in \SIT^*(\alpha)$ be such that $\ell+1$ is right-adjacent to $i$ for some $1\le i\le \ell.$    Then 
either $\pi_\ell(P)=0$ or there is an $i\in[\ell-1]$ such that 
\[P\poRIcover P_1=\pi_\ell(P)\! \poRIcover P_2=\pi_{\ell-1}(P_1)\!\poRIcover
\cdots \!\poRIcover P_{\ell-i}=\pi_{i+1}(P_{\ell-i-1})\]
is a saturated chain in the poset $P\rdI_\alpha,$  with $i+1$ right-adjacent to $i$ in $P_{\ell-i}.$
In either case, there is a sequence of generators $\pi_{i_1},\ldots, \pi_{i_r}$  such that $\pi_{i_1}\cdots \pi_{i_r}(P)=0.$
\end{lemma}
\begin{proof} Since $\pi_\ell(P)=0 \iff \ell+1$ is right-adjacent to $\ell,$ we may assume $\ell+1$ is in row $i\le \ell-1.$ 
Thus $\ell+1$ is strictly below $\ell$, and hence $\pi_\ell(P)=s_\ell(P)$ swaps $\ell$ and $\ell+1.$ That is, in 
$P_1=\pi_\ell(P),$ $\ell$ is now right-adjacent to $i\le \ell-1$.  
If $i=\ell-1,$  the argument stops here.  If not, 
now  $\ell$ is strictly below $\ell-1,$ and hence applying $\pi_{\ell-1}$ produces a tableau $s_{\ell-1}(P_1)$ in which $\ell-1$ is now right-adjacent to $i\le \ell-2.$  Clearly we can continue this process until $i+1$ becomes right-adjacent to $i,$ with the appropriate replacements in the first column above the entry $i,$ 
producing the saturated chain as claimed.

Now applying $\pi_i$ results in 0, and hence 
 $\pi_i \pi_{i+1}\cdots\pi_\ell(P)=0$.
\end{proof}

\begin{example}\label{ex:indecomplemma4-a} Let $\ell=5$, $\alpha=23223.$ 
  We have the following saturated chain beginning at $P\in \SIT^*(\alpha)$:
\[P=\tableau{5 & 7 & 8\\ 4 & 10\\3 & 9\\2 &\bf{6} & 12\\1 & 11}\stackrel{\pi_5}{\longrightarrow} P_1=\tableau{6 & 7 & 8\\ 4 & 10\\3 & 9\\2 &\bf{5} & 12\\1 & 11}\stackrel{\pi_4}{\longrightarrow}
P_2=\tableau{6 & 7 & 8\\ 5 & 10\\3 & 9\\2 &\bf{4} & 12\\1 & 11}
\stackrel{\pi_3}{\longrightarrow}
P_3=\tableau{6 & 7 & 8\\ 5 & 10\\4 & 9\\2 &\bf{3} & 12\\1 & 11},\]
and $\pi_2\pi_3\pi_4\pi_5(P)=\pi_2(P_3)= 0.$
\end{example}

\begin{lemma}\label{lem:indecomplemma4-b} Let $\alpha\vDash n, \ \ell(\alpha)=\ell,$ and let $P,Q \in \SIT^*(\alpha)$ be such that $\ell+1$ is 
\begin{itemize}
\item right-adjacent to $i$ in $P$ for some $2\le i\le \ell,$
\item right-adjacent to $j$ in $Q$ for some $1\le j\le i-1.$ 
\end{itemize}
Let $\hat\pi$ denote the sequence of generators $\pi_i \pi_{i+1}\pi_{i+2}\cdots\pi_\ell.$ Then 
$\hat\pi(P)=0, \hat\pi(Q)\ne 0,$ and 
$\mathrm{rank}(\hat\pi(Q))=(\ell-i+1)+\mathrm{rank}(Q).$  
\end{lemma}
\begin{proof} This is clear from the preceding proof, since it takes 
$(\ell-i)$ consecutive applications of the $\pi_j$ to make 
$i+1$ right-adjacent to $i$ in $P;$  applying $\pi_\ell, \pi_{\ell-1}, \ldots, \pi_i$ successively to $Q$ changes the first column, and, in column 2, it changes only the entry in row $j$ right-adjacent to $j$, diminishing the latter by 1 at each step.  Hence after applying the  $(\ell-i+1)$ factors in the sequence  $\hat\pi$, the entry right-adjacent to $j$ in $Q$ is $i,$ and this is greater than $j$ by hypothesis.  This means that one final application of $\pi_i$ 
gives $\hat \pi(P)=0$, but $\hat \pi(Q)\ne 0.$ 

The statement about the ranks is clear since at each step we are applying a $\pi_k$ to the standard tableau to produce another standard tableau.
\end{proof}

\begin{example}\label{ex:indecomplemma4-b} Let $\ell=5$, $\alpha=23223,$ take $P$ as in Example~\ref{ex:indecomplemma4-a} and consider the 
   saturated chain beginning at $Q\in \SIT^*(\alpha)$:
\[Q=\tableau{5 & 7 & 8\\ 4 & 10\\3 & 9\\2 &11 & 12\\1 & \bf{6}}\stackrel{\pi_5}{\longrightarrow} Q_1=\tableau{6 & 7 & 8\\ 4 & 10\\3 & 9\\2 &11 & 12\\1 & \bf{5}}\stackrel{\pi_4}{\longrightarrow}
Q_2=\tableau{6 & 7 & 8\\ 5 & 10\\3 & 9\\2 &11 & 12\\1 & \bf{4}}
\stackrel{\pi_3}{\longrightarrow}
Q_3=\tableau{6 & 7 & 8\\ 5 & 10\\4 & 9\\2 &11 & 12\\1 & \bf{3}}\]
\[\stackrel{\pi_2}{\longrightarrow}
Q_4=\tableau{6 & 7 & 8\\ 5 & 10\\4 & 9\\3 &11 & 12\\1 & \bf{2}};
\qquad \text{ hence for the sequence }\hat\pi=\pi_2\pi_3\pi_4\pi_5,\  \hat\pi(Q)=Q_4\ne 0.\]
\end{example}

Before proceeding with the proof of the next lemma, it is instructive to work through two more examples.
\begin{example}\label{ex:indecomplemma4-c} Let $\alpha=33223.$ 
\[\text{Then } S^0_\alpha=\tableau{5 & 6 & 7\\4 &8\\3 &9\\2 &\bf{10} &11\\ 1 &12 &13};
\qquad \text{ we take } P=\tableau{5 & 6 & 7\\4 &8\\3 &9\\2 &12 &13\\ 1 &10 &11}.\]

Here $p=\bf{10}$ is right-adjacent to 2 in $S^0_\alpha$ and to 1 in $P$.

Applying $\pi_{p-1}, \pi_{p-2},\ldots$ in succession we have 
\[P \stackrel{\pi_9}{\longrightarrow} \tableau{5 & 6 & 7\\4 &8\\3 &10\\2 &12 &13\\ 1 &\bf{9} &11}
\stackrel{\pi_8}{\longrightarrow} \tableau{5 & 6 & 7\\4 &9\\3 &10\\2 &12 &13\\ 1 &\bf{8} &11}
\stackrel{\pi_7}{\longrightarrow} \tableau{5 & 6 & 8\\4 &9\\3 &10\\2 &12 &13\\ 1 &\bf{7} &11}
\stackrel{\pi_6}{\longrightarrow} \tableau{5 & 7 & 8\\4 &9\\3 &10\\2 &12 &13\\ 1 &\bf{6} &11},\]
so that $\ell+1=6$ ends up right-adjacent to 1 in $\pi_6\pi_7\pi_8\pi_9(P).$

However, it is clear that the same sequence applied to $S^0_\alpha$ results in 
\[S^0_\alpha \stackrel{\pi_9}{\longrightarrow} \tableau{5 & 6 & 7\\4 &8\\3 &10\\2 &\bf{9} &11\\ 1 &12 &13}
\stackrel{\pi_8}{\longrightarrow}\tableau{5 & 6 & 7\\4 &9\\3 &10\\2 &\bf{8} &11\\ 1 &12 &13}
\stackrel{\pi_7}{\longrightarrow}\tableau{5 & 6 & 8\\4 &9\\3 &10\\2 &\bf{7} &11\\ 1 &12 &13}
\stackrel{\pi_6}{\longrightarrow}\tableau{5 & 7 & 8\\4 &9\\3 &10\\2 &\bf{6} &11\\ 1 &12 &13}\]
\end{example}

Hence $\bf{6}$ ends up right-adjacent to 1 in $\pi_6\pi_7\pi_8\pi_9(P)$ but to $2$ in $\pi_6\pi_7\pi_8\pi_9(S^0_\alpha).$

\begin{example}\label{ex:indecomplemma4-d} This is the case $p-1=\ell+1$ in the lemma below.

Let $S^0_\alpha=\tableau{9\\8\\7\\6&10\\5 & \bf{11}\\ 4& 12\\3 \\ 2 &13\\ 1}$, $\ P=\tableau{9\\8\\7\\6&10\\5 & 13\\ 4& 12\\3 \\ 2 &\bf{11}\\ 1}.$\qquad  Then we have 
$P\stackrel{\pi_{10}}{\longrightarrow}\tableau{9\\8\\7\\6&11\\5 & 13\\ 4& 12\\3 \\ 2 &\bf{10}\\ 1}.$

Again, $\bf{10}$ ends up right-adjacent to 2 in $\pi_{10}(P)$, but is right-adjacent to a larger entry, 6, in $S^0_\alpha.$
\end{example}

We are now ready to precisely formulate and prove the facts illustrated by the above examples. 

\begin{lemma}\label{lem:indecomplemma4-c} Let $\alpha\vDash n, \ \ell(\alpha)=\ell,$ and let $P \in \SIT^*(\alpha), P\ne S^0_\alpha$ be such that for all $j,$ whenever $j, j+1$ are right-adjacent in $S^0_\alpha,$ they are also right-adjacent in $P.$  
Let $p$ be the smallest entry in $S^0_\alpha$, say in cell $x,$ which differs from the entry in cell $x$ of $P.$ Then 
\begin{enumerate}
    \item $p\ge \ell+1;$
    \item $S^0_\alpha$ and $P$ coincide in the top-most row, row $\ell;$
    \item  the cell $x$ is located in column 2 and row $i$ for some $i$ such that $\alpha_i\ge 2;$ in Cartesian coordinates, cell $x$ is in position $(i,2)$.  
    \item $S^0_\alpha$ and $P$ coincide in all rows above the row $i$ containing cell $x,$ and hence all entries less than $p$ occupy the same cells in both tableaux;
    \item the entry $p$ must occur in $P$ in column 2 and in a lower row $j$ where $j<i$ and $\alpha_j\ge 2;$ i.e. $p$ occupies the cell with coordinates $(j,2), j<i.$
    \item When $p\ge \ell+2,$ let $\tau$ be the permutation $\tau:= s_{\ell+1 }s_{\ell+2}\cdots s_{p-2}s_{p-1}$ (a reduced word of length $p-\ell-1$); then applying $\pi_\tau=\pi_{\ell+1 }\pi_{\ell+2}\cdots \pi_{p-2}\pi_{p-1}$ to $P$ (resp. $S^0_\alpha$) replaces $p$ with $\ell+1,$ by diminishing it by 1 at each step, making $\ell+1$ right-adjacent to $j$ in $P$ (resp. right-adjacent to $i$ in $S^0_\alpha$.)  When $p= \ell+1,$ $\tau$ is the identity since $\ell+1$ is already right-adjacent to $j$ in $P$ (resp. right-adjacent to $i$ in $S^0_\alpha$), with $i>j$; this is because $\ell+1$ is already in a lower row in $P$ than in $S^0_\alpha$. 
\end{enumerate}
\end{lemma}

\begin{proof} Note first that the hypothesis includes the case when there is no $j$ such that $j, j+1$ are right-adjacent in $S^0_\alpha.$ This forces all parts of $\alpha$ to be at most 2. 

Item (1) is clear.   Item (2) follows by hypothesis, since all the entries of  $S^0_\alpha$ are consecutive  beginning in column 1 for the top row.  

Rows of length $\ge 3$ have consecutive entries from  column 2 onwards in both tableaux, by hypothesis.  The entry in column 2 of such a row is uniquely determined by any other entry in that row, and hence the least entry that differs must occur in column 2.  This is Item (3).

Item (4) is clear by minimality of the entry $p$ in cell $x,$ since all entries in $S^0_\alpha$ occurring to the left of cell $x,$ or higher than it,  must be less than $p.$ Also rows increase left to right, so in $P$, if the entry $p$ were not in column 2, there would be an entry $q,$ $\ell+1\le q<p,$ immediately to its left in its row.  But this contradicts the fact that all entries strictly less than $p$ occupy exactly the same cells in both tableaux.  This establishes Item (5).

Now note that when $p=\ell+1,$ Item (6) is immediate for $\tau$ equal to the identity permutation.  See Example~\ref{ex:indecomplemma4-d} with $P$ replaced with $\pi_{10}(P)$.  Thus we may assume $p\ge \ell+2.$

It is now clear from Example~\ref{ex:indecomplemma4-c} what happens; applying $\pi_{p-1}$ puts $p-1$ into the position formerly occupied by $p,$ in cell $(j,2)$, of $P$, $j<i,$ and moves $p$ UP to the cell formerly occupied by $p-1.$ But in $S^0_\alpha,$ it puts $p-1$ in cell $(i,2)$, and moves $p$ to the same cell as in $P$. This is by virtue of Item (4).

Similarly, applying $\pi_{p-2}$ next will move $p-2$ into the position $(j,2)$, and $p-1$ will now move UP.  In 
$S^0_\alpha,$  it puts $p-2$ into cell $(i,2)$ and moves $p-1$ into the same cell as in $P.$ 

In fact a comparison of the two tableaux after each application of $\pi_k$ shows that they coincide in all rows above the $i$th row and they differ in rows $i$ and $j$, in column 2.

The first step of our induction is depicted in Example~\ref{ex:indecomplemma4-e} below.  Subsequent steps are completely analogous, with $p$ diminished by 1 at each step.

Assume by induction that after applying $\pi_{p-k}$, $(p-k)$ was in cell $(j,2)$ of $P$ and in cell $(i,2)$ of $S^0_\alpha$, (recall $j<i$) and the two tableaux coincide in rows above row $i$. By induction hypothesis, $p-k-1$ occupies the same cell in each tableau, since it is in a row higher than row $i.$ Now apply $\pi_{p-k-1}.$  This swaps $p-k-1$ and $p-k$ in each tableau, putting $p-k-1$ into 
cell $(j,2)$ of $P$ and in cell $(i,2)$ of $S^0_\alpha$, and moving $p-k$ UP to the same cell in both tableaux.

In $P,$ cell $(j,2)$ changes $p$ to $\ell+1.$
In $S^0_\alpha,$ cell $(i,2)$ changes from $p$ to $\ell+1.$
We have thus  established Item (6). 
\end{proof}

\begin{example}\label{ex:indecomplemma4-e} 

\[\text{Let } S^0_\alpha=\tableau{\ell &\scriptstyle{\ell+1}\\ \ldots\\ \ldots\\ \ldots &\scriptstyle{p-1}\\ \ldots\\i & p\\
\ldots &\ldots\\
 j & \scriptstyle{y>p}\\
\ldots\\
\ldots &\ldots \\2\\1}\, ,\  
P=\tableau{\ell &\scriptstyle{\ell+1}\\  \ldots\\ \ldots\\ \ldots &\scriptstyle{p-1}\\ \ldots\\i &\scriptstyle{z>p}\\
\ldots &\ldots\\
 j & p\\
\ldots\\
\ldots &\ldots\\2\\1}\, .\ 
\text{ Then } \pi_{p-1}(P)=\tableau{\ell &\scriptstyle{\ell+1}\\  \ldots\\ \ldots\\ \ldots &\scriptstyle{p}\\ \ldots\\i 
& \scriptstyle{z>p}\\
\ldots &\ldots\\
 j & \scriptstyle{p-1}\\
\ldots\\
\ldots &\ldots\\2\\1}\, .\]

\end{example}

\begin{lemma}\label{lem:indecomplemma4} Let $P\ne S^0_\alpha,$ such that $P\in \SIT^*(\alpha).$    Then there is a sequence $\hat\pi$ of  generators   
$\pi_{{i_1}}\cdots \pi_{{i_r}}$ such that 
\begin{enumerate}
\item  $\hat\pi(S^0_\alpha)=0;$ 
\item $\hat\pi(P)$ is nonzero and has rank $r+\mathrm{rank}(P)$ in the poset $P\rdI_\alpha.$
\end{enumerate}
\end{lemma}
\begin{proof} By hypothesis, $P$ and $S^0_\alpha$ have identical first columns.  

Suppose there is a $j\le n-1$ such that $j, j+1$ are right adjacent in $S^0_\alpha,$ but NOT right-adjacent in $P.$ 
Then from Equation~\eqref{eqn:defn-pi(T)}, $\pi_j(S^0_\alpha)=0$ but $\pi_j(P)\ne 0,$ so we are done.

Otherwise, whenever $j, j+1$ are right-adjacent in $S^0_\alpha,$ they are also right-adjacent in $P.$ Note that the case when $j$ is never adjacent to $j+1$ is included here.  But now Lemma~\ref{lem:indecomplemma4-c}, followed by Lemma~\ref{lem:indecomplemma4-b}, establish the claim.
\end{proof}

\begin{lemma}\label{lem:injectivity} Let $T_1, T_2\in \SIT(\alpha)$ and let
$\hat\pi$ denote the sequence of generators $\pi_{{i_1}}\cdots \pi_{{i_r}}$  
such that $\hat\pi^{\rdI}(T_1)=\hat\pi^{\rdI}(T_2)=U,$ and    $\mathrm{rank}(U)=r+\mathrm{rank}(T_1)=r+\mathrm{rank}(T_2).$ 
Then $T_1=T_2.$
\end{lemma}
\begin{proof} By hypothesis, each application of $\pi_{{i_j}}$ produces a tableau in $\SIT(\alpha)$.  The proof of Lemma~\ref{lem:sameposet} shows that when $S,T\in \SIT(\alpha)$, 
\[\pi^{\rdI}_k(S)=T \iff S=\pi^{\dI}_k(T).\]
Hence we have 
\[T_1=\pi^{\dI}_{{i_r}}\pi^{\dI}_{{i_{r-1}}}\cdots \pi^{\dI}_{{i_1}} (U)  =T_2,\]
as claimed.
\end{proof}

Now  the argument is nearly identical to \cite[Proof of Theorem 4.1]{BS2021}.

\begin{proof}[Proof of Theorem~\ref{thm:Indecomp}] From Lemma~\ref{lem:indecomplemma1}, we have  the equation 
\[f(S^0_\alpha) = \sum_{T\in \SIT^*(\alpha)} a_T T.\]
In the above sum, let $\hat T$ be a tableau of maximal rank in the poset $P\rdI_{\alpha}$ such that $a_{\hat{T}}\ne 0.$ Assume $\hat T\ne S^0_\alpha,$ and let $\hat\pi$ denote the sequence of generators $\pi_{i_1}\cdots \pi_{{i_r}}$  guaranteed by Lemma~\ref{lem:indecomplemma4} for the tableau $\hat T.$ 
 Let $\hat\pi(\hat T)=T'\ne 0.$  In particular, $\mathrm{rank}(T')=r+\mathrm{rank}(\hat T).$
 
We have 
\begin{equation}\label{eq:thm:indecomp2}0=f(\hat\pi(S^0_\alpha))=\hat\pi(f(S^0_\alpha))
 =\sum_{T\in \SIT^*(\alpha)} a_T\ \hat\pi(T).\end{equation}
 Suppose $a_T\ne 0$  and $\hat\pi(T)=T'=\hat\pi(\hat T).$ 

 Comparing ranks, we obtain \[r+\mathrm{rank}(\hat T)=\mathrm{rank}(T')=\mathrm{rank}(\hat\pi( T))\le r+\mathrm{rank}(T),\]
 the weak inequality being due to the fact that some $\pi_{j}$ can fix a tableau.
 This implies $\mathrm{rank}(\hat T)\le \mathrm{rank}( T).$ 
 If the ranks are unequal, however, and $a_T\ne 0,$ we have a contradiction to the maximality of $\mathrm{rank}(\hat T).$ 
 Hence $\mathrm{rank}(\hat T)= \mathrm{rank}( T).$ 
 
By Lemma~\ref{lem:injectivity}, this forces  $T=\hat T.$ 

 It follows that the coefficient of $T'=\hat\pi(\hat T)$ in the right-hand side of  Equation~\eqref{eq:thm:indecomp2} is $a_{\hat T},$ which must therefore be zero, a contradiction.

 Hence \[f(S^0_\alpha)= a_{S^0_\alpha}\ S^0_\alpha;\]
 since $f$ is idempotent, this implies $a_{S^0_\alpha}=0 \text{ or } 1$, finishing the proof.
\end{proof}
\begin{figure}[htb] \centering
		\scalebox{0.6}{			
			\begin{tikzpicture}
			\newcommand*{\xdist}{*3}
			\newcommand*{\ydist}{*2.2}
			\node (n0) at (0.00\xdist,0\ydist) 
			{ \textcolor{blue}{
				$S^0_{223} =
				\tableau{ 
				 3 & 4 & 5 \\ 
				 2 & 6 \\
                      1 & 7}  $ }
			}; 
			\node (n11) at (-1\xdist,1\ydist) 
			{\textcolor{blue}{
				$\tableau{ 
				 3 & 4 & 6 \\ 
				 2 & 5 \\
                      1 & 7}  $  }
			}; 
\node (n12) at (1\xdist,1\ydist) 
			{\textcolor{blue}{
				$\tableau{ 
				 3 & 4 & 5 \\ 
				 2 & 7 \\
                      1 & 6}  $ }
			}; 
			\node (n21) at (-1.5\xdist,2\ydist) 
			{\textcolor{blue}{
				$\tableau{ 
				 3 & 5 & 6 \\ 
				 2 & 4 \\
                      1 & 7}  $ }
			}; 
\node (n22) at (0\xdist,2\ydist) 
			{\textcolor{blue}{
				$\tableau{ 
				 3 & 4 & 7 \\ 
				 2 & 5 \\
                      1 & 6}  $ }
			}; 
\node (n23) at (1.5\xdist,2\ydist) 
			{ \textcolor{blue}{
				$\tableau{ 
				 3 & 4 & 6 \\ 
				 2 & 7 \\
                      1 & 5}  $ }
			}; 
						\node (n31) at (-2\xdist,3\ydist) 
			{
				$\tableau{ 
				 4 & 5 & 6 \\ 
				 2 & 3 \\
                      1 & 7}  $
			}; 
			\node (n32) at (-1\xdist,3\ydist) 
			{\textcolor{blue}{
				$\tableau{ 
				 3 & 5 & 7 \\ 
				 2 & 4 \\
                      1 & 6}  $ }
			}; 
\node (n33) at (1\xdist,3\ydist) 
			{\textcolor{blue}{
				$\tableau{ 
				 3 & 4 & 7 \\ 
				 2 & 6 \\
                      1 & 5}  $}
			}; 
\node (n34) at (2\xdist,3\ydist) 
			{\textcolor{blue}{
				$\tableau{ 
				 3 & 5 & 6\\ 
				 2 & 7 \\
                      1 & 4}  $ }
			}; 

						\node (n41) at (-2\xdist,4\ydist) 
			{
				$\tableau{ 
				 4 & 5 & 7 \\ 
				 2 & 3 \\
                      1 & 6}  $
			}; 
			\node (n42) at (-.5\xdist,4\ydist) 
			{\textcolor{blue}{
				$\tableau{ 
				 3 & 6 & 7 \\ 
				 2 & 4 \\
                      1 & 5}  $ }
			}; 
\node (n43) at (.5\xdist,4\ydist) 
			{\textcolor{blue}{
				$\tableau{ 
				 3 & 5 & 7 \\ 
				 2 & 6 \\
                      1 & 4}  $ }
			}; 
\node (n44) at (2\xdist,4\ydist) 
			{
				$\tableau{ 
				 4 & 5 & 6\\ 
				 2 & 7 \\
                      1 & 3}  $
			}; 

						\node (n51) at (-2\xdist,5\ydist) 
			{
				$\tableau{ 
				 4 & 6 & 7 \\ 
				 2 & 3 \\
                      1 & 5}  $
			}; 
			\node (n52) at (-1\xdist,5\ydist) 
			{\textcolor{magenta}{
				$S^{col}_{223}=\tableau{ 
				 3 & 6 & 7 \\ 
				 2 & 5 \\
                      1 & 4}  $}
			}; 
\node (n53) at (1\xdist,5\ydist) 
			{
				$\tableau{ 
				 4 & 5 & 7 \\ 
				 2 & 6 \\
                      1 & 3}  $
			}; 
\node (n54) at (2\xdist,5\ydist) 
			{
				$\tableau{ 
				 4 & 5 & 6\\ 
				 3 & 7 \\
                      1 & 2}  $
			}; 

	\node (n61) at (-1.5\xdist,6\ydist) 
			{
				$\tableau{ 
				 5 & 6 & 7 \\ 
				 2 & 3 \\
                      1 & 4}  $
			}; 
			\node (n62) at (0\xdist,6\ydist) 
			{\textcolor{red}{
				$\tableau{ 
				 4 & 6 & 7 \\ 
				 2 & 5 \\
                      1 & 3}  $}
			}; 
\node (n63) at (1.5\xdist,6\ydist) 
			{
				$\tableau{ 
				 4 & 5 & 7 \\ 
				 3 & 6 \\
                      1 & 2}  $
			}; 

\node (n71) at (-1\xdist,7\ydist) 
			{\textcolor{red}{
				$\tableau{ 
				 5 & 6 & 7 \\ 
				 2 & 4 \\
                      1 & 3}  $ }
			}; 
\node (n72) at (1\xdist,7\ydist) 
			{\textcolor{red}{
				$\tableau{ 
				  4 & 6 & 7 \\ 
				 3& 5 \\
                      1 & 2}  $ }
			}; 

			\node (n8) at (0\xdist,8\ydist) 
			{\textcolor{red}{
				$S^{row}_{223} =
				\tableau{ 
				 5 & 6 & 7 \\ 
				 3 & 4 \\
                      1 & 2}  $ }
			}; 
			
			\draw [thick, ->] (n0) -- (n11) node [near start, left] {$\pi_{5}$};
                 \draw [thick, ->] (n0) -- (n12) node [near start, right] {$\pi_{6}$};
			\draw [thick, ->] (n11) -- (n21) node [near start, left] {$\pi_{4}$}; 
                \draw [thick, ->] (n11) -- (n22) node [near start, right] {$\pi_{6}$};
			\draw [thick, ->] (n12) -- (n23) node [near start, right] {$\pi_{5}$}; 
			
                \draw [thick, ->] (n21) -- (n31) node [near start, left] {$\pi_{3}$}; 
                \draw [thick, ->] (n21) -- (n32) node [near start, right] {$\pi_{6}$}; 

                \draw [thick, ->] (n22) -- (n32) node [near start, left] {$\pi_{4}$}; 
                \draw [thick, ->] (n22) -- (n33) node [near start, right] {$\pi_{5}$}; 

                \draw [thick, ->] (n23) -- (n33) node [near start, left] {$\pi_{6}$}; 
                \draw [thick, ->] (n23) -- (n34) node [near start, right] {$\pi_{4}$}; 

\draw [thick, ->] (n31) -- (n41) node [near start, left] {$\pi_{6}$}; 
\draw [thick, ->] (n32) -- (n41) node [near start, left] {$\pi_{3}$}; 
\draw [thick, ->] (n32) -- (n42) node [near start, right] {$\pi_{5}$}; 

\draw [thick, ->] (n33) -- (n43) node [near start, left] {$\pi_{4}$}; 
\draw [thick, ->] (n34) -- (n43) node [near start, left] {$\pi_{6}$}; 
\draw [thick, ->] (n34) -- (n44) node [near start, right] {$\pi_{3}$}; 

\draw [thick, ->] (n41) -- (n51) node [near start, left] {$\pi_{5}$}; 
\draw [thick, ->] (n42) -- (n51) node [near start, left] {$\pi_{3}$}; 
\draw [thick, ->] (n42) -- (n52) node [near start, right] {$\pi_{4}$}; 

\draw [thick, ->] (n43) -- (n52) node [left, near start] {$\pi_{5}$};

\draw [thick, ->] (n43) -- (n53) node [near start, right] {$\pi_{3}$}; 
\draw [thick, ->] (n44) -- (n53) node [near start, left] {$\pi_{6}$}; 
\draw [thick, ->] (n44) -- (n54) node [near start, right] {$\pi_{2}$}; 

\draw [thick, ->] (n51) -- (n61) node [near start, left] {$\pi_{4}$}; 
\draw [thick, ->] (n52) -- (n62) node [near start, left] {$\pi_{3}$}; 
\draw [thick, ->] (n53) -- (n62) node [near start, right] {$\pi_{5}$};

\draw [thick, ->] (n53) -- (n63) node [near start, right] {$\pi_{2}$}; 
\draw [thick, ->] (n54) -- (n63) node [near start, left] {$\pi_{6}$}; 

\draw [thick, ->] (n61) -- (n71) node [near start, right] {$\pi_{3}$}; 
\draw [thick, ->] (n62) -- (n71) node [near start, left] {$\pi_{4}$}; 
\draw [thick, ->] (n62) -- (n72) node [near start, right] {$\pi_{2}$};
\draw [thick, ->] (n63) -- (n72) node [near start, left] {$\pi_{5}$}; 

  \draw [thick, ->] (n71) -- (n8) node [near start, right] {$\pi_{2}$}; 
\draw [thick, ->] (n72) -- (n8) node [near start, left] {$\pi_{4}$};               	
			\end{tikzpicture}	
}			
\caption{\small{The row-strict dual immaculate poset $P\rdI_{223}$; the red tableaux and $S^{col}_{223}$ are increasing along rows and up columns, so are in $\textcolor{red}{\SET(223)=[S^{col}_{223}, S^{row}_{223}]}$; the blue tableaux and $S^{col}_{223}=S^{row*}_{223}$ are the elements in $\textcolor{blue}{\SIT^*(223)=[S^0_{223}, S^{row*}_{223}]}$, with the smallest elements in the first column.}}\label{fig:Poset}	
\end{figure}
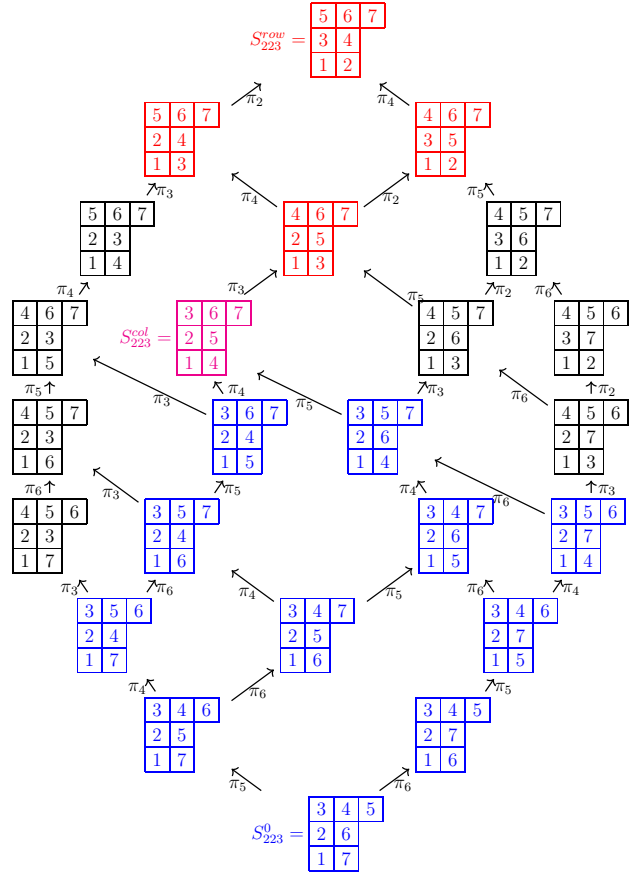

\section{A row-strict analogue for extended Schur functions}\label{sec:row-strict-ext}

This section is motivated by \cite{S2020} and a closer examination of the immaculate Hecke poset $P\rdI_\alpha.$ 
We consider again the $\hn$-module $\mathcal{V}_{\alpha}$ whose quasisymmetric characteristic is $\rdI_\alpha$.

\begin{definition}\label{def:SET} Let $\mathrm{SET}(\alpha)$ be the set of all standard immaculate tableaux of shape $\alpha$ in which ALL columns increase from bottom to top. Equivalently, $\mathrm{SET}(\alpha)$ is simply the set of all standard tableaux of shape $\alpha$ with strictly increasing rows, left to right, and strictly increasing columns, bottom to top. 
\end{definition}

Thus $\mathrm{SET}(\alpha)\subset \SIT(\alpha).$ When the composition $\alpha$ is in fact a partition of $n,$ the set $\mathrm{SET}(\alpha)$ coincides with the set of standard Young tableaux of shape $\alpha.$  In \cite{S2020}, the tableaux in $\SET(\alpha)$ are called 
\textit{s}tandard \textit{e}xtended \textit{t}ableaux.   

Let $\mathcal{Z}_\alpha:=\mathrm{span}\langle\SET(\alpha)\rangle.$ Figure~\ref{fig:Poset} shows the elements of $\SET(223)$ in red; notice that they form a closed interval $[S^{col}_\alpha, S^{row}_\alpha]$ of $P\rdI_{223}$  with respect to our $\rdI$-Hecke action defined 
in~\eqref{eqn:defn-RSdualImm-pi(T)}.  This is no accident.

\begin{lemma}\label{lem:rdIsubmodule} $\mathcal{Z}_\alpha$ is an $\hn$-submodule of $\mathcal{V}_{\alpha}$ for the $\rdI$-action. 
\end{lemma}

\begin{proof} It is enough to verify that $T\in \SET(\alpha)\Rightarrow \pi_i^{\rdI}(T)\in \SET(\alpha).$ From Equation~\eqref{eqn:defn-pi(T)}, we need only consider the case when 
$i+1$ is strictly below $i$ in $T.$ Note that $i+1$ is not in the same column as $i,$ by defnition of $\SET$. The two possible configurations are:
\[T_1=\tableau{\ldots &\ldots &x &\ldots &i\\ 
                 \ldots &\ldots &\ldots  \\
                 \ldots &\ldots &\scriptstyle{i+1} &\ldots &(y) &\ldots} \text{ and } 
T_2=\tableau{\ldots &\ldots &i &\ldots &(x)\\ 
                 \ldots &\ldots &\ldots  \\
                 \ldots &\ldots &y &\ldots &\scriptstyle{i+1} &\ldots}.
\]
The parentheses  indicate the possibility that there may be no cell under $i$ in $T_1,$ and there may be no cell above $(i+1)$ in $T_2.$ But the definition of $\SET(\alpha)$ eliminates $T_1,$ since it forces the contradiction $i+1<x<i.$ Hence $T_2$ is the only possibility, implying $y<i(<x)$ and $y<i+1(<x).$ These conditions are symmetric in $i$ and $i+1$ and hence 
\[\pi^{\rdI}_i(T_2)=s_i(T_2)\in \SET(\alpha).\]
Clearly switching $i$ and $i+1$ does not affect the increasing property of the rows and columns, so this finishes the proof.  \end{proof}

Now $\SET(\alpha)$ is a subposet of $P\rdI(\alpha),$ and it contains the unique top element $ S^{row}_\alpha.$  Thus we can take an arbitrary total order 
\[\{\rtau _1\preccurlyeq^{t}_{\rdI} \cdots \preccurlyeq^{t}_{\rdI}\rtau _m=S^{row}_\alpha\} \]
of $\SET(\alpha)$ exactly as in Section~\ref{sec:partial-order}, obtaining the corresponding  filtration of $\hn$-modules
\[
0\subset \mathcal{Z}_{\rtau_m} \subset \cdots \subset \mathcal{Z}_{\rtau_{2}}\subset \mathcal{Z}_{\rtau_1}=[\rtau_1,S^{row}_\alpha]=\mathcal{Z}_\alpha,
\]
where as before, $\mathcal{Z}_{\rtau_i}$ is the $\mathbb{C}$-linear span 
 $$\mathcal{Z}_{\rtau_i} = \spam \{ \rtau _j \suchthat \rtau _i \poRI ^t \rtau _j \}=[T_i, S^{row}_\alpha]\quad \text{ for } 1\leq i \leq m.$$
The fact that each $\mathcal{Z}_{\rtau_i}$ is an $\hn$-submodule is inherited from the poset $P\rdI(\alpha),$ since we still have $\pi^{\rdI}_{i_1}\cdots \pi^{\rdI} _{i_\ell}\mathcal{Z}_{\rtau_i}\subseteq \mathcal{Z}_{\rtau_i}$ for any $s_{i_1}\cdots s _{i_\ell}\in S_n$.

The same calculation as in the proof of Theorem~\ref{the:bigone} now gives us:
\begin{theorem}\label{thm:rowstrictext-bigone} $\mathcal{Z_\alpha}$ is an $\hn$-submodule of $\mathcal{V}_\alpha$ with quasisymmetric characteristic 
\[\chr(\mathcal{Z_\alpha})=\sum_{T\in\SET(\alpha)} F_{\comp(\Des_{\rdI}(T))}.\]
\end{theorem}

As mentioned at the start of this section, these results were motivated by the observation that $\SET(\alpha)$ is invariant under the $\hn$-action, and by the paper of \cite{S2020}.  In that paper, Searles constructs an $\hn$-module whose quasisymmetric characteristic is the  extended Schur function $\mathcal{E}_\alpha$ defined by Assaf and Searles in \cite{AS2019}.  We restate Searles' formulation of their theorem as follows:

\begin{theorem}\label{thm:Assaf-Searles} \cite{AS2019}, \cite[Theorem 2.7]{S2020}.  Let $\alpha\vDash n.$ Then the extended Schur function 
$\mathcal{E}_\alpha$ expands positively in the fundamental basis of $\QSym$ as follows:
\begin{equation}\label{eq:ext-Schur} 
\mathcal{E}_\alpha= \sum_{T\in \SET(\alpha)} F_{\comp(\Des_{\dI}(T))}.
\end{equation}
\end{theorem}

A remark about this restatement is in order.  Searles defines the descent set of a standard extended tableau $T\in \SET(\alpha)$ as the set 
 \[\{i : i \text{ is weakly to the right of $i+1$ in T}\}.\] 
This coincides with the descent set $\Des_{\dI}(T)$ for dual immaculate tableaux as defined in Theorem~\ref{thm:BBSSZ2dualImm}, since it is easily verified, by an argument analogous to the proof of Lemma~\ref{lem:rdIsubmodule}, that if $T\in \SET(\alpha),$ then 
\[  i  \text{ weakly  right of } i+1 \iff 
i \text{ strictly below } i+1 \iff i\in \Des_{\dI}(T).\]

Assaf and Searles also show that 
\begin{theorem}\label{thm:AS19}\cite{AS2019}  The extended Schur functions $\{\mathcal{E}_\alpha\}_{\alpha\vDash n}$ form a basis of $\QSym,$ with the property that when the composition $\alpha$  is a partition $\lambda$ of $n,$  $\mathcal{E}_\lambda$ coincides with the Schur function $s_\lambda.$
\end{theorem}

Recall from Equation~\eqref{eqn:psiF} that $\psi$ is the $\QSym$-involution which sends the fundamental quasisymmetric function $F_\beta$ to $F_{\beta^c},$ where $\beta^c$ is the complement of the composition $\beta.$  See also \cite[Section 3.6]{LMvW2013}. 
We now immediately have the following:
\begin{proposition}\label{prop:rowstrict-ext-basis} Write $\mathcal{R}\mathcal{E}_\alpha$ for the characteristic of the $\hn$-submodule $\mathcal{Z}_\alpha.$ Then we have 
\[\mathcal{R}\mathcal{E}_\alpha=\sum_{T\in\SET(\alpha)} F_{\comp(\Des_{\rdI}(T))} = \psi(\mathcal{E}_\alpha).\]
 Hence the functions $\{\mathcal{R}\mathcal{E}_\alpha\}_{\alpha\vDash n}$ also form a basis for 
$\QSym$. Furthermore, 
when the composition $\alpha$  is a partition $\lambda$ of $n,$  $\mathcal{R}\mathcal{E}_\lambda$ coincides with the Schur function $s_{\lambda^t}.$
\end{proposition}

\begin{proof} Immediate from the fact that the descent sets $\Des_{\dI}(T)$, $\Des_{\rdI}(T)$ are complements of each other in $[n-1].$ 
\end{proof}

We call the  $\mathcal{R}\mathcal{E}_\alpha$  \textit{row-strict extended Schur} functions.

The extended Schur functions are dual to the \textit{shin} basis of noncommutative symmetric functions  \cite{CFLSX2014}.   It follows that our 
row-strict extended Schur functions $\mathcal{R}\mathcal{E}_\alpha$ also give rise to a dual basis of $\Nsym,$ which we may call the \textit{Rshin} basis.

In analogy with the row superstandard tableau $S^{row}_\alpha$ of Definition~\ref{def:top-elt}, we make the following 
\begin{definition}\label{def:column-bot-elt} Let
$S^{col}_\alpha\in\SET(\alpha)$ denote the \textit{column superstandard} tableau of shape $\alpha,$ whose columns are  filled   bottom to top and left to right, beginning with the first column, with the numbers $\{1,2,\ldots, n\}$  taken in \textbf{consecutive order}. 
As before we note that if $\alpha$ is a hook of the form $ (1^{\ell-1},n-\ell), 1\le \ell\le n$,  then $\SIT(\alpha)$ has cardinality one and $S^0_\alpha=S^{row}_\alpha=S^{col}_\alpha$.
\end{definition}
In Figure~\ref{fig:Poset}, $S^{col}_\alpha$ is the lowest-ranked  tableau in red.   As observed previously,  $S^{row}_\alpha$  is the top element  of the poset $P\rdI_\alpha,$ and hence also of the subposet $\SET(\alpha).$ 

We show next that the module  $\mathcal{Z}_\alpha $ is cyclic and indecomposable.
The following observation is key:
\begin{lemma}\label{lem:Z-keylemma}Suppose  $S=\pi_i(T)$ with $S\in \SET(\alpha)$ and  
$T\in \SIT(\alpha)$.  If $i, i+1$ are NOT in the same column of $S,$ then $T$ must also be in $\SET(\alpha)$, i.e. $T$ must have all columns increasing bottom to top. 
\end{lemma}
\begin{proof} First note that $i$ must be in $\Des_{\rdI}(T),$ and in fact $i+1$ is strictly below $i$ in $T.$ The proof of Lemma~\ref{lem:rdIsubmodule} shows, by definition of $\pi_i,$ that $T$ and $S$ must coincide for all entries not equal to $i,i+1.$ Since switching $i, i+1$ does not affect the increase of their respective columns, it follows that  $T,$ like $S,$ must have all columns increasing.
\end{proof}

\begin{lemma}\label{lem:Zcyclic} Let $\alpha \vDash n.$ Then $S^{col}_\alpha$ is the unique minimal element of the subposet $\SET(\alpha)$, and generates $\mathcal{Z}_\alpha$ as a cyclic 
$\hn$-module.  In particular, $\SET(\alpha)$ coincides with the interval $[S^{col}_\alpha, S^{row}_\alpha]$ in the poset $P\rdI(\alpha).$
\end{lemma}
\begin{proof}   Let $T\in \SET(\alpha), T\ne S^{col}_\alpha.$  We claim that the straightening algorithm of  Proposition~\ref{prop:bot-elt} can be modified to show that there is a sequence of generators $\pi_{j_i}, $ and tableaux $T_i\in \SET(\alpha),$ $i=1,\ldots, r,$  such that $\pi_{j_i}(T_{i})=T_{i-1}, i=1,2,\ldots ,r,$ where we set $T_{0}=T$ and $T_r=S^{col}_\alpha$. 

Note that $S^{col}_\alpha\in \SIT^*(\alpha)\cap\SET(\alpha),$ i.e. the first column of $S^{col}_\alpha$ consists of the entries $[\ell].$ We begin by following Step 1 of the straightening algorithm of  Proposition~\ref{prop:bot-elt}, which  shows how to generate a tableau $T_1$ and a saturated chain in the interval  $[T_1, T]$ of the poset $P\rdI_\alpha,$  such that $T_1$ has the same entries $[\ell]$ in column 1.  Lemma~\ref{lem:Z-keylemma} guarantees that the elements of this chain are in fact in $\SET(\alpha).$ 

We describe the inductive step. Once the first $k-1$ columns have been matched, we proceed to the $k$th column.  Here we  diverge from the proof of Proposition~\ref{prop:bot-elt}, and instead continue the procedure of Step 1. Let $j$ be the largest entry such that $T$ and $S^{col}_\alpha$ coincide for all entries $i\le j.$ 
Then $j\le n-2,$ and $T$ and $S^{col}_\alpha$ coincide for all entries weakly below or to the left of $j$ in $S^{col}_\alpha.$ 
 Note that $j$ cannot be the top-most entry in its column, because this would force $j+1$ into the bottom cell of the next column, and this in turn would contradict the maximality of $j.$  Let $x$ be the entry immediately above $j$ in $T.$  By hypothesis $x>j+1,$ and $x-1$ is strictly to the right, and hence also below, $x$ in $T.$  Hence $T=\pi_{x-1}(T'),$ $T'\in \SET(\alpha)$ by Lemma~\ref{lem:Z-keylemma}, and  now the first disagreement between $T'$ and $S^{col}_\alpha$ has been diminished by 1. This inductive step is repeated until the original entry $x$ is replaced with $j+1,$ so that now our tableau coincides with $S^{col}_\alpha$ for all entries $i\le j+1.$

Since the number of entries in agreement  with $S^{col}_\alpha$ increases at the end of each such step,  continuing in this manner produces the saturated chain in $[S^{col}_\alpha, T]$ as claimed.   
The statement of the lemma now follows.
\end{proof}
The subposet $\SET(223)$ in Figure~\ref{fig:Poset} is very small.  We illustrate the preceding straightening algorithm with two larger examples, using $\alpha=232.$ 
\begin{example}\label{ex:ext-straightening1} Here $S^{col}_{232}=\tableau{3 &6\\2 &5 & 7\\1 &4}$; let $T=S^{row}_{232}=\tableau{6 &7\\ 3 & 4 & 5\\ 1 &2}.$ We have the sequence of operators 
\[T=\tableau{6 &7\\ \bf{3} & 4 & 5\\ 1 &2}\stackrel{\pi_2}{\leftarrow} T_1
=\tableau{\bf{6} &7\\ 2 & 4 & 5\\ 1 &3}
\stackrel{\pi_5}{\leftarrow} T_2=\tableau{\bf{5} &7\\ 2 & 4 & 6\\ 1 &3}
\stackrel{\pi_4}{\leftarrow} T_3=\tableau{\bf{4} &7\\ 2 & 5 & 6\\ 1 &3}\]
\[\stackrel{\pi_3}{\leftarrow} T_4=\tableau{3 &\bf{7}\\ 2 & 5 & 6\\ 1 &4}
\stackrel{\pi_6}{\leftarrow} T_5=\tableau{3 &6\\ 2 & 5 & 7\\ 1 &4}=S^{col}_{232}, \text{and thus } 
  T=\pi_2\pi_5\pi_4\pi_3\pi_6(S^{col}_{232}).\]
\end{example}

\begin{example}\label{ex:ext-straightening2} Here $S^{col}_{232}=\tableau{3 &6\\2 &5 & 7\\1 &4}$; let $T=\tableau{5 &7\\ 3 & 4 & 6\\ 1 &2}.$ We have the sequence of operators 
\[T=\tableau{5 &7\\ \bf{3} & 4 & 6\\ 1 &2}\stackrel{\pi_2}{\leftarrow} T_1
=\tableau{\bf{5} &7\\ 2 & 4 & 6\\ 1 &3}
\stackrel{\pi_4}{\leftarrow} T_2=\tableau{\bf{4} &7\\ 2 & 5 & 6\\ 1 &3}
\stackrel{\pi_3}{\leftarrow} T_3=\tableau{3 &\bf{7}\\ 2 & 5 & 6\\ 1 &4}\]
\[\stackrel{\pi_6}{\leftarrow} T_4=\tableau{3 &6\\ 2 & 5 & 7\\ 1 &4}=S^{col}_{232}, \quad
 \text{and thus } 
  T=\pi_2\pi_4\pi_3\pi_6(S^{col}_{232}).\]
\end{example}

\begin{lemma}\label{lem:Z-indecomplemma} Let $P\in\SET(\alpha),$ $P\ne S^{col}_\alpha.$ Then there is a $j\in [n-1]$ such that $\pi_j(S^{col}_\alpha)=S^{col}_\alpha$ but $\pi_j(P)\ne P.$
\end{lemma}
\begin{proof} By definition of the $\rdI$-action, for any $T\in \SET(\alpha)$, we have 
$\pi_j(T)\ne T\iff j\in \Des_{\rdI}(T).$  In particular, the definition of $S^{col}_\alpha$ implies that 
\[\pi_j(S^{col}_\alpha)\ne S^{col}_\alpha \iff \text{ $j$ is at the top of a column in }S^{col}_\alpha.\]

Let $P\in\SET(\alpha), P\ne S^{col}_\alpha.$  Let $j$ be the largest entry such that $P$ and $S^{col}_\alpha$ coincide for all entries $i\le j.$ Then $j\le n-2,$ and $P$ and $S^{col}_\alpha$ coincide for all entries weakly below or to the left of $j$ in $S^{col}_\alpha.$ Hence, in $P$,  $j+1$ must be strictly to the right of $j$; it cannot be in the same column as $j$ by maximality of $j$ and the fact that $P$ is increasing in rows and in columns.  This also forces  $j+1$ to be either strictly below $j$  or right-adjacent to $j$ in $P$.  In particular $j\in \Des_{\rdI}(P),$ and hence  $\pi_j(P)=s_j(P)\ne P$ or  $\pi_j(P)= 0\ne P$.

If $j\in \Des_{\rdI}(S^{col}_\alpha)$ then $j$ must be at the top of some column of $S^{col}_\alpha,$ but in that case,
 since rows and columns are increasing in $\SET(\alpha)$, $j+1$ must be at the bottom of the next column in both $S^{col}_\alpha$ and $P.$ This contradicts the maximality of $j.$  Hence $j\notin \Des_{\rdI}(S^{col}_\alpha)$ and 
$\pi_j(S^{col}_\alpha)=S^{col}_\alpha.$
\end{proof}

\begin{theorem}\label{thm:Z-indecomp} The cyclic  $\hn$-submodule $\mathcal{Z}_\alpha$ of $\mathcal{V}_\alpha$ is indecomposable.
\end{theorem}

\begin{proof}  Let $f$ be an idempotent $\hn$-module endomorphism of $\mathcal{Z}_\alpha.$ 
Our starting point here is the equation 
\[f(S^{col}_\alpha)=\sum_{T\in\SET(\alpha)} a_T T.\]

Lemma~\ref{lem:Z-indecomplemma} is the analogue of 
\cite[Lemma~3.11]{BBSSZ2015}.
To make our work self-contained, we reproduce  the brief argument in \cite[Theorem~3.12]{BBSSZ2015}.   
Let $P\in \SET(\alpha), P\ne S^{col}_\alpha$; we claim that $a_P= 0.$ Let $j$ be the integer guaranteed by Lemma~\ref{lem:Z-indecomplemma} for this $P,$ and apply $\pi_j$ to the preceding equation.  We then have, since $\pi_j$ fixes $S^{col}_\alpha,$ 
\[f(S^{col}_\alpha)=f(\pi_j(S^{col}_\alpha))=\sum_{T\in \SET(\alpha)} a_T \pi_j(T),\]
but $P\ne \pi_j(P),$  so $P$ does not appear in the right-hand side by Lemma~\ref{lem:indecomplemma0}. Hence $P$ does not appear in the expansion of $f(S^{col}_\alpha),$ i.e. $a_P=0.$  
We conclude  
 that $f(S^{col}_\alpha)$ equals a scalar multiple of $S^{col}_\alpha,$ which can only be 0 or 1 since $f$ is idempotent.   \end{proof}

Next we examine   the quotient module ${\mathcal{V}_\alpha/\mathcal{Z}_\alpha}$ arising from the submodule $\mathcal{Z}_\alpha$  of $\mathcal{V}_\alpha.$ 

\begin{definition}\label{def:NSET} Let $\mathrm{NSET}(\alpha)$ be the set of all standard  tableaux of shape $\alpha$ in which all  rows increase left to right, but at least one \textbf{column} does NOT increase from bottom to top. 
\end{definition}

\begin{lemma}\label{lem:rdIquotientmodule} The quotient module $\mathcal{\bar V}_{\alpha}:={\mathcal{V}_\alpha/\mathcal{Z}_\alpha}$ is also an 
 $\hn$-module  for the $\rdI$-action, with basis of cosets whose  representatives constitute the set $\NSET(\alpha)\cap \SIT(\alpha).$  The module $\mathcal{\bar V}_{\alpha}$ is nonzero if and only if $\alpha$ has at least two parts of size greater than or equal to 2. When it is nonzero, it is cyclically generated by (the coset represented by) $S^0_\alpha.$
\end{lemma}

\begin{proof} The first statement is  clear, since a basis for the quotient module is the complement of $\SET(\alpha)$ in $\SIT(\alpha)$, which is empty if $\alpha$ has at most one part of size greater than or equal 2; recall that the first column is always increasing by definition of $\SIT(\alpha).$  The second statement follows since $S^0_\alpha$ generates all of $\mathcal{V}_\alpha.$
\end{proof}

It is helpful to record the extreme cases of $\mathcal{\bar V}_\alpha.$ 
\begin{lemma}\label{lem:extremes} Let $\alpha\vDash n.$  
\begin{enumerate}
\item  If $\alpha$ has at most one part greater than 1, then 
$\mathcal{\bar V}_\alpha=(0).$
\item If $\alpha$ has exactly two parts greater than 1, and both of these have size 2, then 
$\mathcal{\bar V}_\alpha$ is one-dimensional, and hence irreducible and  indecomposable. 
\end{enumerate}
\end{lemma}

\begin{proof} Item (1) is clear since in this case any $T\in \SIT(\alpha)$ has only one column of length greater than 1, the first, and so all its columns are increasing.  Hence $\NSET(\alpha)\cap\SIT(\alpha)=\emptyset.$

For Item (2), column 2 of any $T\in \SIT(\alpha)$ has only two cells, and hence can  decrease bottom to top in only one way. 
\end{proof}

For the remainder of this section we will assume $\NSET(\alpha)\cap\SIT(\alpha)\ne \emptyset.$  
Equivalently, by Lemma~\ref{lem:extremes}, we assume $\alpha$ has at least two parts greater  than or equal to 2.

 We define a relation on $\NSET(\alpha)\cap\SIT(\alpha)$ by setting, for $S, T \in \NSET(\alpha)\cap\SIT(\alpha),$ $S\preccurlyeq^{\NSET}_{\rdI} T$ if there is a sequence $\hat\pi$ of generators $\pi_{i_1}\cdots \pi_{i_r}$ such that $T=\hat\pi^{\rdI}(S).$   This is simply the relation induced on the subposet   of $P\rdI_\alpha$ consisting of the elements in $\NSET(\alpha)\cap\SIT(\alpha).$  
For any $S,T\in  \NSET(\alpha)\cap\SIT(\alpha),$  since $\SET(\alpha)$ is invariant under the $\hn$-action, the intervals $[S,T]$ in $\SIT(\alpha)$ and $\NSET(\alpha)\cap\SIT(\alpha)$ coincide, and hence the induced subposet $\NSET(\alpha)\cap\SIT(\alpha)$ is also ranked.  However, it has no top element since $ S^{row}_\alpha\notin \NSET(\alpha).$  See the Hasse diagram of $P\rdI_\alpha$ for $\alpha=223,$  Figure~\ref{fig:Poset}.

We therefore immediately have 
\begin{lemma}\label{lem:NSET-partial-order} The relation $\preccurlyeq^{\NSET}_{\rdI}$ is a partial order on $\NSET(\alpha)\cap\SIT(\alpha),$ with minimal element $S^0_\alpha.$
\end{lemma}

Given a composition $\alpha \vDash n,$ extend the partial order $\preccurlyeq^{\NSET}_{\rdI}$ on  $\NSET(\alpha)\cap\SIT(\alpha)$ to an arbitrary total order on  $\NSET(\alpha)\cap\SIT(\alpha),$ denoted by $\preccurlyeq^{\NSET^t}_{\rdI}$. Let the elements of $\NSET(\alpha)\cap\SIT (\alpha)$ under $\preccurlyeq^{\NSET^t}_{\rdI}$ be 
$$\{S^0_\alpha=\rtau _1\preccurlyeq^{\NSET^t}_{\rdI} \cdots \preccurlyeq^{\NSET^t}_{\rdI}\rtau _m  \}.$$
 Now let $\mathcal{\bar V}_{\rtau_i}$ be the $\mathbb{C}$-linear span 
 $$\mathcal{\bar V}_{\rtau_i} = \spam \{ \rtau _j \suchthat \rtau _i \poRI^{\NSET^t} \rtau _j \}\quad \text{ for } 1\leq i \leq m$$
and observe that the definition of $\preccurlyeq^{\NSET^t}_{\rdI}$ implies that $\pi^{\rdI}_{i_1}\cdots \pi^{\rdI} _{i_\ell}\mathcal{\bar V}_{\rtau_i}\subseteq \mathcal{\bar V}_{\rtau_i}$ for any $s_{i_1},\ldots, s _{i_\ell}\in S_n$. This observation combined with the fact that the operators $\{\pi^{\rdI}_i\}_{i=1}^{n-1}$ satisfy the same relations as $\hn$ by Theorem~\ref{thm:Hecke-action-RSImm} gives the following result.
\begin{lemma}\label{lem:hnquotientmodule}
$\mathcal{\bar V}_{\rtau_i}$ is an $\hn$-module, and gives us a filtration of $\hn$-modules 
$$
0\subset \mathcal{\bar V}_{\rtau_m} \subset \cdots \subset \mathcal{\bar V}_{\rtau_{2}}\subset \mathcal{\bar V}_{\rtau_1}.
$$
\end{lemma}

Set $\mathcal{\bar V}_{\alpha}:= \mathcal{\bar V}_{\rtau_1}.$

Essentially the same arguments  leading to  Theorem~\ref{the:bigone} now give us:

\begin{theorem}\label{thm:another-bigone}
Let $\alpha \vDash n$ be such that $\alpha$ has at least two parts of size greater than or equal to 2, and let $\rtau _1=S^0_\alpha \in \SIT (\alpha)$ be the unique minimal element of $\NSET(\alpha)\cap\SIT(\alpha) $. Then $\mathcal{\bar V}_{\alpha}$ is an $\hn$-module, cyclically generated by $S^0_\alpha,$ whose quasisymmetric characteristic is the quasisymmetric function 
\begin{equation}\label{eq:new-ext-Schur}\overline{\mathcal{R}\mathcal{E}}_\alpha= \sum_{T\in \NSET(\alpha)\cap\SIT(\alpha)} F_{\mathrm{comp}(\Des_{\rdI}(T))}.\end{equation} 
Equivalently,
\begin{equation}\label{eqn:rdI-ext-diff} \rdI_\alpha-\overline{\mathcal{R}\mathcal{E}}_\alpha=\sum_{\rtau \in \SET(\alpha)}F_{\comp(\desri(\rtau))}.\end{equation}
\end{theorem}
\begin{proof} In Lemma~\ref{lem:hnquotientmodule}, letting $\mathcal{\bar V}_{\rtau_{m+1}}=0,$ the quotient modules $\mathcal{\bar V}_{\rtau_{i-1}}/\mathcal{\bar V}_{\rtau_{i}}$ for $2\leq i\leq m+1$ are $1$-dimensional $\hn$-modules spanned by $\rtau_{i-1}$.
Since
\begin{eqnarray*}
\pi_{j}(\rtau_{i-1})&=& \left \lbrace \begin{array}{ll}0, & j\in \desri(\rtau_{i-1}),\\\rtau_{i-1}, & \text{otherwise,}\end{array}\right.
\end{eqnarray*}
 as an $\hn$-module, $\mathcal{\bar V}_{\rtau_{i-1}}/\mathcal{\bar V}_{\rtau_{i}}$ is isomorphic to the irreducible module $L_{\beta}$, where $\beta$ is the composition corresponding to the descent set $\desri(\rtau_{i-1})$. 
Hence 
$\mathrm{ch}(\mathcal{\bar V}_{\rtau_{i-1}}/\mathcal{\bar V}_{\rtau_{i}})=F_{\comp(\desri(\rtau_{i-1}))}$
and 
\begin{eqnarray*}
\mathrm{ch}(\mathcal{\bar V}_{\alpha})&=\displaystyle\sum_{i=2}^{m+1}\mathrm{ch}(\mathcal{\bar V}_{\rtau_{i-1}}/\mathcal{\bar V}_{\rtau_{i}})
= \displaystyle \sum_{i=2}^{m+1}F_{\comp(\desri(\rtau_{i-1}))}\nonumber\\ 
&= \displaystyle \sum_{\rtau \in \NSET(\alpha)\cap\SIT(\alpha)}F_{\comp(\desri(\rtau))}
= \overline{\mathcal{R}\mathcal{E}}_\alpha,
\end{eqnarray*}
as claimed.  Equation~\eqref{eqn:rdI-ext-diff} follows from the fact that 
\[\rdI_\alpha= \sum_{\rtau \in \SIT(\alpha)}F_{\comp(\desri(\rtau))}.\] The statement about the cyclic generator was established in Lemma~\ref{lem:rdIquotientmodule}. \end{proof}

We now  show that the module $\mathcal{\bar V}_{\alpha}$ is indecomposable.  This will require a  careful analysis of the technical lemmas in Section~\ref{sec:indecomp-module-and-poset}. In particular,  we need to show that the analogue of Lemma~\ref{lem:indecomplemma4} holds.

\begin{lemma}\label{lem:indecomplemma*}  Let $P\ne S^0_\alpha,$ such that $P\in \NSET(\alpha)\cap\SIT^*(\alpha).$   Assume $\alpha$ has at least three parts of size greater than or equal to 2. 
Then there  is a sequence $\hat\pi$ of generators $\pi_{i_1}\cdots \pi_{i_r}$ such that 
\begin{enumerate}
\item  $\hat\pi(S^0_\alpha)=0;$  
\item $\hat\pi(P)$ is nonzero and has rank $r+\mathrm{rank}(P)$ in the subposet $P\rdI_\alpha\cap\NSET(\alpha)\cap\SIT^*(\alpha).$   
\end{enumerate}
\end{lemma}
\begin{proof} We follow the argument of Lemma~\ref{lem:indecomplemma4}: again, we may assume that whenever $j,j+1$ are adjacent in $S^0_\alpha,$ they are also adjacent in $P.$

Now  we  need the analogues of Lemma~\ref{lem:indecomplemma4-a},  Lemma~\ref{lem:indecomplemma4-b} and Lemma~\ref{lem:indecomplemma4-c}.
It is enough to show that each of these lemmas holds when $P\in \NSET(\alpha)\cap\SIT^*(\alpha).$  

In that case, note first that if column 2 of $P$ increases bottom to top, then there must be a non-increasing column, column $j,$ for $j\ge 3.$  The important observation is that in each of the three lemmas, Lemma~\ref{lem:indecomplemma4-a},  Lemma~\ref{lem:indecomplemma4-b} and Lemma~\ref{lem:indecomplemma4-c}, the algorithms described in the proofs affect only the first two columns of the tableau $P.$ Hence we need only address the case when column 2 of $P$  (is the only column which) does NOT increase from bottom to top.

Consider Lemma~\ref{lem:indecomplemma4-a} and the algorithm which produces the saturated chain in the statement. It is clear that if $\ell+1$ is right-adjacent to $i\le \ell-1$ in $P$, then the algorithm affects only columns 1 and 2, and furthermore in column 2 it changes only the entry in the cell $(i,2),$ diminishing it each time by 1. Clearly if $i\ge 2,$ all entries below row $i$ in column 2 are greater than $\ell+1,$ and hence the algorithm  preserves the fact that column 2 is NOT increasing bottom to top.  If $i=1,$ then in column 2 above the first row there 
must be entries $x<y$ such that $x$ is in a higher row than $y.$ But these entries are untouched by the algorithm, so the intermediate tableaux in the saturated chain are all in $\NSET(\alpha).$ 

The analogous analysis for Lemma~\ref{lem:indecomplemma4-b} brings us to the same conclusion. 

Finally we examine Lemma~\ref{lem:indecomplemma4-c}. This is somewhat more intricate. 
Note that we have assumed $\alpha$ has at least three parts greater than 1. It is also clear from the proof of  Lemma~\ref{lem:indecomplemma4-c} that once again, the algorithm only affects the first two columns.
As before, we assume column 2 of $P$ is NOT increasing.   

Let $p$ be the smallest entry in $S^0_\alpha$ where there is disagreement with $P,$  and say it is in cell $(i,2).$ Then $P$ has an entry $z$, $z>p$ in this cell, $(i,2)$, and $p$ appears in $P$ in cell $(j,2)$ for $j<i.$ Again as in the proof of  Lemma~\ref{lem:indecomplemma4-c}, we may assume $p\ge\ell+2.$  The algorithm of  Lemma~\ref{lem:indecomplemma4-c} successively replaces $P$ with $\pi_j(P)$ for some $j$.  Our claim is that the second column of $\pi_j(P)$ continues to be non-increasing, i.e. 
at each step, $\pi_j(P)\in \NSET(\alpha).$

Consulting Example~\ref{ex:indecomplemma4-e}, we note that all rows above row $i$ coincide in both $S^0_\alpha$ and in $P,$ and  that if the second column of $P$ is non-decreasing, clearly the same is true of $\pi_{p-1}(P).$   At each step the entry in cell $(j,2)$ of $P$ is diminished by 1, and hence continues to be smaller than the entry $z$ in cell $(i,2),$ which is unchanged.  Also, at each step, entries in $P$ above cell $(i,2)$ continue to be smaller than $p.$ It follows that column 2 is always non-increasing.

Our argument is now complete.
\end{proof}

To finish the proof that $\mathcal{\bar V}_\alpha$ is indecomposable, we note first that Lemma~\ref{lem:injectivity} still holds, as does the following  analogue of Lemma~\ref{lem:indecomplemma1}:

\begin{lemma}\label{lem:indecomplemma1*} Suppose $f$ is an idempotent  endomorphism of the $\hn$-module $\mathcal{\bar V}_\alpha$, with 
\[f(S^0_\alpha)=\sum_{T\in \NSET(\alpha)\cap\SIT(\alpha)} a_T T. \] 
Let $P\in \NSET(\alpha)\cap\SIT(\alpha), P\ne S^0_\alpha,$ and 
suppose there is a $j$ such that $j\in \mathrm{Des}_{\rdI}(P)\setminus\Des_{\rdI}(S^0_\alpha).$
Then $a_P=0.$ 

In particular,  we have 
\[f(S^0_\alpha)=\sum_{T\in \NSET(\alpha)\cap\SIT^*(\alpha)} a_T T.\]
\end{lemma}

\begin{proof} The proof of Lemma~\ref{lem:indecomplemma1} goes through verbatim. \end{proof}

It can now be verified that the argument following Lemma~\ref{lem:injectivity} carries through unchanged,  thanks to Lemma~\ref{lem:indecomplemma*}.  Combined with Lemma~\ref{lem:rdIquotientmodule}, this gives us the following:

\begin{theorem}\label{thm:rdI-ext-quotient-indecomp} The  $\hn$-module $\mathcal{\bar V}_\alpha$ 
is nonzero and indecomposable if $\alpha$ has at least two parts greater than or equal to 2, and is zero otherwise.
\end{theorem}

\begin{remark}\label{rem:TODO}
The functions $\overline{\mathcal{R}\mathcal{E}}_\alpha$ cannot form a basis for $\QSym,$ since they equal zero when $\alpha\vDash n$ has at most one part greater than 1.  (Such compositions are called \textit{diving boards} in \cite{CFLSX2014}.) 
In fact the set $\{\overline{\mathcal{R}\mathcal{E}}_\alpha :\alpha \text{ has at least two parts of size $\ge 2$}\}$
need not be linearly independent.
For instance, $\overline{\mathcal{R}\mathcal{E}}_{122}=F_{311}=\overline{\mathcal{R}\mathcal{E}}_{212},$ because 
\\
$\NSET(\!122)\cap\SIT(\!122)$ has only one tableau ${\tableau{\bf{3}&\bf{4}\\2&5\\1}} ,$ 
as does $\NSET(212)\cap\SIT(212)$:$\tableau{\bf{3}&\bf{4}\\2\\1 &5}.$ 

\end{remark}

Next we observe that  each of the quasisymmetric functions of this section 
 is in fact the generating function for an appropriate  class of tableaux.

Call a row-strict immaculate tableau $T$ a \emph{row-strict semistandard immaculate tableau} if ALL the columns of $T$ are weakly increasing, bottom to top (and necessarily all the rows of $T$ are strictly increasing, left to right).   Similarly, call an immaculate tableau $S$ a \emph{column-strict semistandard immaculate tableau} if ALL the columns of $S$ are strictly increasing, bottom to top, (and necessarily all the rows of $S$ are weakly increasing, left to right).

Recall (Definition~\ref{def:rdIfunction}) that $\rdI_\alpha$ is the generating function for all tableaux of shape $\alpha$ with weakly increasing first column, bottom to top, and strict increase along rows, left to right, while $\dI_\alpha$ (Definition~\ref{def:dIfunction}) is the generating function for all tableaux of shape $\alpha$ with strictly increasing first column, bottom to top, and weak increase along rows, left to right.
\begin{proposition}\label{prop:gf-row-ext} Let $\alpha\vDash n$.  For any tableau $D$ of shape $\alpha$, as in Definition~\ref{def:dIfunction}, let $d_i$ be the number of entries equal to $i$ in $D$. 
\begin{enumerate}
\item $\mathcal{R}\mathcal{E}_\alpha=\sum x_1^{d_1}x_2^{d_2}\cdots$ where the sum is over all row-strict semistandard immaculate tableaux $D$  of shape $\alpha$.  In particular when $\alpha$ is a partition $\lambda$, $\mathcal{R}\mathcal{E}_\lambda=s_{\lambda^t}$.
\item  $\overline{\mathcal{R}\mathcal{E}}_\alpha=\sum x_1^{d_1}x_2^{d_2}\cdots$ where the sum is over all row-strict immaculate tableaux of shape $\alpha$ with at least one column that is NOT increasing. 
\item \cite[Definition~8, Theorem~11, and Section~7]{CFLSX2014} $\mathcal{E}_\alpha=\sum x_1^{d_1}x_2^{d_2}\cdots$ where the sum is over all  column-strict semistandard immaculate tableaux $D$  of shape $\alpha$.  In particular when $\alpha$ is a partition $\lambda$, $\mathcal{E}_\lambda=s_{\lambda}$.
\end{enumerate}
\end{proposition}
\begin{proof}  Part (2) is immediate from Part (1) and the fact that $\rdI_\alpha-\mathcal{R}\mathcal{E}_\alpha=\overline{\mathcal{R}{\mathcal{E}}}_{\alpha}$.
Part (3) is due to \cite{CFLSX2014}, where the authors define the shin basis $\shin_\alpha$ of $\Nsym$; their notation for our 
 $\mathcal{E}_\alpha$ is $\shin^*_\alpha$.  Their proof arrives at our definition of column-strict semistandard immaculate tableaux using Pieri rules.  However, one can also adapt the following  proof of Part (1).

Our proof follows \cite[Ex.~7.90(a)]{RPSEC21999}. From Proposition~\ref{prop:rowstrict-ext-basis} we have 
\[\mathcal{R}\mathcal{E}_\alpha=\sum_{T\in\SET(\alpha)} F_{\comp(\Des_{\rdI}(T))}  .\] 
It suffices to show that for any $\beta\vDash n$, the coefficient of the monomial $x_1^{\beta_1}x_2^{\beta_2}\cdots$  in the right-hand side is the number of row-strict semistandard immaculate tableaux of shape $\alpha$ and content $\beta$, i.e. with $\beta_i$ $i$'s.
To see this, we recall \cite{LMvW2013} that for a composition $\beta$, the  monomial $x_1^{\beta_1}x_2^{\beta_2}\cdots$ appears in $F_\gamma$ if and only if $\beta$ is a finer composition than $\gamma$, or equivalently, if and only if $\set(\gamma)\subseteq \set(\beta)$.  Hence the monomial $x_1^{\beta_1}x_2^{\beta_2}\cdots$ appears in $\mathcal{R}\mathcal{E}_\alpha$  if and only if $\Des_{\rdI}(T)\subseteq \set(\beta).$

We claim that 
$|\{T\in \SET(\alpha): \Des_{\rdI}(T)\subseteq \set(\beta)\}|$
equals the number of  row-strict semistandard immaculate tableaux of shape $\alpha$ with $\beta_i$ entries equal to $i$, for all $i\ge 1$.

By definition, in a   row-strict semistandard immaculate tableau $\tau$, the entries $1,2,\ldots,i$ appear in the first $i$ columns of $\tau$ from the left. 

Fix a composition $\beta\vDash n$.
Given such a tableau $\tau$ with $\beta_i$ $i$'s,  we associate to $\tau$ a tableau $T\in \SET(\alpha)$ obtained by replacing the $\beta_i$ entries equal to $i$ with the $\beta_i$ consecutive entries in the interval $[\beta_1+\cdots+\beta_{i-1}, \beta_1+\cdots+\beta_{i-1}+\beta_i],$ where we read all the entries equal to $i$ from \textit{bottom to top}.   Since the columns increase weakly and the rows increase strictly, no two $i$'s are in the same row, so this implies that the descents in $T$ can occur only where $i$'s change to $(i+1)$'s as we read up the columns of $\tau$, starting with the leftmost column.  Hence  
$\Des_{\rdI}(T)\subseteq \set(\beta)$.

Conversely, suppose we have $T\in \SET(\alpha)$ such that $\Des_{\rdI}(T)\subseteq \set(\beta)$.  To recover $\tau$, replace $1,2,\ldots ,\beta_1$ with $1$'s, $\beta_1+1,\ldots ,\beta_1+\beta_2$ with 2's and so on.  The condition $\Des_{\rdI}(T)\subseteq \set(\beta)$ guarantees that for all $i$, the first $\beta_i$ entries appear in the first $i$ columns of $\tau$, and hence we have a tableau  with weakly increasing columns, and strictly increasing rows.  This finishes the proof.
\end{proof}
 In his paper, Searles shows that the set $\NSET(\alpha)$ constitutes an $\hn$-invariant 
  submodule, whose basis is the set of all row-increasing tableaux of shape $\alpha$, of a larger parent module for the $\dI$-action; he then shows that the resulting \textit{quotient} module has the following properties: 
it has basis $\SET(\alpha)$ and 
 quasiymmetric characteristic $\mathcal{E}_\alpha$, specialising to the Schur function $s_\lambda$ when $\alpha$ is the partition $\lambda$, and 
it is indecomposable and cyclically generated by the top element $ S^{row}_\alpha$ (the row superstandard tableau) of the poset $P\dI_\alpha$. 
 Also, the following difference expands positively in the fundamental basis for $\QSym$:
\[ \dI_\alpha-\mathcal{E}_\alpha=\sum_{T\in \NSET(\alpha)\cap\SIT(\alpha)}  F_{\comp(\Des_{\dI}(T))}.  \]
By contrast, in this section we have shown that  $\SET(\alpha)$ constitutes an $\hn$-invariant subposet of $P\rdI_\alpha$ for the $\rdI$-action, and therefore spans an $\hn$-submodule $\mathcal{Z}_\alpha$ of $\mathcal{V}_\alpha$  with the following properties: 

$\mathcal{Z}_\alpha$ has basis $\SET(\alpha)$
 and quasisymmetric characteristic $\mathcal{R}\mathcal{E}_\alpha=\psi(\mathcal{E}_\alpha)$, 
specialising to the Schur function $s_{\lambda^t}$ when $\alpha$ is the partition $\lambda$, and 
 it is indecomposable and cyclically generated by the column superstandard element  $S^{col}_\alpha$ of the poset $P\rdI_\alpha$, which is also the bottom element of the subposet $\SET(\alpha)$. Furthermore, 
 the resulting quotient module $\mathcal{V}_\alpha/\mathcal{Z}_\alpha,$ with basis $\NSET(\alpha)\cap\SIT(\alpha)$ and characteristic 
$\overline{\mathcal{R}\mathcal{E}}_\alpha,$ is cyclically generated by $S^0_\alpha$, and is indecomposable if $\alpha$ has at least two parts greater than or equal to 2, and is zero otherwise. Finally, 
the following difference expands positively in the fundamental basis for $\QSym$:
\[\rdI_\alpha-\mathcal{R}\mathcal{E}_\alpha=\overline{\mathcal{R}{\mathcal{E}}}_{\alpha}= \sum_{T\in \NSET(\alpha)\cap\SIT(\alpha)} F_{\mathrm{comp}(\Des_{\rdI}(T))}.\]

We now observe that, similarly, the poset $P\rdI(\alpha)$  gives a quotient module of  $\mathcal{W}_\alpha$ with characteristic equal to the extended Schur function $\mathcal{E}_\alpha$. In  Figure~\ref{fig:Poset}, by reversing the arrows, one sees that  the tableaux that are NOT in $\SET(\alpha)$ form a 
  closed subset under the $\dI$-action.
\begin{proposition}\label{prop:shin-module} The set $\NSET(\alpha)\cap\SIT(\alpha)$ is a basis for a submodule $\mathcal{Y}_\alpha$ of $\mathcal{W}_\alpha$ for the $\dI$-action.   The resulting quotient module 
$\mathcal{W}_\alpha/\mathcal{Y}_\alpha$ has basis of cosets represented by the set $\SET(\alpha)$, with characteristic $\mathcal{E}_\alpha$, and is cyclically generated by (the coset represented by) $S^{row}_\alpha$.  It  is indecomposable.
\end{proposition}
\begin{proof}
 Lemma~\ref{lem:rdIsubmodule} shows that $\NSET(\alpha)\cap\SIT(\alpha)$ spans a submodule $\mathcal{Y}_\alpha$ of $\mathcal{W}_\alpha$ for the $\dI$-action, since, by Lemma~\ref{lem:sameposet}, for $T\in \SIT(\alpha)$, $\pi_i^{\dI}(T)=S\in\SIT(\alpha)\iff T=\pi_i^{\rdI}(S)$. Clearly the resulting quotient module has basis 
 indexed by $\SET(\alpha)$, characteristic $\mathcal{E}_\alpha$, and is cyclically generated by $S^{row}_\alpha$.  It remains to check that this quotient module is indecomposable.   But this follows by invoking \cite[Lemma~3.11]{BBSSZ2014}, which states that for $P\ne  S^{row}_\alpha$ there is an $i$ such that 
$\pi_i^{\dI}$ fixes the generator $S^{row}_\alpha$ but not $P$. The argument is now identical to \cite[Theorem~3.12]{BBSSZ2014} and the proof of Theorem~\ref{thm:Z-indecomp}.
\end{proof}

\section{New 0-Hecke modules from the dual immaculate action}\label{sec:dualImmsubmodules}

In this section we re-examine the dual immaculate action $\dI$ described in  Theorem~\ref{thm:BBSSZ2dualImm}, originally defined in \cite{BBSSZ2015}, on the vector space with basis $\SIT(\alpha)$.  Recall from Section~\ref{sec:indecomp-module-and-poset} that $\SIT^*(\alpha)$ 
 is the set of standard immaculate tableaux whose first column consists of the integers in $[\ell(\alpha)]$. In Figure~\ref{fig:Poset}, these are the blue tableaux lying in the interval $[S^0_\alpha, S^{col}_\alpha]$.  

In \cite{S2020}, Searles constructs the $\hn$-module (for the dual immaculate action) whose characteristic is the extended Schur function by taking a quotient of a larger module, different from $\mathcal{W}_\alpha.$   
We showed in Proposition~\ref{prop:shin-module} how this $\hn$-module can also be obtained  by taking a quotient of $\mathcal{W}_\alpha.$
In this section we will  consider a different submodule, and resulting quotient module of $\mathcal{W}_\alpha$ itself, for the dual immaculate action.  This continues the analogy with  the previous section, where we constructed both a submodule and quotient module of $\mathcal{V}_\alpha$ for the row-strict dual immaculate action.  

Note that thanks to Lemma~\ref{lem:sameposet}, the $\dI$-Hecke action on the poset $P\dI(\alpha)$ is obtained from the $\rdI$-Hecke action on the poset $P\rdI(\alpha)$ simply by reversing the arrows in the Hasse diagram. In particular the labels remain unchanged.

Before stating the theorem, we extract the two main techniques used in this paper for establishing indecomposability of a cyclic $\hn$-module $\mathcal{H}$.  The technique of Part (1) below was originally used to prove indecomposability in \cite[Lemma~7.7, Theorem~7.8]{TvW2015}, for the  immaculate $\hn$-module $\mathcal{W}_\alpha$ in \cite[Lemma 3.11, Theorem 3.12]{BBSSZ2015}, as well as the extended immaculate module of \cite[Theorem~3.13]{S2020}; in the present paper it is applied to the row-strict extended module $\mathcal{Z}_\alpha$  in Lemma~\ref{lem:Z-indecomplemma} and Theorem~\ref{thm:Z-indecomp} (in the special case where $\sigma_v$ is a consecutive transposition).  The technique of Part (2) is used in 
\cite[Proof of Theorem~4.1]{BS2021} and in this paper for the row-strict Hecke module $\mathcal{V}_\alpha$, in the proofs of Theorem~\ref{thm:Indecomp}, as well as in Lemma~\ref{lem:indecomplemma1*} and Theorem~\ref{thm:rdI-ext-quotient-indecomp}.

\begin{proposition}\label{prop:indecomp-proof-methods} 
Let $\mathcal{H}$ be an $\hn$-module cyclically generated by $u_0,$ with $\mathbf{K}$-basis $\{v\}_{v\in B}$ for some finite set $B$ such that $u_0\in B$.  Let $f$ be an  endomorphism of $\mathcal{H}.$ Either of the following two conditions implies that $f$ is a scalar multiple of the identity.  In particular, if $f$ is idempotent, then it must be either the identity or the zero map.

\begin{enumerate}
\item For every basis element $v\in B$, $v\ne u_0$, there is a sequence of  generators   $\hat\pi_v=\pi_{i_1}\cdots\pi_{i_r}$ in $\hn$ such that 
 \[\hat\pi_{v}(u_0)=u_0, \qquad \text{but } \hat\pi_{v}(v)\ne v.\]
Here we also need to know that $v\ne\hat\pi_{v}(v)\Rightarrow  v\ne \hat\pi_{v}( w) $ for any other basis element $w$. The latter statement holds automatically when the sequence of generators  $\hat\pi_v$ is a single generator $\pi_i$, see Lemma~\ref{lem:indecomplemma0}. 
\item For every basis element $v\in B$, $v\ne u_0$, there is a sequence of generators   $\hat\pi_v=\pi_{i_1}\cdots\pi_{i_r}$ in $\hn$ such that 
 \[\hat\pi_{v}(u_0)=0, \qquad  \text{but } \hat \pi_{v}(v)\ne 0.\]
Here we also need to know that $\hat\pi_(v)=\hat\pi(v')\Rightarrow v=v'$.
\end{enumerate}
\end{proposition}
\begin{definition}\label{def:top-elt-SIT*}  Define $S^{row*}_\alpha$ to be the standard immaculate tableau in $\SIT(\alpha)$ whose first column consists of $1,2,\ldots,\ell(\alpha)$, and whose remaining cells  are filled with the entries $\ell+1, \ldots, n$ in consecutive order along  rows, bottom to top and left to right.  Clearly  $S^{row*}_\alpha$ has strictly increasing columns bottom to top, and hence $S^{row*}_\alpha\in \SIT^*(\alpha)\cap \SET(\alpha)$.  Compare this with the definition of $S^{row}_\alpha$ in Definition~\ref{def:top-elt}.

For example, we have
 \[S^{row*}_{223}=S^{col}_{223}=\tableau{3 &6 &7\\2&5 \\1& 4}\!,\ 
S^{row*}_{332}=\tableau{3 &8 \\2 &6 &7 \\1 & 4 & 5}\ne S^{col}_{332},\  S^{row*}_{1323}
=\tableau{4 &8  &9 \\3 &7  \\2 & 5 & 6\\1}\ne S^{col}_{1323}.\]
\end{definition}

\begin{lemma}\label{lem:SIT*-interval} The set $\SIT^*(\alpha)$ is a bounded interval of the immaculate Hecke poset $\mathcal{P}\dI(\alpha)$, the interval $[S^0_\alpha, S^{row*}_\alpha]$.  Furthermore, for any $T\in \SIT^*(\alpha)$, 
there is a sequence $\hat\pi$ of generators  $\pi_{i_1}\cdots\pi_{i_r}$ such that, for the $\dI$-action, 
\[T=\hat\pi^{\dI}(S^{row*}_\alpha).\]
\end{lemma}
\begin{proof} Let $\ell=\ell(\alpha)$, the length of $\alpha$. The invariance of $\SIT^*(\alpha)$ under the $\dI$-action is clear from the fact (see Theorem~\ref{thm:BBSSZ2dualImm}) that 
\[\pi^{\dI}_i(T)=s_i(T)\notin \{ T,0\}\]
\[ \iff  i, i+1 \text{ are NOT in column 1 of $T$ and $i+1$ is strictly above $i$ in $T$},\] 
and hence the first column, consisting of $\{1,2,\ldots,\ell(\alpha)\}$,  is preserved under the action.  

Also $\Des_{\rdI}(S^{row*}_\alpha)=\{\ell\}$. Hence either $S^{row*}_\alpha=S^{0}_\alpha$, which occurs if $\alpha_\ell$ is the only part greater than 1, or there is exactly one $j$ such that $\pi_j^{\rdI}(S^{row*}_\alpha)=T\in\SIT(\alpha)$,  namely $j=\ell$. Moreover in that case $T\notin \SIT^*(\alpha)$.  This shows that for the $\dI$-action, if for any $j$, $\pi_j^{\dI}(T)=S^{row*}_\alpha  , T\in \SIT(\alpha)$, then necessarily $T \notin \SIT^*(\alpha)$. 

For clarity and consistency with the previous sections, we revert to the partial order defined in Section~\ref{sec:partial-order} and the poset $P\rdI(\alpha)$. The lemma will then follow upon invoking 
Lemma~\ref{lem:sameposet}.

It suffices to apply the straightening algorithm of Proposition~\ref{prop:top-elt} for the $\rdI$-action. 
Recall that the algorithm produces a saturated chain from any tableau $T\in\SIT(\alpha)$ to the top element, the row superstandard tableau $S^{row}_\alpha.$ 
We straighten a tableau $T\in \SIT^*(\alpha)$ to the tableau $S^{row*}_\alpha$ in exactly the same manner,  namely, by working on the rows from top to bottom, starting with the \textit{largest} entry of $T$ that differs from its counterpart in $S^{row}_\alpha$, and moving to the next largest entry. This ensures that the entries $\{1,2,\ldots,\ell(\alpha)\}$ are preserved in the first column. 
Exactly as in the proof of  Proposition~\ref{prop:top-elt}, it follows that if $T\in \SIT^*(\alpha)$, then there is a sequence of operators $\pi_{j_i}, $ and tableaux $T_i\in \SIT^*(\alpha),$ $i=1,\ldots, r,$ such that $\pi_{j_i}(T_{i-1})=T_{i}, i=1,2,\ldots ,r,$ where we set $T_{0}=T$ and $T_r=S^{row*}_\alpha.$ Hence we conclude 
\[ T \poRI S^{row*}_\alpha \text{ and } \pi_{j_r}\pi_{j_{r-1}}\cdots\pi_{j_1}(T)=S^{row*}_\alpha.\]
In particular $S^{row*}_\alpha$ is the unique maximal element of the subposet $\SIT^*(\alpha)$ of the poset $P\rdI_\alpha.$  
Thanks to Lemma~\ref{lem:sameposet}, by reversing the arrows, this establishes the result for the dual immaculate $\dI$-action as well.
\end{proof}

We illustrate the straightening with two examples.

\begin{example}\label{ex:SIT*-straightening-top-elt} Let $\alpha=431$ so that
\ $S^{row*}_\alpha=\tableau{3\\2 &7 & 8\\1 &4 & 5 & 6}. $
Indicating in bold the largest top-most entry that differs from that of $S^{row*}_\alpha$, we have, for $T=S^0_\alpha=\tableau{3\\2 &4 & \bf{5}\\1 &6 & 7 & 8}, $ 
 \[T_0\stackrel{\pi_5}{\longrightarrow} \tableau{3\\2 &4 &\bf{6} \\1  &5 & 7 & 8}\!\!=T_1
 \stackrel{\pi_6}{\longrightarrow}\tableau{3\\2 &4 &\bf{7} \\1  &5 & 6 & 8}\!\!=T_2
 \stackrel{\pi_7}{\longrightarrow} \tableau{3\\2 &\bf{4} &{8} \\1  &5 & 6 & 7}\!\!=T_3
 \stackrel{\pi_4}{\longrightarrow} \tableau{3\\2 &\bf{5} &{8} \\1  &4 & 6 & 7}\!\!\!=T_4,
 \]
\[\text{and } T_4\stackrel{\pi_5}{\longrightarrow} \tableau{3\\2 &\bf{6} &{8} \\1  &4 & 5 & 7}\!\!\!=T_5
\stackrel{\pi_6}{\longrightarrow} \tableau{3\\2 &\bf{7} &{8} \\1  &4 & 5 & 6}\!\!\!=T_6=S^{row*}_\alpha.\]
Similarly for $U=S^{col}_\alpha$ we have 
\[S^{col}_\alpha=\tableau{3\\2 & 5 & \bf{7}\\1 &4 &6 &8}\stackrel{\pi_7}{\longrightarrow} \tableau{3\\2 &\bf{5} &{8} \\1  &4 & 6 & 7}\!\!=U_1
\stackrel{\pi_5}{\longrightarrow} \tableau{3\\2 &\bf{6} &{8} \\1  &4 & 5 & 7}\!\!=U_2
\stackrel{\pi_6}{\longrightarrow} \tableau{3\\2 &{7} &{8} \\1  &4 & 5 & 6}\!\!=U_3=S^{row*}_\alpha.\]
\end{example}

\begin{theorem}\label{thm:dualImm-submodule-quotient} Let $\alpha\vDash n,$ and let $\mathcal{W}_\alpha$ be the $\hn$-module with basis $\SIT(\alpha)$ as in Theorem~\ref{thm:BBSSZ2dualImm}.  Then 
\begin{enumerate}
\item
$\SIT^*(\alpha)=[S^0_\alpha, S^{row*}_\alpha]$  is a basis for an  $\hn$-submodule $\mathcal{X}_\alpha$ of $\mathcal{W}_\alpha$ for the $\dI$-action, of dimension $\binom{n-\ell(\alpha)}{\alpha_1-1,\alpha_2-1,\ldots, \alpha_\ell-1}.$
\item $\mathcal{X}_\alpha$ is cyclically generated by $S^{row*}_\alpha$;  it has characteristic 
\[\chr(\mathcal{X}_\alpha)= \sum_{T\in \SIT^*(\alpha)} F_{\comp(\Des_{\dI}(T))}.\]
If $\alpha$ has at most one part greater than 1, then $\SIT^*(\alpha)$ has cardinality 1, and the module $\mathcal{X}_\alpha$ is the irreducible $\hn$-module indexed by the composition $\alpha.$
\item Assume $\alpha$ has at least two parts of size greater than 1. Then the quotient module $\mathcal{W}_\alpha/\mathcal{X}_\alpha$ is  nonzero and cyclically generated by $S^{row}_\alpha$, and furthermore it is indecomposable.  It has characteristic 
\[\chr(\mathcal{W}_\alpha/\mathcal{X}_\alpha)
= \sum_{T\in \SIT(\alpha)\setminus \SIT^*(\alpha)} F_{\comp(\Des_{\dI}(T))}.\]
If $\alpha$ has at most one part greater than 1, $\SIT(\alpha)=\SIT^*(\alpha)$ and the quotient module is zero.
\end{enumerate}
\end{theorem}

\begin{proof}[Proof of Part (1):]   
The first statement follows from Lemma~\ref{lem:SIT*-interval}.
The dimension count is clear, since, ignoring the constant first column, the tableau is uniquely determined by a sequence of subsets $S_1, \ldots, S_\ell,$ whose union is $[n]\setminus [\ell]$;  $S_i$ consists of the entries in  row $i$ excluding the first column, and thus $|S_i|=\alpha_i-1$. 

\textit{Proof of Part (2):}  Lemma~\ref{lem:SIT*-interval} gives the cyclic generator as well. 
 The characteristic follows from the now familiar argument:  impose a linear extension on the interval $[S^0_\alpha, S^{row*}_\alpha],$ inducing an $\hn$-equivariant filtration whose quotients are one-dimensional $\dI$-Hecke modules generated by $T\in [S^0_\alpha, S^{row*}_\alpha],$ and then sum up the individual quotient characteristics.

The last statement is clear, since $\SIT^*(\alpha)$ has cardinality 1 if and only if $\alpha$ has at most one part of size greater than 1.

\textit{Proof of Part (3):} As above, it is clear that the quotient module is nonzero if and only if $\alpha$ has at least two parts of size greater than 1.  The statement about the cyclic generator follows from the fact that the parent module $\mathcal{W}_\alpha$ is itself cyclically generated by the row superstandard tableau $S^{row}_\alpha.$   A basis for the quotient is given by the complement of the  set $\SIT^*(\alpha)$ in $\SIT(\alpha),$ and the characteristic is computed via the usual filtration arising from a linear extension.  

In order to prove indecomposability, we appeal to the method in Part (1) of Proposition~\ref{prop:indecomp-proof-methods}.   
 Following the proof of Theorem~\ref{thm:Z-indecomp}, for the quotient module our starting point is the equation 
\[f(S^{row}_\alpha)=\sum_{T\in\SIT(\alpha)\setminus\SIT^*(\alpha)} a_T T.\]
We need the analogue of Lemma~\ref{lem:Z-indecomplemma}, but this is precisely the content  of \cite[Lemma 3.11]{BBSSZ2015}, namely that for every tableau $P\ne S^{row}_\alpha,$ there is a $j\in [n-1]$ such that $\pi_j$ fixes $S^{row}_\alpha$ but does NOT fix $P.$ The rest of the proof is identical to the proof of Theorem~\ref{thm:Z-indecomp}.
\end{proof}

The functions $\chr(\mathcal{X}_\alpha)$ are not independent, as the following simple counterexample shows:
\begin{example}\label{ex:Steph-counterex} Let $n=3$ and consider the compositions 
$\alpha=12, \beta =21$.  Then $\SIT^*(\alpha)=\left\{\tableau{2 &3\\1}\right\},\SIT^*(\beta)= \left\{\tableau{2\\1 &3}\right\}.$  Clearly 
$\chr(\mathcal{X}_\alpha)=\chr(\mathcal{X}_\beta),$ since  both tableaux have the same $\dI$-descent set, namely the set $\{1\}.$
\end{example}

We have the following pleasing expression for the characteristic  of the module $\mathcal{X}_\alpha.$  
\begin{proposition}\label{prop:ElizabethSIT*-char} Let $\alpha\vDash n$ have length $\ell$, and let $\bar\alpha$ be the composition $(\alpha_1-1, \alpha_2-1, \ldots, \alpha_\ell-1)$ of $n-\ell,$ where we omit any part that is zero.  Then \[\chr(\mathcal{X}_\alpha)= 
\sum_{k\ge \ell} e_{\ell-1}(x_1,\ldots,x_{k-1}) x_k h_{\bar{\alpha}}(x_k,x_{k+1},\ldots ) ,\]
where $e_r$ is the $r$th elementary symmetric function, and $h_\beta=h_{\beta_1}h_{\beta_2}\cdots $ is the product of the homogeneous symmetric functions indexed by the parts of the composition $\beta.$
\end{proposition}

\begin{proof} 
\svwtwo{Recall from Theorem~\ref{thm:AS19} that for a partition $\lambda$ we have the Schur function expansion 
$$s_\lambda = \sum _{T\in \SET (\lambda)} F_{\comp (\DesI (T))}$$and this naturally generalizes (see \cite[Theorem~7.19.7]{RPSEC21999}) for partitions $\lambda, \mu$  to the skew Schur function expansion 
$$s_{\lambda/\mu} = \sum _{T\in \SET (\lambda/\mu)} F_{\comp (\DesI (T))}$$where $\lambda/\mu$ is the diagram of $\lambda$ with that of $\mu$ removed from the bottom left corner. See \cite{gessel} for further details.}

\svwtwo{Now note that the first column of every tableau $T\in \SIT ^* (\alpha)$ is in natural bijection with the unique tableau in $\SET(1^\ell)$.
 Also note that $T$ with the leftmost column removed and every remaining tableau entry (that is at least $\ell +1$) reduced by $\ell$ is in natural bijection with a standard extended tableau denoted by $D_T$,  whose shape consists of disconnected rows of length $\alpha _1 -1, \alpha _2 -1, \ldots , \alpha _\ell -1$ from bottom to top. Moreover, 
$$\{ i-\ell \suchthat i\in \DesI (T), i>\ell \} = \{ i \suchthat i\in \DesI (D_T) \}.$$ Furthermore, we have that $\ell -1 \in \DesI (T)$ but $\ell \not \in \DesI (T)$, so in the monomials appearing in $\chr(\mathcal{X}_\alpha)$ we have that the $(\ell-1)$th variable does not equal the $\ell$th variable, but the $\ell$th variable may equal the $(\ell+1)$th variable. Hence if the $\ell$th variable is $x_k$, then comparing the expressions for $\chr(\mathcal{X}_\alpha)$ and $s_{\lambda /\mu}$, the first column of $T$ contributes
$$s_{(1^{\ell -1})}(x_1,\ldots,x_{k-1})x_k=e_{\ell-1}(x_1,\ldots,x_{k-1}) x_k$$and the remainder of the shape contributes
$$s_{D_T}(x_k,x_{k+1},\ldots )= h_{\bar{\alpha}}(x_k,x_{k+1},\ldots ).$$ Here we exploit the fact that the skew shape $\lambda/\mu$ consists of disjoint pieces. Summing over all $k\geq \ell$ completes the proof.}
\end{proof}

It is not clear how to investigate the indecomposability of the module $\mathcal{X}_\alpha$.  The method in Part (2) of Proposition~\ref{prop:indecomp-proof-methods} fails,
since 
\[\pi^{\dI}_j(P)=0 \iff  j, j+1 \text{ are in column 1 of $P$ }\iff j\in [\ell-1],\]
but this is trivially true for \textit{every} $P\in \SIT^*(\alpha).$
Similarly, 
there are many examples which show that the technique of Part (1) of Proposition~\ref{prop:indecomp-proof-methods} also fails.   Recall from Theorem~\ref{thm:BBSSZ2dualImm} that for the dual immaculate action, we have 
\[\pi^{\dI}_j(P)=P \iff j+1 \text{ is weakly below $j$ in $P$.}\]

Let $\alpha=223$  as in Figure~\ref{fig:Poset}, so that  $\svw{S^{row*}_\alpha}=\tableau{3 &6 &7\\2 &5\\1 &4}.$ 

Then $\pi^{\dI}_j(S^{row*}_\alpha)=S^{row*}_\alpha \iff j\in \{3,6\}.$

However, the tableau 
$P=\tableau{3 &5 &6\\2 &7\\1 &4}$  is also fixed by both $\pi^{\dI}_3$ and $ \pi^{\dI}_6$, as well as by $\pi^{\dI}_5.$

Recall from the previous section that $
\NSET(\alpha)$ is the set of all standard fillings of $\alpha$ with increasing rows, but with at least one column that is not increasing bottom to top.  We now consider two more modules, one for the $\dI$-action, and one for the $\rdI$-action, whose basis will be the set $\SET(\alpha)\cap\SIT^*(\alpha)$. 
\begin{theorem}\label{thm:SET-intersect-SIT*} Let $\alpha\vDash n$ be of length $\ell$.
\begin{enumerate}
\item $\SET(\alpha)\cap\SIT^*(\alpha)$ is the closed interval $[S^{col}_\alpha, S^{row*}_\alpha]$ of the poset $P\rdI(\alpha)\simeq P\dI(\alpha)$.
\item The set $D(\alpha)=\NSET(\alpha)\cap \SIT^*(\alpha)$ is a basis for a $\dI$-invariant submodule $\mathcal{Y}_\alpha$ of $\mathcal{X}_\alpha$. 
\item The quotient module $\mathcal{X}_\alpha/\mathcal{Y}_\alpha$, under the $\dI$-action,  has a basis of cosets 
whose representatives constitute the set $\SET(\alpha)\cap\SIT^*(\alpha)$, and is cyclically generated by $S^{row*}_\alpha$. 
\item The set $RD(\alpha)=\SET(\alpha)\setminus \SIT^*(\alpha)$ is a basis for an $\rdI$-invariant submodule ${R\mathcal{Y}}_\alpha$ of the row-strict extended module $\mathcal{Z}_\alpha$ whose basis is $\SET(\alpha)$. 
\item The quotient module  $\mathcal{Z}_\alpha/R{\mathcal{Y}}_\alpha$, under the $\rdI$-action,has a basis of cosets 
whose representatives constitute the set  $\SET(\alpha)\cap\SIT^*(\alpha)$, and is cyclically generated by $S^{col}_\alpha$.  Furthermore, it is indecomposable.
\item The quasisymmetric characteristic of the quotient $\dI$-module $\mathcal{X}_\alpha/\mathcal{Y}_\alpha$ is 
\begin{equation}\label{eqn:Steph} \chr(\mathcal{X}_\alpha/\mathcal{Y}_\alpha)=
\sum _{k\geq \ell} e_{\ell -1} (x_1, \ldots, x_{k-1}) x_k \mathcal{E}_{\overline{\alpha}} (x_k, \ldots),
\end{equation}
where $\overline{\alpha}=(\alpha_1-1, \alpha_1-1,\ldots)$ is the composition obtained from $\alpha$ by diminishing each part by 1, and discarding parts equal to zero, and $\mathcal{E}_\beta$ is the extended Schur function of Theorem~\ref{thm:Assaf-Searles}, Equation~\eqref{eq:ext-Schur}.
\item The quasisymmetric characteristic of the quotient $\rdI$-module $\mathcal{Z}_\alpha/R\mathcal{Y}_\alpha$ is 
\begin{equation}\label{Sheila-Elizabeth2022Feb21}
 \chr(\mathcal{Z}_\alpha/R{\mathcal{Y}}_\alpha)=
\sum _{k\geq \ell} h_{\ell -1} (x_1, \ldots, x_k) x_k R\mathcal{E}_{\overline{\alpha}} (x_{k+1}, \ldots)
\end{equation}
where $\overline{\alpha}=(\alpha_1-1, \alpha_1-1,\ldots)$ is as above, and $R\mathcal{E}_\beta$ is the row-strict extended Schur function of Proposition~\ref{prop:rowstrict-ext-basis}.
\end{enumerate}
\end{theorem}

\begin{proof} For Part (1), let $P\in \SET(\alpha)\cap\SIT^*(\alpha)$.  The straightening argument of Lemma~\ref{lem:Zcyclic} shows that $P$  can be obtained by applying a sequence of operators $\pi^{\rdI}_i$ to $S^{col}_\alpha$, which preserves the first column.  This is because the first columns already match, so we can proceed to the inductive step, which only affects the remaining columns.   Similarly, the proof of Lemma~\ref{lem:SIT*-interval} shows that $P$ can be obtained by applying a sequence of operators $\pi^{\dI}_i$ to $S^{row*}_\alpha$, again without affecting the first column.   It remains to check that the sequence of operators produces tableaux lying entirely within $\SET(\alpha)$. But if  the $\dI$-straightening went from $S^{row*}_\alpha $ to $T$ and then $T$ to $P$, where $T\notin \SET(\alpha)$, then by reversing the arrows this would mean that, with respect to the $\rdI$-action, $T$ can be obtained by applying a sequence of $\rdI$-operators to $P$. Since $P$ is in $\SET(\alpha)$, this contradicts Lemma~\ref{lem:Zcyclic}, which established that $\SET(\alpha)$ is closed under the $\rdI$-action.

For Part (2), if $S\in \NSET(\alpha)$ and $\pi_i^{\dI}(S)=T\in\SIT^*(\alpha)$, then $S=\pi_i^{\rdI}(T)$.  Hence $T\notin \SET(\alpha)$, otherwise by Lemma~\ref{lem:Zcyclic}, $S$ would also be in $\SET(\alpha)$.   Part (4) follows similarly, and Part (3), as well as the first statement of Part (5), are then 
immediate, since $S^{row*}_\alpha$ generates $\mathcal{X}_\alpha$ and $S^{col}_\alpha$ generates $\mathcal{Z}_\alpha$.

To see that the quotient module $\mathcal{Z}_\alpha/R{\mathcal{Y}}_\alpha$ in Part (5) is indecomposable under the $\rdI$-action, we observe that  the proof of Theorem~\ref{thm:Z-indecomp} goes through, although we now start with the equation 
\[ f(S^{col}_\alpha)=\sum_{T\in [S^{col}_\alpha, S^{row*}_\alpha]} a_T T;\]
this is because, as is easily verified,  Lemma~\ref{lem:Z-indecomplemma} also applies here.

For Part (6),  the  filtration argument of Section~\ref{sec:partial-order} gives 
\[\chr(\mathcal{X}_\alpha/\mathcal{Y}_\alpha)=\sum_{T\in \SET(\alpha)\cap\SIT^*(\alpha)} F_{\comp(\Des_{\dI}(T))}.\]
The first column with $1,\ldots,\ell$ gives us $e_{\ell -1}$ as in Proposition~\ref{prop:ElizabethSIT*-char}. Also clearly we do not have a descent at $\ell$.
Since columns and rows all increase in $\SET(\alpha)$, the remainder of the tableau must have $\ell+1$ in the bottom left corner and numbers $\ell+2, ..., n$ such that the columns increase and so do the rows. This gives tableaux for the generating function for the extended Schur function $\mathcal{E} _{\overline{\alpha}}$. 

Similarly for Part (7), the filtration  gives 
\[\chr(\mathcal{Z}_\alpha/R\mathcal{Y}_\alpha)=\sum_{T\in \SET(\alpha)\cap\SIT^*(\alpha)} F_{\comp(\Des_{\rdI}(T))}.\]
This is clearly just the image  of $\chr(\mathcal{X}_\alpha/\mathcal{Y}_\alpha)$ under the involution $\psi$.
Since only $\ell$ at the top of the first column is in $\Des_{\rdI}(T)$ for every $T\in \SET(\alpha)\cap\SIT^*(\alpha)$,  the first column  contributes  $h_{\ell-1}(x_1,\ldots, x_k)x_k$.  Since $\ell \in \Des_{\rdI}(T)$, there must be a strict increase in index $\ell$, so the remaining shape, being column strict in every column, generates $\mathcal{R}\mathcal{E}_{\overline{\alpha}}(x_{k+1},\ldots)$.
\end{proof}

We have observed that Lemma~\ref{lem:sameposet} allowed us to interpret the passage between the dual immaculate and row-strict dual immaculate functions, via the map $\psi$, in terms  of the partial orders induced by the respective Hecke actions on standard immaculate tableaux.  We conclude by pointing out that precisely the same relationship holds between the $\hn$-actions defined in \cite{TvW2015} and \cite{BS2021}.

\begin{remark}\label{rem:Steph-TvW-BardwellSearles}
The $\hn$-action on standard reverse composition tableaux (SRCT), as defined in \cite{TvW2015}, bears the same relationship  to the action defined by Bardwell-Searles \cite{BS2021} for row-strict Young quasisymmetric functions, in the sense of Lemma~\ref{lem:sameposet}.  First we recall the actions that are involved in these two situations.  We also need to recall the passage between \svw{standard} reverse composition tableaux and \svw{standard} Young composition tableaux \svw{of size $n$}, as described in \cite[Chapter 4, Chapter 5]{LMvW2013}, \svw{namely replace each $i$ with $n+1-i$ and reverse the composition.}
The $\hn$-action on standard reverse composition tableaux defined in \cite{TvW2015} is:
\begin{equation}\label{eqn:defn-TvW(T)}\pi_i^{TvW}(T)=\begin{cases} T, & \text{if $i$ is strictly right of $i+1$},\\
0,  & \text{if $i$ is strictly NW of, or in the same column as, $i+1$},\\
s_i(T), & \text{if $i$ is strictly SW of $i+1$}.
\end{cases}\end{equation}
The second action is the one defined by Bardwell-Searles   for row-strict Young tableaux \cite{BS2021},  
but with left and right swapped to reflect using reverse tableaux when considering row-strict quasisymmetric Schur functions:
\begin{equation}\label{eqn:defn-RBS(T)}\pi_i^{revBS}(T)=\begin{cases} T, & \text{if $i$ is weakly left of $i+1$},\\
0,  & \text{if $i$ is right-adjacent to $i+1$},\\
s_i(T), & \text{if $i$ is strictly right of $i+1$ but not in the same row}.
\end{cases}\end{equation}
\end{remark}
Then we have the following  analogue of 
Lemma~\ref{lem:sameposet}:
\begin{lemma}\label{lem:arrows} Let $S$ and $T$ be standard reverse composition tableaux. If $S\neq T$, $S\neq 0$ then
$$\pi_i^{TvW}(T)=S \quad \iff \quad \pi_i^{revBS}(S)=T.$$
\end{lemma}
\begin{proof}
\begin{align*}
\pi_i^{TvW}(T)=S &\iff \text{$i$ is strictly SW of $i+1$ in $T$}\\
&\iff \text{$i$ is strictly NE of $i+1$ in $S$}\\
&\iff \pi_i^{revBS}(S)= s_i(S)=T.
\end{align*}
The argument is now complete.
\end{proof}

A summary of our results for row-strict dual immaculate functions appears in Table~\ref{table:RSImm} at the end of the next section.  Table~\ref{table:QSymfunctions-modules} summarises the families of quasisymmetric functions, with corresponding  $0$-Hecke modules, that have appeared in the literature to date.

\section{A new descent set}\label{sec:new-descent-set}
The work of this section is motivated by a desire to complete the analysis of 0-Hecke actions on the set $\SIT(\alpha)$, by considering all the possible variations on the descent set definitions in Section~\ref{sec:Background}. There are four possible relative positions of $i$ and $i+1$: (strictly/weakly) and  (above/below), two of which give the dual immaculate and row-strict dual immaculate $\hn$-modules. 

As in Section~\ref{sec:Hecke-action-RSImm}, for each composition $\alpha$ of $n$, let $\mathcal{V}_\alpha$ denote the vector space whose basis is the set $\SIT(\alpha)$ of standard immaculate tableaux of shape $\alpha$.  In this section we define two new descent sets which are complementary, and construct 0-Hecke modules for each one.  This in turn gives two new families of quasisymmetric functions which are related by the involution $\psi$. The general framework is tantalisingly  similar to  Section~\ref{sec:Hecke-action-RSImm}, but the details are sufficiently different that once again a  careful analysis is required to ensure that one does indeed have a 0-Hecke action in each case.

\begin{definition}\label{def:new-des}  For $T\in\SIT(\alpha),$ define $\Des_{\mathcal{A}^*}(T):=\{i: i+1 \text{ is strictly below $i$ in $T$}\}.$
\end{definition}

Now we turn to the complement of this descent set.  

\begin{definition}\label{def:new-des-complement}  For $T\in\SIT(\alpha),$ define $\Des_{\bA^*}(T):=\{i: i+1 \text{ is weakly above $i$ in $T$}\}$.
\end{definition}

We will prove the following.

\begin{theorem}\label{thm:Hecke-module-A} There is a cyclic $\hn$-module $\mathcal{A}_\alpha$, generated by the bottom element  $S^0_\alpha$ of the poset $P\rdI(\alpha)$,  whose quasisymmetric characteristic is
\[ \chr({\mathcal{A}_\alpha})=
\sum_{T\in \SIT(\alpha)} F_{\comp(\Des_{\mathcal{A}^*}(T))}.\]
Here $\Des_{\mathcal{A}^*}(T)=\{i: i+1 \text{ is strictly below  } i \text{ in }T\}.$
\end{theorem}

Similarly:

\begin{theorem}\label{thm:Hecke-module-barA}  There is a cyclic $\hn$-module $\mathcal{\bA}_\alpha$, generated by the top element  $S^{row}_\alpha$ of the poset $P\rdI(\alpha)$,  whose quasisymmetric characteristic is
\[ \chr({\bA_\alpha})=
\sum_{T\in \SIT(\alpha)} F_{\comp(\Des_{\bA^*}(T))}.\]
Here $\Des_{\bA^*}(T)=\{i: i+1 \text{ is weakly above  } i \text{ in }T\}.$
\end{theorem}

Note that, as is the case with $\dI_\alpha$ and $\rdI_\alpha$,  the two characteristics are related by the involution $\psi: \psi(\chr({\bA_\alpha}))=
 \chr({\mathcal{A}_\alpha})$.
 
The $\hn$-actions we define in Theorem~\ref{thm:Hecke-action-A} and Theorem~\ref{thm:Hecke-action-barA}, will also establish the following. The first two parts are analogues of results from  Section~\ref{sec:row-strict-ext}; see  Theorem~\ref{thm:rowstrictext-bigone} and Proposition~\ref{prop:shin-module}.  The third is the analogue of Proposition~\ref{prop:ElizabethSIT*-char} in Section~\ref{sec:dualImmsubmodules}.   Once again we have a pleasing expression for the quasisymmetric characteristic of the action of $\bA_\alpha$ on $\SIT^*(\alpha)$.

\begin{proposition}\label{prop:Z-A-X-barA} The vector space with basis 
\begin{enumerate}\item
 $\SET(\alpha)$ is a cyclic $\hn$-submodule $\mathcal{A}_{\SET(\alpha)}$ of $\mathcal{A}_\alpha$, generated by $S^{col}_\alpha$, with characteristic 
 \[\sum_{T\in \SET(\alpha)}F_{\comp(\Des_{\mathcal{A}^*}(T))};\]
\item 
$ \SET(\alpha)$ is also a quotient module $\bA_{\SET(\alpha)}$ of $\bA_\alpha$,   cyclically generated by $S^{row}_\alpha$, with characteristic
 \[\sum_{T\in \SET(\alpha)}F_{\comp(\Des_{\mathcal{\bA}^*}(T))}.\]
\item  $\SIT^*(\alpha)$ is a cyclic $\hn$-submodule $\bA_{\SIT^*(\alpha)}$ of $\bA_\alpha$ generated by $S^{row*}_\alpha$, with characteristic 
\[\sum_{T\in \SIT^*(\alpha)}F_{\comp(\Des_{\bA^*}(T))} = \sum_{k\geq \ell} e_{\ell-1}(x_1,\cdots,x_{k-1})x_k e_\beta(x_k,\ldots)e_{\alpha_\ell - 1}(x_{k+1},\ldots).\]
where $\alpha$ has length $\ell$ and $\beta = (\alpha_1-1,\ldots,\alpha_{\ell-1}-1)$ ignoring any parts of size 0.
\end{enumerate}
\end{proposition}
\begin{proof}[Proof of the quasisymmetric characteristic in Part (3):] 

\svwtwo{Note the similarity with the proof of Proposition~\ref{prop:ElizabethSIT*-char}. Recall from Theorem~\ref{thm:AS19} that for a partition $\lambda$ we have the Schur function expansion
$$s_\lambda = \sum _{T\in \SET (\lambda)} F_{\comp (\DesI (T))}$$and this naturally generalizes (see \cite[Theorem~7.19.7]{RPSEC21999}) for partitions $\lambda, \mu$ to the skew Schur function expansion 
$$s_{\lambda/\mu} = \sum _{T\in \SET (\lambda/\mu)} F_{\comp (\DesI (T))}$$where $\lambda/\mu$ is the diagram of $\lambda$ with that of $\mu$ removed from the bottom left corner. See \cite{gessel} for further details.}

\svwtwo{Now note that the first column of every tableau $T\in \SIT ^* (\alpha)$ is in natural bijection with the unique tableau in $\SET(1^\ell)$.  Also note that $T$ with the leftmost column removed and every remaining tableau entry (that is at least $\ell +1$) reduced by $\ell$ is in natural bijection with a standard extended tableau denoted by $\tilde{D}_T\cup C$  of shape $\tilde{D}_T$, consisting of disconnected columns of length $\alpha _1 -1, \alpha _2 -1, \ldots , \alpha _{\ell -1} -1$ from bottom to top, and a further shape $C$ at the top consisting of a column of length $\alpha _{\ell} -1$. Moreover, 
$$\{ i-\ell \suchthat i\in Des_{\bA^*}  (T), i>\ell \} = \{ i \suchthat i\in \DesI (\tilde{D}_T \cup C) \}.$$It is worth noting that if $i+1$ is strictly above $i$ (so in a different row), or in the same row as $i$ in $T$, then this descent is maintained in $\tilde{D}_T \cup C$ when we switch rows to columns (with entries increasing bottom to top) so these descents all get transferred from $T$ to $\tilde{D}_T \cup C$.
Furthermore, we have that $\ell -1 \in \Des _{\bA^*}  (T)$, so in the monomials appearing in $\chr(\mathcal{X}_\alpha)$ we have that the $(\ell-1)$th variable does not equal the $\ell$th variable, but the $\ell$th variable may or may not equal the $(\ell+1)$th variable. In particular, it may, unless in $T$ the entry $\ell+1$ appears in  row  $\ell$, so $\ell \in Des_{\bA^*}(T)$, in which case the $\ell$th variable may not equal the $(\ell+1)$th variable. Hence if the $\ell$th variable is $x_k$, then comparing the expressions for $\chr(\mathcal{X}_\alpha)$ and $s_{\lambda /\mu}$, the first column of $T$ contributes
$$s_{(1^{\ell -1})}(x_1,\ldots,x_{k-1})x_k=e_{\ell-1}(x_1,\ldots,x_{k-1}) x_k$$and the remainder of the shape contributes
$$s_{\tilde{D}_T}(x_k,\ldots) s_{(1^{\alpha _\ell -1})}(x_{k+1},\ldots)= e_\beta(x_k,\ldots)e_{\alpha_\ell - 1}(x_{k+1},\ldots).$$ Again we exploit the fact that the skew shape $\lambda/\mu$ consists of disjoint pieces. Summing over all $k\geq \ell$ completes the proof.}
\end{proof}

  We also have the following analogue of Theorem~\ref{thm:SET-intersect-SIT*}, giving 
 two more modules, one for the $\mathcal{A}$-action, and one for the $\bA$-action, whose basis is the set $\SET(\alpha)\cap\SIT^*(\alpha)$. 
\begin{theorem}\label{thm:A-barA-SET-intersect-SIT*} Let $\alpha\vDash n$ be of length $\ell$.
\begin{enumerate}
\item $\SET(\alpha)\cap\SIT^*(\alpha)$ is the closed interval $[S^{col}_\alpha, S^{row*}_\alpha]$ of the poset $P\rdI(\alpha)\simeq P\dI(\alpha)$.
\item The set $D(\alpha)=\NSET(\alpha)\cap \SIT^*(\alpha)$ is a basis for a $\bA$-invariant submodule $\bA\mathcal{Y}_\alpha$ of $\bA_{\SIT^*(\alpha)}$. 
\item The quotient module $\bA_{\SIT^*(\alpha)}/\bA\mathcal{Y}_\alpha$, under the $\bA$-action,  has coset basis whose representatives constitute the set $\SET(\alpha)\cap\SIT^*(\alpha)$, and is cyclically generated by $S^{row*}_\alpha$. 
\item The set $RD(\alpha)=\SET(\alpha)\setminus \SIT^*(\alpha)$ is a basis for an $\mathcal{A}$-invariant submodule ${R\mathcal{A}\mathcal{Y}}_\alpha$ of the module $\mathcal{A}_{\SET(\alpha)}$ whose basis is $\SET(\alpha)$. 
\item The quotient module  $\mathcal{A}_{\SET(\alpha)}/R\mathcal{A}{\mathcal{Y}}_\alpha$, under the $\mathcal{A}$-action, has coset basis
 whose representatives constitute the set $\SET(\alpha)\cap\SIT^*(\alpha)$, and is cyclically generated by $S^{col}_\alpha$.
\item The quasisymmetric characteristic of the quotient $\bA$-module $\bA_{\SIT^*(\alpha)}/\bA\mathcal{Y}_\alpha$ is 

$\chr(\bA_{\SIT^*(\alpha)}/\bA\mathcal{Y}_\alpha)$
\begin{equation}\label{eqn:Elizabeth2022Feb21-barA}
=\begin{cases}
\sum _{k\geq \ell} e_{\ell -1} (x_1, \ldots, x_{k-1}) x_k \chr(\bA_{\SET(\overline{\alpha})})(x_k,x_{k+1},\ldots), 
& \alpha \neq (1^m,n-m),\\
e_n, &\text{ otherwise},
\end{cases}
\end{equation}
where $\overline{\alpha}=(\alpha_1-1, \alpha_1-1,\ldots)$ is the composition obtained from $\alpha$ by diminishing each part by 1, and discarding parts equal to zero. 
\item The quasisymmetric characteristic of the quotient $\mathcal{A}$-module $\mathcal{A}_{\SET(\alpha)}/R\mathcal{A}\mathcal{Y}_\alpha$ is

$\chr(\mathcal{A}_{\SET(\alpha)}/R{\mathcal{AY}}_\alpha)$
\begin{equation}\label{eqn:Elizabeth2022Feb21-A}
=
\begin{cases}
\sum_{k\geq \ell} h_{\ell -1} (x_1, \ldots, x_{k})x_k\chr(\mathcal{A}_{\SET(\overline{\alpha})})(x_{k+1},\ldots), & \alpha \neq (1^m,n-m),\\
h_n, \text{ otherwise}.
\end{cases}
\end{equation}
where $\overline{\alpha}=(\alpha_1-1, \alpha_1-1,\ldots)$ is as above. 
\end{enumerate}
\end{theorem}

\begin{proof} We comment only on the last two parts, since the proofs are otherwise similar to Theorem~\ref{thm:SET-intersect-SIT*}.
For Part (6), when $\alpha = (1^m,n-m)$,  $\Des_{\bA}(T) = \{1,2,\ldots, n-1\}$ and hence 
$\chr(\bA_{\SIT^*(\alpha)}/\bA\mathcal{Y}_\alpha)=e_n$.

For Part (7), note that when $\alpha = (1^m,n-m)$, the $\mathcal{A}^*$-descent set is always empty, and hence 
$\chr(\mathcal{A}_{\SET(\alpha)}/R{\mathcal{AY}}_\alpha)=h_n$.
\end{proof}

The example below shows that the quasisymmetric functions $\chr(\mathcal{A}_\alpha)$ do not form a basis for $\Qsym$, since they fail to be independent:
$\chr(\mathcal{A}_\alpha)=F_{(n)}$ for $\alpha\in\{ (1^{\ell-1},n-\ell+1), 1\le \ell\le n\}$.
\begin{example}\label{ex:new-quasisym} 
Let $\alpha=(1^{\ell-1},n-\ell+1), 1\le \ell\le n.$ There  is only one tableau $T$ in $\SIT(\alpha)$.  
The descent sets are

$\Des_{\dI}(T)=[\ell-1] \text{ or }\emptyset \text{ if } \ell=1, \Des_{\rdI}(T)=\{\ell, \ell+1,\ldots,n-1\} \text{ or }\emptyset \text{ if } \ell=n, \Des_{\mathcal{A}}(T)=\emptyset, \Des_{\bA}(T)=[n-1].$

Hence the quasisymmetric characteristics are 
\[\dI_\alpha=F_{(1^{\ell-1}, n-\ell+1)}=F_\alpha,\  \rdI_\alpha=F_{(\ell, 1^{n-\ell})},  \ 
\chr(\mathcal{A}_\alpha)=F_{(n)},\ \chr({\bA}_\alpha)=F_{(1^n)}.\]
\end{example}

\begin{remark}\label{rem:descent-inclusions} Since
\[\Des_{\mathcal{A}^*}(T)\subseteq \Des_{\rdI}(T), 
\quad \Des_{\bA^*}(T)\supseteq \Des_{\dI}(T),\]
 the fundamental quasisymmetric functions in the expansion of $\rdI_\alpha$ (resp. $\dI_\alpha$) correspond to compositions that are refinements of (resp. are coarser than) 
 those appearing in the fundamental expansion of $\chr(\mathcal{A}_\alpha)$ (resp. $\chr(\bA_\alpha)$).
 For example, the three elements in $\SIT(13)$ are 
 $T_1=\tableau{2\\1&3 &4}, T_2=\tableau{3\\1&2 &4}, T_3=\tableau{4\\1&2 &3}$, with respective descent sets 
 $\{2,3\}$, $\{1,3\}$ and $\{1,2\}$ for $\rdI_{31}$, and 
 $\{2\}$, $\{3\}$ and $\emptyset$ for $\chr(\mathcal{A}_{31})$.
 Hence $\rdI_{31}=F_{211}+F_{121}+F_{112}$, 
 $\chr(\mathcal{A}_{31})=F_{22}+F_{31}+F_4$.
\end{remark}

We turn now to constructing the appropriate $\hn$-modules for each new descent set.   The standard model in the literature for doing this is the following.  We are given some basis \svw{that} is a subset $ST(\alpha)$, say, of the set of all standard tableaux  of shape $\alpha$,  and some definition of descents.  For $T\in ST(\alpha)$, let  $s_i(T)$ be the operator  switching $i$ and $i+1$ in   $T$.  This may or may not produce a basis element in $ST(\alpha)$. Define the action of the generator $\pi_i$ on $T\in ST(\alpha)$ by
\begin{equation}\label{eqn:standard-model-Hecke-action}
\pi_i(T)=\begin{cases}
T & \text{ if $i$ is NOT a descent of $T$},\\
s_i(T)  & \text{ if $i$ is a descent of $T$ and } s_i(T)\in ST(\alpha),\\
0 &\text{ otherwise}.
\end{cases}
\end{equation}

Under suitable conditions, the generators will satisfy the 0-Hecke relations, thereby defining a 0-Hecke module. We will show that these propitious circumstances occur for $\mathcal{A}$ and $\bA$, in addition to $\dI$ and $\rdI$, and that 
 in each case, the partial order resulting from the 0-Hecke action gives the immaculate Hecke poset $P\rdI_\alpha \simeq  P\dI_{\alpha}$.

For Definition~\ref{def:new-des}, \eqref{eqn:standard-model-Hecke-action} gives us the following.

Define, for each $1\le i\le n-1$ and each standard immaculate tableau $T$ of shape $\alpha,$ the action of the generator $\pi_i$ on $T$ to be 
\[\pi_i^{\mathcal{A}^*}(T)=\begin{cases} T, & \text{if } 
i\notin \Des_{\mathcal{A}^*}(T) \\
\phantom{T} &\iff i+1 \text{ strictly above or right-adjacent to $i$ in $T$} , \\
s_i(T), &  i\in \Des_{\mathcal{A}^*}(T) \iff i+1 \text{ strictly below $i$ in $T$}. \\
\end{cases}\]
where $s_i(T)$ is the standard immaculate tableau obtained from $T$ by swapping $i$ and $i+1.$ 

The following analogue of Lemma~\ref{lem:prelim} \svw{then follows by definition.}
\begin{lemma}\label{lem:prelimA} Let $T$ be a standard immaculate tableau and let $i\in \Des_{\mathcal{A}^*}(T).$  Then 
\begin{enumerate}
\item
$i, i+1$ cannot both be in the leftmost column of $T;$
\item if $s_i(T)$ is a standard immaculate tableau, then $i\notin \mathrm{Des}_{\mathcal{A}^*}(s_i(T)).$
\end{enumerate}
\end{lemma}
 \begin{theorem}\label{thm:Hecke-action-A} The operators $\pi_i^{\mathcal{A}^*}$ define an action of the 0-Hecke algebra on the vector space $\mathcal{V}_\alpha.$ 
\end{theorem}
\begin{proof} As usual for simplicity we will simply write $\pi_i$ for $\pi_i^{\mathcal{A}^*}$. Clearly from the preceding analysis, $\pi_i(T)\in \mathcal{V}_\alpha$  for every standard immaculate tableau $T$ of shape $\alpha.$ We must verify that the operators satisfy the 0-Hecke algebra relations.
\\
To show $\pi_i^2(T)=\pi_i(T),$ we need only check the case when $i+1$ is strictly below $i$ in $T$. In this case $\pi_i(T)=s_i(T),$ and $i$ is now strictly below $i+1$. Hence $\pi_i(s_i(T))=s_i(T)$ and we are done.
\\
Let $1\le i,j\le n-1$ with $|i-j|\ge 2.$ Then $\{i, i+1\}\cup \{j, j+1\}=\emptyset,$ so the actions of $\pi_i$ and $\pi_j$ are independent of each other, and hence commute.
\\
It remains to show that 
\begin{equation}\label{eqn:keyclaim*}\pi_i \pi_{i+1} \pi_i(T)=\pi_{i+1} \pi_i \pi_{i+1}(T).
\end{equation}  We examine  four separate cases.
\begin{enumerate}
\item[Case 1:] Assume $i\notin \Des_{\mathcal{A}^*}(T), i+1\notin \Des_{A^*}(T)$: 
Then $\pi_i(T)=T, \pi_{i+1}(T)=T, $ and the claim is clear.
\item[Case 2:] Assume $i\in \Des_{\mathcal{A}^*}(T), \text{ but }i+1\notin \Des_{A^*}(T)$:  Then $\pi_{i+1}(T)=T,$ 
 $\pi_i(T)=s_i(T).$  Hence \eqref{eqn:keyclaim*} becomes 
\begin{equation}\label{eqn:step1*}\pi_i \pi_{i+1} ( s_i(T))= \pi_{i+1}(s_i(T)),
\end{equation}
which we need to verify.  
\\
Assume $i+1\notin \Des_{\mathcal{A}^*}(s_i(T)).$ 
The left-hand side then equals $\pi_i(s_i(T))=s_i(T)$  by \svw{Lemma~\ref{lem:prelimA},} and this is also the right-hand side. 
\\
Finally assume $i+1\in \Des_{\mathcal{A}^*}(s_i(T)).$   We now have $i+1$ strictly below $i$ in $T,$ so that $i+1$ is strictly above $i$ in $s_i(T),$ and $i+2$ strictly below $i+1$ in $s_i(T).$   Also recall that $i+2$ was weakly above $i+1$ in $T.$ It follows that 
\begin{equation}\label{eqn:step2*} \text{In } s_i(T),  i+2 \text{ is now weakly above } i 
\text{ and strictly below }i+1. \end{equation}  
This implies $\pi_{i+1}(s_i(T))=s_{i+1}(s_i(T)),$ and in the latter we now have $i$ (weakly) below \svw{$i+1$,} which is strictly below $i+2.$ In particular $i$ is not a descent of $\pi_{i+1}(s_i(T))=s_{i+1}(s_i(T)),$ and hence the latter tableau is fixed by $\pi_i.$ Equation~\eqref{eqn:step1*} is thus verified.
\\
\item[Case 3:]  Assume $i\notin \Des_{\mathcal{A}^*}(T), \text{ but }i+1\in \Des_{\mathcal{A}^*}(T)$:   
\\
Now \eqref{eqn:keyclaim*} becomes 
\begin{equation}\label{eqn:step11*}\pi_i(s_{i+1}(T))= \pi_{i+1}\pi_i(s_{i+1}(T)),
\end{equation}
which we need to verify. 
\\
Thus  $i+2$ is strictly below $i+1$ in $T.$ Also $i$ is  weakly BELOW $i+1$ in $T$. 
\\
Suppose $i$ and $i+1$ are NOT right-adjacent in $T$.  Then  $i+1$ is strictly above both $i$ and $i+2$ in $T,$ and thus  $i+2$ is strictly above both $i,i+1$ in $s_{i+1}(T).$  Here we have  two possibilities for $s_{i+1}(T)$:
\begin{itemize}
\item
 Either $i$ is below $i+1$, which is below $i+2$, and hence 
$\pi_i$ and $\pi_{i+1}$ both fix $s_{i+1}(T);$
\item
or $i+1$ is below $i$, and $i$ is below $i+2$.  In the latter case, applying $\pi_i$ to $s_{i+1}(T)$ switches $i$ and $i+1$, so that in \svw{$\pi_i(s_{i+1}(T))$} we now have $i$ below $i+1$, and $i+1$ (still) below $i+2$.  But then $\pi_{i+1}$ fixes $\pi_i(s_{i+1}(T))$.
\end{itemize}
Now suppose $i$ and $i+1$ ARE right-adjacent in $T$.  Since $i+2$ is strictly below $i+1$ in $T$, the only possibility here for $s_{i+1}(T)$ is that $i, i+2$ are 
right-adjacent, lying strictly above $i+1.$ But then in $\pi_i(s_{i+1}(T))$, 
$i+1, i+2$ are 
right-adjacent, lying strictly above $i.$ 
\\
Hence $\pi_{i+1}$ fixes 
$\pi_i(s_{i+1}(T))$, and 
  Equation~\eqref{eqn:step11*} has been established.
\item[Case 4:] Assume $i, i+1\in \Des_{\mathcal{A}^*}(T)$:   This means $i$ is strictly above $i+1$ which is strictly above $i+2$ in $T$.
 Then Equation~\eqref{eqn:keyclaim*} becomes 
\begin{equation}\label{eqn:step4*} s_i s_{i+1} s_i(T) = s_{i+1} s_i s_{i+1}(T), \end{equation}
and it is easy to see that this is indeed true.  
\end{enumerate}
We have verified Equation~\eqref{eqn:keyclaim*} in all cases, thereby completing the proof that the action of the generators $\pi_i$ extends to an action of $H_n(0)$ on $\mathcal{V}_\alpha.$ 
\end{proof}    

This $\hn$-action on $\mathcal{V}_\alpha$ induces a partial order $\poA$ on $\SIT(\alpha)$, exactly as in Section~\ref{sec:partial-order}.
 In fact it produces the same poset $P\rdI_\alpha$, since clearly 
\[\pi_i^{{\mathcal A}^*}(T)\notin \{T, 0\} \iff \pi_i^{\rdI}(T)\notin \{T, 0\} .\]  

Moreover, we have the following:
\begin{proposition}\label{prop:cyclic-A} The module $\mathcal{A}_\alpha$ is cyclically generated by the standard immaculate tableau $S^0_\alpha$.
\end{proposition}
\begin{proof}  Examining  Proposition~\ref{prop:bot-elt}, we see that, since we have 
\[\pi_i^{{\mathcal A}^*}(T)=s_i(T)\iff \pi_i^{\rdI}(T)=s_i(T)\iff i+1 \text{ is strictly below $i$ in $T$},\]
the straightening algorithm goes through without change, giving the same conclusion.
\end{proof}
\begin{proof}[Proof of Theorem~\ref{thm:Hecke-module-A}]

As in Section~\ref{sec:partial-order}, 
 extend the partial order $\poA$ on $\SIT (\alpha)$ to an arbitrary total order on $\SIT (\alpha)$, denoted by $\poA^t$.   Let the elements of $\SIT (\alpha)$ under $\poA ^t$ be $$\{\rtau _1 \poA^t  \cdots \poA ^t \rtau _m  \}.$$ 
The induced filtration, as in Section~\ref{sec:indecomp-module-and-poset}, is
\[ 0\subset \mathrm{span}([T_m,S^{row}_\alpha]) 
\subset\dots\subset \mathrm{span}([T_i,S^{row}_\alpha])\subset \dots \subset 
\mathrm{span}([S^0_\alpha, S^{row}_\alpha]).\]

  The key observation here is that
 \[\svw{\pi_i^{\mathcal{A}^*}}(T)=T\iff i\notin \svw{\Des_{\mathcal{A}^*}(T),}\]
which guarantees that the successive quotients in the filtration are one-dimensional irreducible modules.
Theorem~\ref{thm:Hecke-module-A} now follows.\end{proof}

Consider next Definition~\ref{def:new-des-complement}.
Define, for each $1\le i\le n-1$ and each standard immaculate tableau $T$ of shape $\alpha,$ the action of the generator $\pi_i$ on $T$ to be 
\[\bpi_i(T)=\pi_i^{\bA^*}(T)=\begin{cases} T, & 
i\notin \Des_{\bA^*}(T)\iff i+1 \text{ strictly below } i \text{ in } T, \\
0 & i+1 \text{  right-adjacent to $i$ in $T$} \\
&\text{ or $i, i+1$ both in 1st column}, \\
s_i(T), &   i+1 \text{ strictly above $i$ in $T$}, \\
&\text{ and } i, i+1 \text{ not both in 1st column}. 
\end{cases}\]
where $s_i(T)$ is the standard immaculate tableau obtained from $T$ by swapping $i$ and $i+1.$ 

\begin{lemma}\label{lem:bA} Suppose $i, i+1$ are not both in column 1 of $T\in \SIT(\alpha)$, and $i\in \svw{\Des_{\bA ^*}}(T).$  
\svw{If $\bpi_i(T)= s_i(T)$}, then $i$ cannot be in column 1 of $T$.
\end{lemma}
\begin{proof}
\svw{This follows} since rows must increase left to right, and $i+1$ is strictly above $i$ in $T$.
\end{proof}
\begin{theorem}\label{thm:Hecke-action-barA} The operators $\bpi_i$ define an action of the 0-Hecke algebra on the vector space $\mathcal{V}_\alpha.$ 
\end{theorem}
\begin{proof}  Clearly from the preceding analysis, $\bpi_i(T)\in \mathcal{V}_\alpha$  for every standard immaculate tableau $T$ of shape $\alpha.$ We must verify that the operators satisfy the 0-Hecke algebra relations.
\\
That $\bpi_i^2(T)=\bpi_i(T)$ holds is clear if $\bpi_i(T)\in \{T,0\}, $ so we need only check the case when $i+1$ is strictly above $i$ in $T$, and $i, i+1$  not in column 1. In this case $\bpi_i(T)=s_i(T),$ and $i$ is now strictly below $i+1$ in $s_i(T)$. Hence $\bpi_i(s_i(T))=s_i(T)$ and we are done.
\\
Let $1\le i,j\le n-1$ with $|i-j|\ge 2.$ Then $\{i, i+1\}\cup \{j, j+1\}=\emptyset,$ so the actions of $\bpi_i$ and $\bpi_j$ are independent of each other, and hence commute.
\\
It remains to show that 
\begin{equation}\label{eqn:keyclaim*barA}\bpi_i \bpi_{i+1} \bpi_i(T)=\bpi_{i+1} \bpi_i \bpi_{i+1}(T).
\end{equation}
 Again  there are  four separate cases.
\begin{enumerate}
\item[Case 1:] $ i, i+1\notin \Des_{\bA^*}(T)$: This case is clear as before, since $\bpi_i, \bpi_{i+1}$ both fix $T$.
\item[Case 2:] $ i\in \Des_{\bA^*}(T), i+1\notin \Des_{\bA^*}(T)$: thus $\bpi_{i+1}$ fixes $T$ and $i+2$ is strictly below $i+1$.

If $\bpi_i(T)=0$, we are done, so assume $\bpi_i(T)=s_i(T).$   By Lemma~\ref{lem:bA}, $i$ cannot be in column 1, but $i+1$ could be.  Note that Equation~\eqref{eqn:keyclaim*barA} becomes 
\begin{equation}\label{eqn:barAstep2}\bpi_i \bpi_{i+1} s_i(T)=\bpi_{i+1} s_i(T).\end{equation}

Now $i$ is strictly below $i+1$ in $T$, so either $i+2$ is below both of them, or $i+2$ is above $i$ and below $i+1$. 

In $s_i(T)=\bpi_i(T)$, we have, in the first case,  $i+2$ below $i+1$ below $i$ and so $\bpi_{i+1}$ (as well as  $\bpi_i$)  fixes $s_i(T);$ 
and in the second case  $i+1$ below $i+2$ below $i$, so $\bpi_{i+1}(s_i(T))$ has $i+2$ below $i+1$ below $i$. 

Thus in either case, $i$ is NOT a descent of $\bpi_{i+1}(s_i(T))$, which is thus fixed by $\bpi_i$.  This verifies \eqref{eqn:barAstep2}.
\item[Case 3:]  Assume $i\notin \Des_{\bA^*}(T), \text{ but }i+1\in \Des_{\bA^*}(T)$:  

\svw{This is Case 2 with the roles of $i$ and $i+1$ interchanged, and the argument follows mutatis mutandis.}

\item[Case 4:] Assume $i, i+1\in \Des_{\bA^*}(T)$:   This means $i+2$ is weakly above $i+1$ which is weakly above $i$ in $T$.  If both relations are strict, and neither of the pairs $i+2, i+1$ nor $i+1,i$ is in column 1, 
  Equation~\eqref{eqn:keyclaim*barA} becomes 
\begin{equation}\label{eqn:step4bar*} s_i s_{i+1} s_i(T) = s_{i+1} s_i s_{i+1}(T), \end{equation}
and it is easy to see that this is indeed true.  

If $\bpi_i(T)=0=\bpi_{i+1}(T)$, each side of Equation~\eqref{eqn:keyclaim*barA} is 0.

Otherwise, we have two sub-cases.

\textbf{Case 4a:} $\bpi_{i+1}(T)\ne 0$ and  $\bpi_i(T)=0$:

This makes the left side of~\eqref{eqn:keyclaim*barA} 
equal to 0.   We must show that the right side of~\eqref{eqn:keyclaim*barA} is also zero.

We have $\bpi_{i+1}(T)=s_{i+1}(T)$, and $i+2, i+1$ are not both in column 1.  The right side
 of~\eqref{eqn:keyclaim*barA} is then $\bpi_{i+1}\bpi_i(s_{i+1}(T))$. 

Since  $\bpi_i(T)=0$, either $i, i+1$ are right-adjacent in $T$ with $i+2$ strictly above them 
or $i, i+1$ are in the first column with $i+2$ above them, not in the first column.

In the first case, applying $s_{i+1}$ makes $i$ and $i+2$ right-adjacent with $i+1$ above them, and this followed by $\bpi_i$ makes $i+1$ and $i+2$ right-adjacent.  This tableau is thus sent to 0 by  $\bpi_{i+1}$.

In the second case, applying $s_{i+1}$ makes $i$ and $i+2$ adjacent in column 1 with $i+1$ above them and strictly to the right, and following this with $\bpi_i$ puts both $i+1$ and $i+2$ in column 1.  Again, this tableau is  sent to 0 by  $\bpi_{i+1}$.

\textbf{Case 4b:} $\bpi_{i+1}(T)=0$ but $\bpi_i(T)\ne 0$:

This is Case 4a with the roles of $i$ and $i+1$ interchanged, and the argument follows mutatis mutandis.

\end{enumerate}
We have verified Equation~\eqref{eqn:keyclaim*barA} in all cases, thereby completing the proof that the action of the generators $\bpi_i$ extends to an action of $H_n(0)$ on $\mathcal{V}_\alpha.$ 
\end{proof}    

The analogue of Proposition~\ref{prop:cyclic-A} is 
\begin{proposition}\label{prop:cyclic-barA} The module $\bA_\alpha$ is cyclically generated by the standard immaculate tableau $S^{row}_\alpha$.
\end{proposition}
\begin{proof}  Examining  Proposition~\ref{prop:top-elt}, we see that, since we have 
\[\pi_i^{{\bA}^*}(T)=s_i(T)\iff \pi_i^{\dI}(T)=s_i(T)\]
\[\iff i+1 \text{ is strictly above $i$ in $T$, and $i, i+1$ are not both in column 1},\]
the straightening algorithm goes through without change, with the same conclusion.
\end{proof}

\begin{proof}[Proof of Theorem~\ref{thm:Hecke-module-barA}] This is proved using the filtration
of $\hn$-modules 
\[0\subset \mathrm{span}([S^0_\alpha, T_1])
\subset \dots \subset 
\mathrm{span}([S^0_\alpha, T_i])\subset \dots \subset \mathrm{span}([S^0_\alpha, S^{row}_\alpha]),\]
induced by a linear extension 
$$\{\rtau _1 \poAbar^t  \cdots \poAbar ^t \rtau _m  \}$$ 
 of the poset, exactly as before. Again the key observation  is that
 \[\pi_i^{\bA^*}(T)=T\iff i\notin \Des_{\bA^*}(T),\]
 guaranteeing that the successive quotients in the filtration are one-dimensional irreducible modules.
\end{proof}
Finally, we have:
\begin{proof}[Proof of Proposition~\ref{prop:Z-A-X-barA}]
This is proved by examining the  arguments in Lemma~\ref{lem:Zcyclic} and Lemma~\ref{lem:SIT*-interval}, which in turn rely on the straightening algorithms of Proposition~\ref{prop:bot-elt} and Proposition~\ref{prop:top-elt}, respectively, 
and observing that they go through unchanged.\end{proof}

\begin{remark}\label{rem:indecomp?}
It is not clear how to approach the question of indecomposability for the module $\mathcal{A}_\alpha$.  The method of Part (2) of Proposition~\ref{prop:indecomp-proof-methods} fails since the action of the generators is never 0 by definition.  As for Part (1), consider the  basis element $v=S^{row}_\alpha$, and note that $\Des_{\mathcal{A}^*}(S^{row}_\alpha)$ is the empty set; thus  $\pi_i^{\mathcal{A}^*}$ fixes 
the top element $S^{row}_\alpha$ for all $i$. Since the generator of the module is the bottom element $S^{0}_\alpha$, Part (1) of Proposition~\ref{prop:indecomp-proof-methods} fails for the basis element $S^{row}_\alpha$.  This discussion also applies to the module $\mathcal{A}_{\SET(\alpha)}$, since $S^{row}_\alpha\in \SET(\alpha)$. 

For the module $\bA_\alpha$,  our only recourse is Part (2) of Proposition~\ref{prop:indecomp-proof-methods}, because the cyclic generator is now $S^{row}_\alpha$ , and its $\bA$-descent set is $[n-1]$.  An example shows that this fails as well. Let $\alpha=3121$.  It is easy to check that the Hecke generators sending $S^{row}_\alpha$ to 0 coincide with those sending $P=\pi_6^{\bA}(S^{row}_\alpha)$ to 0.  
\end{remark}
 While Example~\ref{ex:new-quasisym} showed that the quasisymmetric characteristics arising from these modules are not independent, we conclude the paper by showing that they are nonetheless  combinatorially  interesting.  
In analogy with $\dI_\alpha$ and $\rdI_\alpha$, we will establish that these new characteristics are the generating functions for appropriately defined sets of  tableaux with possibly repeated entries.  See Figure~\ref{fig:4-descent-sets} for a schematic summary.

\sheilaFeb{
 Let $\alpha\vDash n$.   Given a filling $T$ of the diagram of $\alpha$ with entries $\{1,2,\ldots\}$, let $\mathrm{cm}(T)$, the \emph{content monomial} of $T$,  denote the monomial $x_1^{d_1} x_2^{d_2}\cdots$, 
where $d_i$ is the number of entries equal to $i$ in $T$. We then refer to the composition $(d_1,d_2,\ldots)$ of $n$ as the \emph{content} of the tableau $T$.
\\
Define $\mathcal{T}_\alpha({\text{1st col}<, \text{rows}\le})$ to be the set of tableaux of shape $\alpha$ whose first column entries increase strictly bottom to top, and whose rows all increase weakly left to right.
\\
Similarly, define $\mathcal{T}_\alpha({\text{1st col}\le, \text{rows}<})$ to be the set of tableaux of shape $\alpha$ whose first column entries increase weakly bottom to top, and whose rows all increase strictly left to right.
\\
From Definition~\ref{def:dIfunction} and  Definition~\ref{def:rdIfunction}, we have 
\[\dI_\alpha=\sum_{T\in  \mathcal{T}_\alpha({\text{1st col}<, \text{rows}\le})}  \mathrm{cm}(T), \text{ and } \rdI_\alpha=\sum_{T\in  \mathcal{T}_\alpha({\text{1st col}\le, \text{rows}<})}  \mathrm{cm}(T).\]
\\
Define $\mathcal{T}_\alpha({\text{cols}<, \text{rows}\le})$ to be the set of column-strict  tableaux of shape $\alpha$, whose  columns ALL increase strictly bottom to top, and whose rows all increase weakly left to right.  From Proposition~\ref{prop:gf-row-ext}, the tableau generating function for the extended Schur function is
\[\mathcal{E}_\alpha=\sum_{T\in  \mathcal{T}_\alpha({\text{cols}<, \text{rows}\le})}  \mathrm{cm}(T).\]
\\
Define $\mathcal{T}_\alpha({\text{cols}\le, \text{rows}<})$ to be the set of row-strict  tableaux of shape $\alpha$ whose  columns ALL increase weakly bottom to top, and whose rows all increase strictly left to right.
From Proposition~\ref{prop:gf-row-ext}, the tableau generating function for the row-strict extended Schur function is
\[\mathcal{R}\mathcal{E}_\alpha=\sum_{T\in  \mathcal{T}_\alpha({\text{cols}\le, \text{rows}<})}  \mathrm{cm}(T).\]
Our final result, captured in Figure~\ref{fig:4-descent-sets}, is:
\begin{theorem}\label{thm:8flavours-tableaux}  Let $\alpha\vDash n$.   
\begin{enumerate}
\item Let $\mathcal{T}_\alpha({\text{1st col}\le, \text{rows}\le})$ denote the set of tableaux of shape $\alpha$ whose first column entries increase weakly bottom to top, and whose rows all increase weakly left to right. Then 
\begin{equation*}\label{eqn:gf-A}
\sheilaFeb{\chr({\mathcal{A}_\alpha}) =\sum_{T\in \SIT(\alpha)} F_{\comp(\Des_{\mathcal{A}^*}(T))}=\sum_{T\in  \mathcal{T}_\alpha({\text{1st col}\le, \text{rows}\le})}  \mathrm{cm}(T)}.
\end{equation*}
\item Let $\mathcal{T}_\alpha({\text{1st col}<, \text{rows}<})$ denote the set of tableaux of shape $\alpha$ whose first column entries increase strictly bottom to top, and whose rows all increase strictly left to right. Then
\begin{equation*}\label{eqn:gf-barA}
\sheilaFeb{\chr(\bA_\alpha) =\sum_{T\in \SIT(\alpha)} F_{\comp(\Des_{\bA^*}(T))}=\sum_{T\in  \mathcal{T}_\alpha({\text{1st col}<, \text{rows}<})}  \mathrm{cm}(T)}.
\end{equation*}
\item Let $\mathcal{T}_\alpha({\text{cols}\le, \text{rows}\le})$ be the set of  tableaux of shape $\alpha$ having weakly increasing entries in ALL columns, bottom to top,  and all rows, left to right. Its generating function is the characteristic of the  submodule $\mathcal{A}_{\SET(\alpha)}$ of $\mathcal{A}_\alpha$. Equivalently,
\begin{equation*}\label{eqn:gf-A-SET}
\sheilaFeb{\chr(\mathcal{A}_{\SET(\alpha)})= \sum_{T\in \SET(\alpha)} F_{\comp(\Des_{\mathcal{A}^*}(T))}
=\sum_{T\in  \mathcal{T}_\alpha({\text{cols}\le, \text{rows}\le})}  \mathrm{cm}(T)}.
\end{equation*}
\item Let $\mathcal{T}_\alpha({\text{cols}<, \text{rows}<})$ be the set of  tableaux of shape $\alpha$ having strictly increasing entries in ALL columns, bottom to top,  and all rows, left to right.  Its generating function is the characteristic of the quotient module $\bA_{\SET(\alpha)}$ of $\bA_\alpha$.  Equivalently,
\begin{equation*}\label{eqn:gf-barA-SET}
\sheilaFeb{\chr(\bA_{\SET(\alpha)})=\sum_{T\in \SET(\alpha)} F_{\comp(\Des_{\bA^*}(T))} =\sum_{T\in  \mathcal{T}_\alpha({\text{cols}<, \text{rows}<})}  \mathrm{cm}(T)}.
\end{equation*}
\end{enumerate}
\end{theorem}
}

\begin{proof} 
\sheilaFeb{Recall that a composition $\beta$ is finer than $\alpha$ if and only if $\set(\beta)\supseteq\set(\alpha).$ 
 We prove  \eqref{eqn:gf-A} by showing that, for a fixed $\beta\vDash n$, the cardinality of the set 
\begin{center}$S_1=\{T\in \SIT(\alpha):\Des_{{\mathcal A}^*}(T)\subseteq \set(\beta)\}$\end{center}
equals the cardinality of the set 
\begin{center}$S_2=\{T\in \mathcal{T}_\alpha({\text{1st col}\le, \text{rows}\le})   : T \text{ has shape }\alpha \text{ and content } \beta\}$.\end{center} 
Given $U\in S_2$, replace the 1's in $U$ left to right, bottom to top with the consecutive entries $1,\ldots,\beta_1$, then the 2's with the next $\beta_2$ consecutive entries, 
again left to right, bottom to top, and so on.  It is clear that this produces increasing entries going up the first column of the resulting standard tableau $T$, as well as along the rows.  Also the descents can occur only where  the last $i$ of $U$ in this reading order (left to right, bottom to top) changes to an $i+1$, so 
$\set(\beta)\supseteq \Des_{{\mathcal A}^*}(T)$.
\\
\\
Conversely let $T\in S_1.$ The entries of the standard tableau $T$ constitute a labelling of the diagram of $\alpha$.  
We can then fill this diagram  with $\beta_1$ 1's,  $\beta_2$'s, etc. consecutively following the order of the labels in $T$.   
If the cell labelled $m$ is filled with entry $i$, then the cell labelled $m+1$ can be filled with $i$ or $i+1$, but it \emph{must} be filled with the label $i+1$ if $m$ is a descent of $T$. Such a filling is necessarily weakly row-increasing and also increasing in the first column.
\\
\\
In order to prove  \eqref{eqn:gf-barA} we must show that, for a fixed $\beta\vDash n$, the sets 
$S_3=\{T\in \SIT(\alpha):\Des_{{\bA}^*}(T)\subseteq \set(\beta)\}$
and 
$S_4=\{T\in \mathcal{T}_\alpha({\text{1st col}<, \text{rows}<})   : T \text{ has shape }\alpha \text{ and content } \beta\}$ are of the same cardinality.
The argument is similar,  except for the change in reading order.  Given $U\in S_4$, note first that there is at most one $i$ in each row, in view of the strict increase.  In particular, because of the strictly increasing first column, all entries equal to $i$ go from left to right  and top to bottom in $U$.  Now we replace the 1's in $U$ left to right, \textbf{top to bottom}, with the consecutive entries $1,\ldots,\beta_1$, then the 2's with the next $\beta_2$ consecutive entries, 
again  left to right,  \textbf{top to bottom}, and so on.  To go backwards from a standard tableau in $S_3$ to one with repeated entries, note that when $i+1$ is weakly above $i$, then $i$ must be in $\set(\beta)$, so if $i$ was replaced by an entry $j$, then $i+1$ must be replaced by $j+1$, guaranteeing row-strictness.  
\\
\\
The reading order for~\eqref{eqn:gf-A-SET} is the same as for~\eqref{eqn:gf-A}: left to right, bottom to top, and likewise, the reading order for~\eqref{eqn:gf-barA-SET}  is the same as for~\eqref{eqn:gf-barA}:  left to right, \textbf{top to bottom}.  We omit the details for the otherwise identical arguments.  See also the proof of Proposition~\ref{prop:gf-row-ext}.}
\end{proof}
\begin{example} Let $\alpha=122$.  Then the  immaculate Hecke poset $P\rdI(\alpha)$ is a 3-element chain. Here $\SET(\alpha)$ consists of only two tableaux, 
$S^{col}_\alpha=\tableau{3&5\\2&4\\1}$ and $S^{row}_\alpha=\tableau{4&5\\2&3\\1}$, 
while $\SIT(\alpha)$ has one additional tableau, $S^0_\alpha=\tableau{3 &4\\2&5\\1}$. 

We have $\Des_{\mathcal{A}^*}(S^{row}_\alpha)=\emptyset, \Des_{\mathcal{A}^*}(S^{col}_\alpha)=\{3\}, \Des_{\mathcal{A}^*}(S^0_\alpha)=\{4\},$ and hence 
$\chr(\mathcal{A}_{\SET(\alpha)})=F_5+F_{32},$  while 
$\chr(\mathcal{A}_{\alpha})=F_5+F_{41}+F_{32}.$

Consider $\chr(\mathcal{A}_{\SET(\alpha)})$.
Now $F_5$ will contribute all monomials with exponents $\beta\vDash 5$, while $F_{32}$ will contribute only monomials with exponents $\beta$ finer than the composition $32$. 
One such monomial that would appear in both expansions is $x_1x_2^2x_3^2$, corresponding to the composition $122$.
There are exactly two tableaux of content  $122$ with weakly increasing rows and weakly increasing columns.

Following the reading order in the above proof, left to right, bottom to top, we see that 
$S^{col}_\alpha=\tableau{3&5\\2&4\\1}$ maps to $T_1=\tableau{2&3\\2&3\\1}$, 
while $S^{row}_\alpha=\tableau{4&5\\2&3\\1}$ maps to $T_2=\tableau{3&3\\2&2\\1}$. 
Note that these are weakly increasing in all rows and all columns.  Hence $T_1$ comes from  $F_{32}$, and $T_2$ comes from $F_5$.

Now consider the module $\bA_\alpha$.  Here we have  $\Des_{\bA^*}(S^{row}_\alpha)=\{1,2,3,4\},  \Des_{\bA^*}(S^{col}_\alpha)=\{1,2,4\},$  and $ \Des_{\bA^*}(S^0_\alpha)=\{1,2,3\}, $ and hence 
$\chr(\bA_{\SET(\alpha)})=F_{1^5}+F_{1121}$, while  $\chr(\bA_\alpha)=F_{1^5}+F_{1121}+F_{1112}.$ 

$F_{1^5}$ can only contribute the unique square-free monomial of degree 5, corresponding to the tableau $S^{row}_\alpha$. 
Now $F_{1121}$ contributes to $\chr(\bA_\alpha)$  the monomial $x_1x_2 x_3^2 x_4$; following the reading order for $\bA^*$, namely  left to right and \textbf{top to bottom}, we see that 
$S^{col}_\alpha=\tableau{3&5\\2&4\\1}$ maps to $T_3=\tableau{3&4\\2&3\\1}$.  Note that $T_3$ is a non-standard tableau that is strictly increasing in all rows and all columns; also, it is the unique such tableau of content  $1121$.

Similarly $F_{1112}$ contributes the monomial  $x_1x_2 x_3 x_4^2$; again  following the reading order for $\bA^*$, namely  left to right and \textbf{top to bottom}, we see that 
$S^{0}_\alpha=\tableau{3&4\\2&5\\1}$ maps to $T_4=\tableau{3&4\\2&4\\1}$.  Note that $T_4$ is strictly increasing in all rows, but only strictly increasing in the first column.  Again, it is (necessarily) the unique such tableau of content  $1112$.
\end{example}

\begin{table}[htbp]
\caption{$\dI_\alpha$ versus the new basis $\rdI_\alpha$}
\begin{center}
\scalebox{0.8}{
\begin{tabular}{|c|c|c|l|}
\hline
\tclb{Immaculate Tableaux $\rightarrow$}  &  $\dI_\alpha\ $ \cite{BBSSZ2015} &  \tclb{$\rdI_\alpha$} \\
 &  Dual immaculate &  \tclb{ Row-strict dual imm.}\\
[2pt]\hline\hline
 \bf{1st Col} \bf{bottom to top} &strict $\nearrow$ &\tclb{weak $\nearrow$ }\\[2pt]\hline
 \bf{Rows left to right}
& weak $\nearrow$&\tclb{strict  $\nearrow$} \\[2pt]\hline
Descents (for fund. expansion) &$\{i: i+1 \text{ strictly above } i\}$
  & \tclb{$\{i: i+1 \text{ weakly below } i\}$}\\[2pt]\hline\hline
Action of $\psi$  in $\QSym$ 
& $\dI_\alpha(x_1,\ldots,x_n)$ & \tclb{$\rdI_\alpha=\psi(\dI_\alpha)$}\\[2pt]\hline
 \textcolor{blue}{$0$-Hecke action on $\SIT(\alpha)$}&\textit{Note that  standard}  &\textit{tableaux are the same} \\[2pt]\hline\hline
 $\pi_i(T)=T$ & $ i+1 \text{ weakly below } i$ 
& \tclb{$ i+1 \text{ strictly above } i$}
 \\[2pt]\hline
$\pi_i(T)=0$ & $i, i+1$ in 1st column & \tclb{$i, i+1$ in same row}\\[2pt]\hline
 $T, s_i(T)$ standard, &$ i+1 $\text{ strictly above }$ i$,& \tclb{$ i+1 \text{ strictly below } i$}\\[2pt]
and $\pi_i(T)=s_i(T)$ & $i,i+1$ NOT both in 1st column & \phantom{\tclb{ $i, i+1$ NOT in same row}}\\[2pt]\hline\hline
 \textcolor{blue}{Partial order on $\SIT(\alpha)$} &Poset $P\dI(\alpha)=[S^{0}_\alpha, S^{row}_\alpha]$&\tclb{ Poset $P\rdI(\alpha)\simeq P\dI(\alpha)$}\\[2pt]\hline\hline
Cover relation & $S\!\poIcover\!  T\!\iff\!  S=\pi_i^{\dI}(T)$ & \tclb{$S\!\poRIcover\!  T\!\iff\!  T=\pi_i^{\rdI} (S)$}\\[2pt]\hline
Imm. module generated by& top element: $\mathcal{W}_\alpha=\langle S^{row}_\alpha\rangle$  & \tclb{bottom  element: $\mathcal{V}_\alpha=\langle S^0_\alpha\rangle$} \\[2pt]\hline
 Indecomposable? & Yes \cite{BBSSZ2015} & \textcolor{blue}{Yes}\\[2pt]\hline\hline
Extended Schur fn basis & $\mathcal{E}_\alpha,$ $\mathcal{E}_\lambda=s_\lambda$\cite{AS2019} & \tclb{$\psi(\mathcal{E}_\alpha)=\mathcal{R}\mathcal{E}_\alpha$, $\mathcal{R}\mathcal{E}_\lambda=s_{\lambda^t}$}  \\[2pt]\hline
Module, Extended Schur fn & cyclic $\langle S^{row}_\alpha\rangle$, indecomp. \cite{S2020} & \tclb{cyclic $\langle S^{col}_\alpha\rangle$, indecomp.}\\[2pt]
Basis \tclb{$\SET(\alpha)$} & (quotient of larger module) & \tclb{submodule of $\mathcal{V}_\alpha$}\\[2pt]\hline
Quotient module, Ext Schur fn& None & \tclb{Yes, cyclic $\langle S^{row}_\alpha\rangle$, indecomp.} \\[2pt]
Basis \tclb{$\NSET(\alpha)\cap\SIT(\alpha)$}&  & \tclb{quotient of $\mathcal{V}_\alpha$}\\[2pt]\hline
\end{tabular}
}
\end{center}
\label{table:RSImm}
\end{table}

\begin{table}[htbp]
\caption{Various families in $\QSym$ and associated $\hn$-modules;\\\phantom{Table 2.\quad } they form a basis of $\QSym$ \textbf{unless otherwise indicated}}
\begin{center}
\scalebox{0.7}{
\begin{tabular}{|c|c|}
\hline
Quasisymmetric function indexed by $\alpha\vDash n$& $\hn$-module \\[2pt]\hline\hline
Fundamental $F_\alpha$  &  Irreducible, one-dimensional\\[2pt]\cite{gessel} & \cite{DKLT1996}  \\[2pt]   \hline   
   Dual immaculate  &   Cyclic ($ S^{row}_\alpha$)=$\mathcal{W}_\alpha$, indecomp. \T\\[2pt]
$\dI_\alpha=\sum_{T\in\SIT(\alpha)}F_{\mathrm{comp}(\Des_{\dI}(T))}$ & acts on standard immaculate tableaux $\SIT(\alpha)$\\[2pt]
\cite{BBSSZ2014} & \cite{BBSSZ2015}\\[2pt]\hline   
Row-strict dual immaculate & Cyclic ($S^0_\alpha$)=$\mathcal{V}_\alpha$, indecomp. \T \\[2pt]
$\rdI_\alpha=\textcolor{blue}{\psi(\dI_\alpha)}=\sum_{T\in\SIT(\alpha)}F_{\mathrm{comp}(\Des_{\rdI}(T))}$
& acts on $\SIT(\alpha)$\\[2pt]
  &  \\[2pt]\hline
Quasisymmetric Schur &   Indecomp. iff $\alpha$ is simple; \\[2pt]
$\check{S}_\alpha=\sum_{T\in\SRCT(\alpha)}F_{\mathrm{comp}(\Des_{\check S}({T}))}$ & acts on standard reverse \\[2pt]
$\Des_{\check S}({T}):=\{i: i+1 \text{ weakly right of } i\}$ & composition tableaux  $\SRCT(\alpha)$\\[2pt]
 \cite{HLMvW2011} & \cite{TvW2015}\\[2pt]\hline
Row-strict quasisym Schur & \\[2pt]
$\textcolor{blue}{\psi(\check S_\alpha)}=R\check S_\alpha=\sum_{T\in\SRCT(\alpha)} F_{\mathrm{comp}(\Des_{R\check S}({T}))}$ &\\[2pt]
$\Des_{R\check S}({T}):=\{i: i+1 \text{ strictly left of } i\}$ & \cite[Chapter 4]{LMvW2013}, \cite{BS2021}\\[2pt] 
  \cite{MR2014} & \\[2pt]\hline
(Column-strict) Young quasisym function  & \\[2pt]
$ \svw{\hat{S}_\alpha} = \rho(\check{S}_{\alpha^r})$ via $\rho(F_\alpha)=F_{\alpha^r}$ &
\cite[Chapter 4]{LMvW2013}, \cite{TvW2015}\\[2pt]
\cite{LMvW2013}  & \\
[2pt]\hline
Row-strict Young quasisym function  & Indecomp. iff $\alpha$ is simple;\\[2pt]
$R\hat S_\alpha = \textcolor{blue}{\psi(\hat {S}_{\alpha})}$  & acts on standard Young row-strict tableaux \\[2pt] 
 \cite{MN2015} & \cite{BS2021}\\[2pt]\hline 
 Extended Schur function & Cyclic $\langle S^{row}_\alpha \rangle$, \textit{quotient} of $\mathcal{W}_\alpha$,  indecomp.;  \T \\[2pt]
$\mathcal{E}_\alpha= \sum_{T\in \SET(\alpha)} F_{\comp(\Des_{\dI}(T))}$ & acts on $\SET(\alpha)=[S^{col}_\alpha, S^{row}_\alpha]$ \\[2pt]
 \cite{AS2019}&  \cite{S2020}\\[2pt]\hline\hline 
Row-strict Extended  Schur  & Cyclic $\langle S^{col}_\alpha \rangle$, \textit{submodule} of $\mathcal{V}_\alpha$, 
indecomp.; \T\\[2pt]
$\mathcal{R}{\mathcal{E}}_\alpha=\textcolor{blue}{\psi(\mathcal{E}_\alpha)}=\ \sum_{T\in \SET(\alpha)} F_{\mathrm{comp}(\Des_{\rdI}(T))}$ 
 & acts on $\SET(\alpha)=[S^{col}_\alpha, S^{row}_\alpha]$\\[2pt]
 \hline  
Row-strict Extended Quotient Schur  & Cyclic  $\langle S^0_\alpha \rangle$, \textit{quotient} of $\mathcal{V}_\alpha$, indecomp.; \T \\[2pt]
$ \overline{\mathcal{R}\mathcal{E}}_\alpha= \sum_{T\in \NSET(\alpha)\cap\SIT(\alpha)} F_{\mathrm{comp}(\Des_{\rdI}(T))}$ &  acts on $\NSET(\alpha)\cap\SIT(\alpha)$\\[2pt]
\textcolor{red}{*NOT a basis*}    &  \\ 
\hline
$\chr$ of Dual immaculate submodule& Cyclic $\langle  S^{row*}_\alpha \rangle$, \textit{submodule} of $\mathcal{W}_\alpha$; 
\T\\[2pt]
$\chr(\mathcal{X}_\alpha)= \sum_{T\in \SIT^*(\alpha)} F_{\comp(\Des_{\dI}(T))}$ 
& $\mathcal{X}_\alpha$ acts on $[S^0_\alpha, S^{row*}_\alpha]=\SIT^*(\alpha)$\\[2pt]
\textcolor{red}{*NOT a basis*}    &  \\ [2pt]\hline
$\chr$ of Dual immaculate quotient module& Cyclic $\langle S^{row}_\alpha\rangle$, \textit{quotient} of $\mathcal{W}_\alpha$;
\T\\[2pt]
$\chr(\mathcal{W}_\alpha/\mathcal{X}_\alpha)
= \sum_{T\in \SIT(\alpha)\setminus \SIT^*(\alpha)} F_{\comp(\Des_{\dI}(T))}$ 
& $\mathcal{W}_\alpha/\mathcal{X}_\alpha$ acts on $\SIT(\alpha)\setminus\SIT^*(\alpha)$\\[2pt]
\textcolor{red}{*NOT a basis*}     &  \\ [2pt]\hline
The five functions in Section 9: & Cyclic modules  $\langle S^0_\alpha \rangle, \langle S^{row}_\alpha \rangle$ acting on $\SIT(\alpha)$;\\[2pt]
 Theorem~\ref{thm:Hecke-module-A}, Theorem~\ref{thm:Hecke-module-barA}, Proposition~\ref{prop:Z-A-X-barA},& submodule $\langle S^{col}_\alpha\rangle$ of $\mathcal{A}_\alpha$ acting on $\SET(\alpha)$; \\[2pt]
Theorem~\ref{thm:8flavours-tableaux} and Figure~\ref{fig:4-descent-sets}& quotient module $\langle S^{row}_\alpha \rangle$ of $\bA_\alpha$ acting on $\SET(\alpha)$;\\[2pt]
 \textcolor{red}{*NOT a basis*}  & submodule  $\langle S^{row*}_\alpha\rangle$ of $ \langle  S^{row}_\alpha \rangle$ acting on $\SIT^*(\alpha)$\\[2pt]\hline
\end{tabular}
}    
\end{center}
%
\label{table:QSymfunctions-modules}
\end{table}

\bibliographystyle{plain}
\bibliography{refs}
\end{document}